\documentclass{amsart}

\usepackage{etex}

\usepackage{amsthm,amsmath,amssymb,booktabs,braket,enumerate}
\usepackage{float,geometry,soul,stmaryrd,subfigure,tikz,verbatim}
\usepackage{color,graphicx,ltablex,mathrsfs,numprint,pstricks,transparent,xcolor}
\usepackage{tikz-cd}
\usetikzlibrary{arrows,calc,decorations.text,shapes.multipart,positioning}
\usepackage{quiver}

\usepackage[english]{babel}
\usepackage[all]{xy}

\newtheorem{theorem}{Theorem}[section]
\newtheorem{corollary}[theorem]{Corollary}
\newtheorem{proposition}[theorem]{Proposition}
\newtheorem{lemma}[theorem]{Lemma}
\newtheorem{conjecture}[theorem]{Conjecture}
\newtheorem*{theorem*}{Theorem}

\theoremstyle{definition}
\newtheorem{definition}[theorem]{Definition}
\newtheorem{example}[theorem]{Example}
\newtheorem{remark}[theorem]{Remark}

\newtheorem{notation}[theorem]{Notation}

\newtheorem{claim}[theorem]{Claim}

\newtheorem{condition}[theorem]{Condition}

\newcommand{\id}{\mathrm{id}}         
\newcommand{\se}{\subseteq}           
\newcommand{\ppr}{^{\prime}}          
\newcommand{\pprr}{^{\prime\prime}}   
\newcommand{\fa}{\forall}             
\newcommand{\iv}{^{-1}}               
\newcommand{\uas}{^{\ast}}            
\newcommand{\sas}{_{\ast}}            
\newcommand{\Ext}{\mathrm{Ext}}       

\newcommand{\lla}{\longleftarrow}     
\newcommand{\EQ}{\Leftrightarrow}     
\newcommand{\dra}{\dashrightarrow}  
\newcommand{\lra}{\longrightarrow}    

\newcommand{\sfr}{\mathfrak{s}}  

\newcommand{\Ebb}{\mathbb{E}}        
\newcommand{\Zbb}{\mathbb{Z}}        

\newcommand{\Asc}{\mathscr{A}}   
\newcommand{\Bsc}{\mathscr{B}}   

\newcommand{\Dcal}{\mathcal{D}} 
\newcommand{\Ical}{\mathcal{I}} 
\newcommand{\Pcal}{\mathcal{P}} 
\newcommand{\Tcal}{\mathcal{T}} 
\newcommand{\Ucal}{\mathcal{U}} 
\newcommand{\Vcal}{\mathcal{V}} 
\newcommand{\Wcal}{\mathcal{W}} 

\newcommand{\Dbf}{\mathbf{D}}   
\newcommand{\Vbf}{\mathbf{V}}   

\newcommand{\und}{\underline}   
\newcommand{\ovl}{\overline}    
\newcommand{\ov}{\overset}      
\newcommand{\un}{\underset}     
\newcommand{\ti}{\times}        

\newcommand{\cat}{\mathscr{C}}

\newcommand{\add}{\operatorname{add}}

\newcommand{\ind}{\operatorname{ind}}

\newcommand{\al}{\alpha}         
\newcommand{\gam}{\gamma}        
\newcommand{\ups}{\upsilon}      
\newcommand{\vp}{\varphi}        

\newcommand{\vt}{\vartheta}      
\newcommand{\kap}{\kappa}        
\newcommand{\lam}{\lambda}       

\newcommand{\om}{\omega}         
\newcommand{\thh}{\theta}        
\newcommand{\del}{\delta}        
\newcommand{\ep}{\varepsilon}    

\newcommand{\sig}{\sigma}        

\newcommand{\Sig}{\Sigma}        
\newcommand{\Om}{\Omega}         
\newcommand{\Thh}{\Theta}
\newcommand{\Fbf}{\mathbf{F}}     
\newcommand{\Gam}{\Gamma}     

\newcommand{\bsm}{\begin{smallmatrix}}
\newcommand{\esm}{\end{smallmatrix}}

\numberwithin{equation}{section}

\newcommand{\Hom}{\mathrm{Hom}}      
\newcommand{\Cok}{\operatorname{Cok}}      
\newcommand{\cok}{\operatorname{cok}}      

\newcommand{\ush}{^\sharp}           
\newcommand{\ssh}{_\sharp}           

\newcommand{\Ab}{\mathit{Ab}}        
  
\newcommand{\proj}{\operatorname{proj}}  

\newcommand{\op}{^\mathrm{op}}       

\newcommand{\CEs}{(\cat,\Ebb,\sfr)}           

\newcommand{\Nbb}{\mathbb{N}}        
\newcommand{\mr}{\bullet}

\newcommand{\imono}[1]{ \;\xymatrix{  \ar@{>->}^{#1}[r] &  \\} }
\newcommand{\iepi}[1]{ \;\xymatrix{  \ar@{->>}^{#1}[r] &  \\} }
\newcommand{\mono}{ \;\xymatrix{  \ar@{>->}[r] &  \\} }
\newcommand{\epi}{ \xymatrix{   \ar@{->>}[r] &  \\} }

\def\add{\mbox{add}\,}

\newcommand{\A}{\mathcal{A}}

\def\Ab{\mbox{\bf Ab}}

\newcommand{\ca}{{\mathcal A}}
\newcommand{\cA}{{\mathcal A}}

\newcommand{\cd}{{\mathcal D}}

\newcommand{\coh}{\operatorname{\mathrm{coh}}}
\newcommand{\defff}{\operatorname{\mathrm{def}}\nolimits}
\newcommand{\cper}{\operatorname{\mathrm{per}}\nolimits}

\newcommand{\Mod}{\operatorname{\mathrm{Mod}}\nolimits}
\renewcommand{\mod}{\operatorname{\mathrm{mod}}\nolimits}
\newcommand{\fp}{\operatorname{\mathrm{}{fp}}\nolimits}
\newcommand{\fg}{\operatorname{\mathrm{fg}}\nolimits}

\def\Hom{\mbox{Hom}}
\def\Ext{\mbox{Ext}}

\def\Mod{\mbox{Mod\,}}

\def\Im{\mbox{Im}}
\def\Ab{\mbox{Ab}}

\def\Ker{\mbox{Ker}}

\newcommand{\ot}{\otimes}  
\newcommand{\otr}{\otimes_R}  
\newcommand{\dia}{\diamond}  
\newcommand{\Fbb}{\mathbb{F}}  
\newcommand{\EF}{\Ebb F}  
\newcommand{\Sett}{\mathrm{Set}}  
\newcommand{\Fm}{F^{\bullet}}  

\newcommand{\EbbI}{\Ebb_{\mathrm{I}}}  
\newcommand{\EbbII}{\Ebb_{\mathrm{I\!I}}}  
\newcommand{\Irm}{\mathrm{I}}
\newcommand{\IIrm}{\mathrm{I}\!\mathrm{I}}

\author{Mikhail Gorsky}
\address{Institut f\"ur Algebra und Zahlentheorie, Universit\"at Stuttgart, Pfaffenwaldring 57, 70569 Stuttgart, Germany}
\email{mikhail.gorsky@iaz.uni-stuttgart.de}
\author{Hiroyuki Nakaoka}
\address{Graduate School of Mathematics, Nagoya University, Furocho, Chikusaku, Nagoya 464-8602, Japan}
\email{nakaoka.hiroyuki@math.nagoya-u.ac.jp}
\author{Yann Palu}
\address
{LAMFA, Universit\'e Picardie Jules Verne, 80039 Amiens, France}
\email{yann.palu@u-picardie.fr}
\begin{document}

\title{Positive and negative extensions in extriangulated categories}

\begin{abstract}
We initiate the study of derived functors in the setting of extriangulated categories. By using coends, we adapt Yoneda's theory of higher extensions to this framework. We show that, when there are enough projectives or enough injectives, thus defined extensions agree with the ones defined earlier via projective or injective resolutions. For categories with enough projective or enough injective morphisms, we prove that these are right derived functors of the $\Hom$-bifunctor in either argument.

Since $\Hom$ is only half-exact in each argument, it is natural to expect \lq\lq negative extensions", i.e. its left derived functors, to exist and not necessarily vanish. We define negative extensions with respect to the first and to the second argument and show that they give rise to universal $\delta$-functors, when there are enough projective or injective morphisms, respectively. In general, they are not balanced. However, for topological extriangulated categories, the existence of a balanced version of negative extensions follows from combining the work of Klemenc on exact $\infty$-categories with results of the second and third authors. We discuss various criteria under which one has the balance of the above bifunctors on the functorial or on the numerical levels. This happens, in particular, in the cases of exact or triangulated categories, and also in the case of the category of $n$-term complexes with projective components over a finite-dimensional algebra.

Given a connected sequence of functors on an extriangulated category $\CEs$, we determine the maximal relative extriangulated structure, with respect to which the sequence is a $\delta$-functor. We also find several equivalent criteria for the existence of enough projective or injective morphisms in a given extriangulated category.

\end{abstract}

\maketitle

\tableofcontents

\section{Introduction}\label{Section_Intro}

The notion of \emph{extriangulated categories} was recently introduced by the second and the third authors \cite{NP1} as a unification of exact categories (in the sense of Quillen \cite{Buh, Ke, Q}) and of triangulated categories \cite{V}. 
One can also consider it as an axiomatization of full extension-closed subcategories of triangulated categories, although it is not known whether every extriangulated category admits an embedding (as a full extension-closed subcategory) into a triangulated category. The datum defining an extriangulated category consists of an additive category $\cA$, a bifunctor $\Ebb\colon \cA\op \times \cA \to \Ab$ and a so-called \emph{realization} correspondence $\mathfrak{s}$ that essentially defines the class of \emph{conflations}. This datum must then satisfy certain axioms that simultaneously generalize axioms of exact and of triangulated categories. 

The class of extriangulated categories is stable under several natural operations. An extension-closed subcategory of an extriangulated category becomes extriangulated with a naturally induced structure \cite{NP1}. Any closed sub-bifunctor of $\Ebb$ induces a \emph{relative} extriangulated structure, or extriangulated \emph{substructure} \cite{HLN, ZH}. An ideal quotient of an extriangulated category by an ideal generated by morphisms that factor through projective-injective objects \cite{NP1} or, more generally, by an ideal generated by morphisms that factor through morphisms with injective domain and projective codomain \cite{FGPPP}, has an induced extriangulated structure. 

\normalcolor

Barwick \cite{Barwick} defined (additive) \emph{exact $\infty$-categories} as a higher categorical framework where higher algebraic $K$-theory is well-defined and well-behaved. Recently, the second and the third authors \cite{NP2} proved that the homotopy category of an exact $\infty$-category admits a natural extriangulated structure. They called extriangulated categories admitting such $\infty$-categorical enhancements \emph{topological}. Klemenc \cite{K} proved that each exact $\infty$-category has a \emph{stable hull}, that is, it embeds into a stable $\infty$-category, and this embedding is universal. Since the homotopy category of a stable $\infty$-category is triangulated, this indicates that each topological extriangulated category admits an embedding as an extension closed subcategory into a triangulated category. Therefore, topological extriangulated categories inherit a notion of higher positive and negative extensions from this triangulated embedding. We study existence of such extensions in extriangulated categories, without reference to any specific embedding.

In Subsection~\ref{Subsection_BasicDef} we recall the definitions and basic properties of extriangulated categories. In particular in Subsection~\ref{Subsection_Defect}, we summarize known results concerning the category of defects.

By using coends, in Section~\ref{Section_Positive} we define a version of Yoneda's theory of higher (positive) extensions (\cite{Buch, Buh, BH, M, Y}) for extriangulated categories. We prove the existence of long exact sequences in both arguments and prove that this definition agrees with the notion of higher extensions defined by Herschend, Liu, and the second author \cite[Section 5.1]{HLN} via projective or injective resolutions. 
As the main theorem of Section~\ref{Section_Positive}, we show the existence of long exact sequences in the positive direction. 

We use this to introduce the notion of \emph{positive global dimension of extriangulated categories}. For exact categories, this recovers the usual definition. Each non-zero triangulated category has infinite global dimension. A non-trivial example of an extriangulated category of global dimension $n$ is given by the category of $n$-term complexes in the bounded homotopy category $K^b(\proj \Lambda)$ of finitely generated projective modules over a finite-dimensional algebra $\Lambda$. More general extriangulated categories of global dimension $n$ appeared in the recent work by Liu and Zhou \cite{LZ}.

The case of positive global dimension $1$ is quite special. This condition is equivalent to asking the bifunctor $\Ebb$ to be right exact in each argument. We will initiate the detailed study of such categories, which we call \emph{hereditary}, in the upcoming sequel \cite{GNP}. We 
 expect this notion to be of great importance in representation theory of finite-dimensional algebras and beyond, in particular, in $\tau$-tilting theory. The reason is that $K^{[-1,0]}(\proj \Lambda)$ (where $\Lambda$ is a finite-dimensional algebra) carries a natural hereditary extriangulated structure, such that $\mod \Lambda$ is its costable category. 
It is important to note that the (abelian) extriangulated structure on $\mod \Lambda$ is not induced by this hereditary extriangulated structure on $K^{[-1,0]}(\proj \Lambda)$ whenever $\Lambda$ is not hereditary.
More generally, \emph{extended cohearts of co-t-structures} introduced by Pauksztello and Zvonareva \cite{PZ} are hereditary. Further examples of hereditary extriangulated categories appeared in \cite{PPPP}: these are given by $2$-Calabi--Yau triangulated categories with a cluster tilting object $T$, endowed with the relative extriangulated structure making $T$ projective. More examples will be discussed in \cite{GNP}.

Generalizing standard notions from the setting of exact categories to that of extriangulated ones, in Section \ref{Subsection_PA} we define \emph{connected sequences of functors} and \emph{$\delta$-functors} in this framework. We also consider \emph{partial $\delta$-functors}, for which the corresponding long sequences are bounded at least on one side and exact everywhere except at the end terms. For example, each half-exact functor is a partial $\delta$-functor. More precisely, we define all these notions for sequences of bifunctors which are contravariant in one argument and covariant in the second and considered as functors in one of the arguments. This will be sufficient for our studies of positive and negative extensions.

These notions clearly depend on the choice of an extriangulated structure on an additive category. Given a connected sequence of functors for an extriangulated category $\CEs$, we find in Section \ref{SectionDeltaRel} the maximal relative structure with respect to which the sequence becomes a $\delta$-functor. Each partial $\delta$-functor can be extended in one or two directions. For each new term, one can pose a natural universality condition. As usual, the $n$-th term of the universal $\delta$-functor having a half-exact functor $F$ as the degree zero term is called its $(-n)$-th left, resp. $n$-th right derived functor, depending of the sign of $n$. If $F$ is right or left exact, its derived functors on one side vanish.

In the abelian or even the exact setting, there are several ways to find derived functors of a given functor $F$. One of the construction was first considered by Cartan and Eilenberg \cite{CE} under the name of \emph{satellite} functors of $F$. This notion actually predates the notion of derived functors, but satisfies the same universality properties. Satellite functors are defined by an iterative procedure. The classical definition uses projective covers or injective envelopes of a functor \cite{BH, CE, M}, see also \cite{MR} and references therein. They can be also defined in more general situations \cite{J}. A slightly more conceptual way to understand derived functors is to define them via projective resolutions of one argument or injective resolutions of the other. An even more conceptual approach to derived functors is to consider them as cohomologies of objects in derived categories with respect to natural $t$-structures.

In each extriangulated category, the $\Hom$- and the $\Ebb$-functors are at least half-exact in each argument (with respect to the extriangulated structure). In this paper, we look for derived functors of the bifunctor $\Hom$. For categories having enough projective, resp. enough injective morphisms, we prove that $\Ebb^n$ is the $n$-th right derived functor of $\Hom$ considered as a functor in the first or in the second variable, for $n > 0$. The definition of $\Ebb^n$ in \cite{HLN} follows the second strategy above and uses projective or injective resolutions. Note that $\Ebb$ is the first derived functor of $\Hom$ precisely with respect to the extriangulated structure defined by $\Ebb$, this bifunctor is not the first derived functor of $\Hom$ with respect to an arbitrary extriangulated structure on the underlying additive category $\cat$.

Remark that if $\CEs$ does not come from an exact category, the $\Hom$-functor is only half-exact in each argument. That means that the inflations $A\ov{f}{\lra}B$ are not necessarily monomorphic, so $\cat(-,f)\colon\cat(-,A)\to\cat(-,B)$ may have a non-zero kernel, and dually for deflations. On the other hand, if $\CEs$ corresponds to a triangulated category, there naturally exist exact sequences in negative direction given by $\cat(-,-[n])$ $(n<0)$. In Section~\ref{Section_Negative}, we investigate the question of the existence of \lq natural' negative extensions in general extriangulated categories. 

There might be many possibilities to complete the partial covariant and contravariant $\delta$-functors $(\Ebb^n, n > 0; \del\ssh)$ and $(\Ebb^n, n > 0; \del\ush)$ to actual $\delta$-functors. We discuss some natural options in Section~\ref{Subsection_PA}. One of them, first considered in \cite{INP} has projectively, resp. injectively stable morphisms as the degree zero terms. As explained in \cite{MR}, in the case of module categories, these can be seen as the $0$-th derived functors of the functor $\Hom$; more precisely, as the $1$-st left derived functors of the functor $\Ebb$. We expect
that the same happens in the general extriangulated case and these functors satisfy the relevant universal property. Let us emphasize that while $\Ebb$ is the $1$-st right derived functor of $\Hom$ (with respect to the structure $\CEs$), the latter is not the $1$-st left derived functor of $\Ebb$.

In the present paper, we are more interested in the problem of defining left derived functors of the bifunctor $\Hom$, which we call the \emph{negative extensions}. In Section \ref{Section_Negative}, we define a covariant and a contravariant connected sequences of functors in negative degrees by taking the duals of the functors $\Ebb$ with respect to the first, resp. to the second argument. For categories having enough projective, resp. injective morphisms, we prove that they are actually $\delta$-functors and prove their universality. Thus, in a certain reasonable generality, these give the covariant and the contravariant version of negative extensions.

The nontrivial problem in our setting is the question of \emph{balance}. The functors $\Hom$- and $\Ebb$-are in general only half-exact. There is therefore no reason to expect their derived functors in the first argument to agree with their derived functors in the second argument. Also, derived categories of extriangulated categories have not yet been defined (and the na\"ive definition via categories of complexes fails). 

Our definition of positive extensions ensures that they give derived functors in either variable, so right derived functors of $\Hom$ and of $\Ebb$ are \emph{balanced}, i.e. one has a bivariant partial $\delta$-functor in non-negative degrees, which is universal both as a covariant and as a contravariant $\delta$-functor among those with $\Hom$ as the degree zero term.

For the negative extensions defined above, however, we do not have such a balance in general. We study in detail the question of the existence of the balance. We formulate and prove precise conditions under which our construction gives balanced derived functors. This happens, in particular, in all exact and all triangulated categories, but also in the case of $n$-term complexes of finitely generated projectives over a finite-dimensional algebra (Example~\ref{ExNTerm}). One can find some natural non-examples via non-positively graded differential graded algebras.

The work of Klemenc suggests that at least for topological extriangulated categories $\CEs$, there must exist an embedding as an extension-closed subcategory into a certain triangulated category $\Tcal$, namely into the homotopy category of the stable hull of the exact $\infty-$categorical enhancement of $\CEs$. While it is not a priori unique or universal on the level of additive categories, it still makes sense to expect such an embedding to behave particularly nicely, compared to arbitrary embeddings of $\CEs$ as full extension-closed subcategories into triangulated categories.
The restriction of shifted $\Hom$-functors $(\Tcal(-, -[n])|_{\cat\op\times\cat}, n \in \Zbb)$, together with natural connecting morphisms, defines a bivariant $\delta$-functor on $\CEs$, which has $\cat(-, -)$ as the degree zero term. We expect that this $\delta$-functor is universal \emph{among the bivariant $\delta$-functors having $\cat(-, -)$ as the degree zero term}. It is thus natural to expect that its terms in negative degrees give \emph{balanced negative extensions} in the category $\CEs$. Moreover, in Section \ref{Subsection_Conj}, we conjecture that balanced negative extensions exist even beyond the topological case and give conjectural bounds on the degrees of the non-vanishing negative extensions in categories of finite positive global dimension. In general, we cannot expect the balanced negative extensions to agree with their non-balanced covariant and contravariant counterparts. We describe a conjectural relation between them. Explicit calculations of balanced negative extensions in special cases will appear in the upcoming sequel \cite{GNP}. The partial $\delta$-functors  $(\Ebb^n, n > 0; \del\ssh)$ and $(\Ebb^n, n > 0; \del\ush)$ and their completions to bivariant $\delta$-functors will play an important role in the definition of Hall algebras of extriangulated categories in the forthcoming work of the first author \cite{G}. Namely, they will be used to define so-called \emph{defect-counting} bilinear forms on the split Grothendieck group of the category $\cat$ used as a part of the definition of these algebras, and also to define various bilinear forms on the  Grothendieck group of the category $\CEs$ yielding twists of the multiplication.

Balanced partial $\delta-$functors play a crucial role in the recent work by Adachi, Enomoto, and Tsukamoto \cite{AET}. For an extriangulated category with a certain extra datum provided by a bifunctor and a pair of connecting morphisms, they define \emph{$s-$torsion pairs}, which form a poset in a natural way, and study intervals in this poset. They construct certain bijections generalizing those induced by HRS-tilts of $t-$structures and certain bijections related to $\tau-$tilting reduction. They call this extra datum of a bifunctor and a pair of morphisms \emph{a negative first extension}. Let us comment on a minor difference in the notation. 
$(F^{-1}, \del\ssh, \del\ush)$ is a \lq negative first extension' in their terminology if and only if $(F^{-1}, \Hom)$ (or, equivalently, $(F^{-1}, \Hom, \Ebb)$) together with connecting morphisms form a \lq bivariant partial $\delta-$functor' in our terminology. In other words, they require the partial $\delta-$functor  $(F^{-1}, \Hom, \del\ssh, \del\ush)$ to be balanced, but do not require it to be universal. In particular, an extriangulated category may admit many different \lq negative first extensions', in their terminology. This might explain why their datum is called \emph{a} negative first extension, while our derived functors are called \emph{the} negative extensions. Their work can be seen as a non-trivial indication of the importance of balanced, but not necessarily universal, \lq negative first extensions', or, in general, of not necessarily universal bivariant (partial) $\delta-$functors.

As mentioned above, we prove the universality for positive extensions and both the acyclicity and the universality for negative extensions for categories \emph{having enough projective (resp. injective) morphisms}. These notions appeared earlier in ideal approximation theory (\cite{BM, FGHT, ZH}). They are slightly weaker requirements than the existence of enough projective (resp. injective) objects. We further explore these conditions and prove that, quite surprisingly, they are equivalent to many interesting properties of our category. Indeed, by Corollaries \ref{CorEquivProjDeflAdded} and \ref{CorDeflAdj}, the following are equivalent:
\begin{itemize}
\item $\CEs$ has enough projective morphisms.
\item $\Ebb(C,-)$ is finitely generated for any $C\in\cat$.
\item $\Ebb^n(C,-)$ are finitely presented for any $C\in\cat$ and any $n\in\Nbb_{\ge0}$.
\item The inclusion $\iota\colon\defff\Ebb\hookrightarrow\fp\cat$, of the category of defects into the category of finitely presented functors, has a left adjoint.
\end{itemize}

In fact, the authors first noticed the importance of some of these conditions from the perspective of the present work, and on the very late stages of work on the paper realized that they are equivalent to a previously known notion. 

\subsection{Acknowledgments}

Parts of this work were done during the first and the third authors' participation at the Junior Trimester Program "New Trends in Representation Theory" at the Hausdorff Institute for Mathematics in Bonn. They are very grateful to the Institute for the perfect working conditions. We are grateful to Thomas Br\"ustle, Bernhard Keller, and Gustavo Jasso for useful discussions and providing important references. The first author would like to thank Steffen Koenig for his support. The second author is supported by JSPS KAKENHI Grant Number JP20K03532. He is grateful to Norihiro Hanihara for his comments on the $2$nd extension groups. The third author is supported by the French ANR grant CHARMS (ANR-19-CE40-0017).
\normalcolor

\subsection{Notation}
Throughout this article, we fix a commutative ring $R$ with $1$, and $\cat$ denotes an essentially small additive $R$-linear category.

\section{Preliminaries}

\subsection{Extriangulated categories}\label{Subsection_BasicDef}
Let us briefly recall the definition of extriangulated categories from \cite{NP1}.

\begin{definition}\label{DefExtension}
Suppose $\cat$ is equipped with an $R$-bilinear functor $\Ebb\colon\cat^\mathrm{op}\times\cat\to\Mod R$. For any pair of objects $A,C\in\cat$, an element $\del\in\Ebb(C,A)$ is called an {\it $\Ebb$-extension}. In particular, for any $A,C\in\cat$, the zero element $0\in\Ebb(C,A)$ is called the {\it split $\Ebb$-extension}.

For any $\Ebb$-extension $\del\in\Ebb(C,A)$, any $a\in\cat(A,A\ppr)$ and any $c\in\cat(C\ppr,C)$, we abbreviate
\[ \Ebb(C,a)(\del)\in\Ebb(C,A\ppr)\ \ \text{and}\ \ \Ebb(c,A)(\del)\in\Ebb(C\ppr,A) \]
simply to $a\sas\del$ and $c\uas\del$.
Remark that we have
$\Ebb(c,a)(\del)=c\uas a\sas\del=a\sas c\uas\del$
in $\Ebb(C\ppr,A\ppr)$.
\end{definition}

\begin{definition}\label{DefMorphExt}
Let $\del\in\Ebb(C,A),\del\ppr\in\Ebb(C\ppr,A\ppr)$ be any pair of $\Ebb$-extensions. A {\it morphism} $(a,c)\colon\del\to\del\ppr$ of $\Ebb$-extensions is a pair of morphisms $a\in\cat(A,A\ppr)$ and $c\in\cat(C,C\ppr)$ in $\cat$, satisfying $a\sas\del=c\uas\del\ppr$.
\end{definition}

\begin{definition}\label{DefSumExtension}
Let $\del_1=(A_1,\del_1,C_1),\del_2=(A_2,\del_2,C_2)$ be any pair of $\Ebb$-extensions. Let
\[ C_1\ov{\iota_1}{\lra}C_1\oplus C_2\ov{\iota_2}{\lla}C_2\ \text{and}\ 
A_1\ov{p_1}{\lla}A_1\oplus A_2\ov{p_2}{\lra}A_2 \]
be coproduct and product in $\cat$, respectively. By the additivity of $\Ebb$, we have a natural isomorphism
$\Ebb(C_1\oplus C_2,A_1\oplus A_2)\cong \Ebb(C_1,A_1)\oplus\Ebb(C_1,A_2)\oplus\Ebb(C_2,A_1)\oplus\Ebb(C_2,A_2)$ of $R$-modules.

Let $\del_1\oplus\del_2\in\Ebb(C_1\oplus C_2,A_1\oplus A_2)$ be the element corresponding to $(\del_1,0,0,\del_2)$ through this isomorphism. Namely, $\del_1\oplus\del_2$ is the unique element which satisfies
\[
\Ebb(\iota_i,p_j)(\del_1\oplus\del_2)=
\begin{cases}
\del_i& \text{if}\ i=j,\\
0& \text{if}\ i\ne j.
\end{cases}
\]
\end{definition}

\begin{definition}\label{DefSqEquiv}
Let $A,C\in\cat$ be any pair of objects. Two sequences of morphisms $A\ov{x}{\lra}B\ov{y}{\lra}C$ and $A\ov{x\ppr}{\lra}B\ppr\ov{y\ppr}{\lra}C$ in $\cat$
are said to be {\it equivalent} if there exists an isomorphism $b\in\cat(B,B\ppr)$ such that $b\circ x=x\ppr$, $y\ppr\circ b=y$.
The equivalence class of $A\ov{x}{\lra}B\ov{y}{\lra}C$ is denoted by $[A\ov{x}{\lra}B\ov{y}{\lra}C]$.

We also use the following notations.
\begin{enumerate}
\item For any $A,C\in\cat$, we denote as
\[ 0=[A\ov{\Big[\raise1ex\hbox{\leavevmode\vtop{\baselineskip-8ex \lineskip1ex \ialign{#\crcr{$\scriptstyle{1}$}\crcr{$\scriptstyle{0}$}\crcr}}}\Big]}{\lra}A\oplus C\ov{[0\ 1]}{\lra}C]. \]

\item For any $[A\ov{x}{\lra}B\ov{y}{\lra}C]$ and $[A\ppr\ov{x\ppr}{\lra}B\ppr\ov{y\ppr}{\lra}C\ppr]$, we denote as
\[ [A\ov{x}{\lra}B\ov{y}{\lra}C]\oplus [A\ppr\ov{x\ppr}{\lra}B\ppr\ov{y\ppr}{\lra}C\ppr]=[A\oplus A\ppr\ov{x\oplus x\ppr}{\lra}B\oplus B\ppr\ov{y\oplus y\ppr}{\lra}C\oplus C\ppr]. \]
\end{enumerate}
\end{definition}

\begin{definition}\label{DefRealization}
Let $\cat$ and $\Ebb$ be as before.
\begin{enumerate}
\item A {\it realization} of $\Ebb$ is an association which associates an equivalence class $\sfr(\del)=[A\ov{x}{\lra}B\ov{y}{\lra}C]$ to each $\Ebb$-extension $\del\in\Ebb(C,A)$, which satisfies the following condition $(\ast)$.
\begin{itemize}
\item[$(\ast)$] Let $\del\in\Ebb(C,A)$ and $\del\ppr\in\Ebb(C\ppr,A\ppr)$ be any pair of $\Ebb$-extensions, with
\[\sfr(\del)=[A\ov{x}{\lra}B\ov{y}{\lra}C]\text{ and } \sfr(\del\ppr)=[A\ppr\ov{x\ppr}{\lra}B\ppr\ov{y\ppr}{\lra}C\ppr].\]
Then, for any morphism $(a,c)\colon\del\to\del\ppr$, there exists some $b\in\cat(B,B\ppr)$ such that $b\circ x=x\ppr\circ a$, $y\ppr\circ b=c\circ y$.
\end{itemize}

\item A realization $\sfr$ of $\Ebb$ is {\it additive} if it satisfies the following conditions.
\begin{itemize}
\item For any $A,C\in\cat$, the split $\Ebb$-extension $0\in\Ebb(C,A)$ satisfies $\sfr(0)=0$.
\item For any pair of $\Ebb$-extensions $\del\in\Ebb(C,A)$ and $\del\ppr\in\Ebb(C\ppr,A\ppr)$, we have $\sfr(\del\oplus\del\ppr)=\sfr(\del)\oplus\sfr(\del\ppr)$.
\end{itemize}
\end{enumerate}
\end{definition}

\begin{definition}
Let $\sfr$ be an additive realization of $\Ebb$.
A sequence $A\ov{x}{\lra}B\ov{y}{\lra}C$ is said to {\it realize} $\del$, whenever it satisfies $\sfr(\del)=[A\ov{x}{\lra}B\ov{y}{\lra}C]$.
We also use the following terms.
\begin{itemize}
\item An {\it $\sfr$-triangle} is a pair $(A\ov{x}{\lra}B\ov{y}{\lra}C,\del)$ consisting of $\del\in\Ebb(C,A)$ and a sequence $A\ov{x}{\lra}B\ov{y}{\lra}C$ which realizes $\del$. We denote such a pair abbreviately by
\begin{equation}\label{sTriabb}
A\ov{x}{\lra}B\ov{y}{\lra}C\ov{\del}{\dra}.
\end{equation}
\item A morphism $x\in\cat(A,B)$ is called an {\it $\sfr$-inflation} if it admits some $\sfr$-triangle $(\ref{sTriabb})$.
\item A morphism $y\in\cat(B,C)$ is called an {\it $\sfr$-deflation} if it admits some $\sfr$-triangle $(\ref{sTriabb})$. 
\item A sequence $A\ov{x}{\lra}B\ov{y}{\lra}C$ is called an {\it $\sfr$-deflation} if it admits some $\sfr$-triangle $(\ref{sTriabb})$. 
\end{itemize}
When $\sfr$ is clear from the context, we often simply say {\it inflation} instead of $\sfr$-inflation. Similarly for {\it deflation} and {\it conflation}. We also abbreviately say {\it extriangle}, instead of $\sfr$-triangle.
\end{definition}

\begin{definition}\label{DefExtCat}
A triplet $(\cat,\Ebb,\sfr)$ is called an $R$-linear {\it extriangulated category} if it satisfies the following conditions.
\begin{itemize}
\item[{\rm (ET1)}] $\Ebb\colon\cat^{\mathrm{op}}\times\cat\to\Mod R$ is an $R$-bilinear functor.
\item[{\rm (ET2)}] $\sfr$ is an additive realization of $\Ebb$.
\item[{\rm (ET3)}] For any pair of $\Ebb$-extensions $\del\in\Ebb(C,A)$, $\del\ppr\in\Ebb(C\ppr,A\ppr)$ with
\[ \sfr(\del)=[A\ov{x}{\lra}B\ov{y}{\lra}C],\ \ \sfr(\del\ppr)=[A\ppr\ov{x\ppr}{\lra}B\ppr\ov{y\ppr}{\lra}C\ppr] \]
and any commutative square
\begin{equation}\label{SquareForET3}
\xy
(-12,6)*+{A}="0";
(0,6)*+{B}="2";
(12,6)*+{C}="4";
(-12,-6)*+{A\ppr}="10";
(0,-6)*+{B\ppr}="12";
(12,-6)*+{C\ppr}="14";
{\ar^{x} "0";"2"};
{\ar^{y} "2";"4"};
{\ar_{a} "0";"10"};
{\ar^{b} "2";"12"};
{\ar_{x\ppr} "10";"12"};
{\ar_{y\ppr} "12";"14"};
{\ar@{}|\circlearrowright "0";"12"};
\endxy
\end{equation}
in $\cat$, there exists a morphism $(a,c)\colon\del\to\del\ppr$ satisfying $c\circ y=y\ppr\circ b$.
\item[{\rm (ET3)$^{\mathrm{op}}$}] Dual of {\rm (ET3)}.
\item[{\rm (ET4)}] Let $\del\in\Ebb(D,A)$ and $\del\ppr\in\Ebb(F,B)$ be $\Ebb$-extensions realized by
\[ A\ov{f}{\lra}B\ov{f\ppr}{\lra}D\ \ \text{and}\ \ B\ov{g}{\lra}C\ov{g\ppr}{\lra}F \]
respectively. Then there exist an object $E\in\cat$, a commutative diagram
\begin{equation}\label{DiagET4}
\xy
(-21,7)*+{A}="0";
(-7,7)*+{B}="2";
(7,7)*+{D}="4";
(-21,-7)*+{A}="10";
(-7,-7)*+{C}="12";
(7,-7)*+{E}="14";
(-7,-21)*+{F}="22";
(7,-21)*+{F}="24";
{\ar^{f} "0";"2"};
{\ar^{f\ppr} "2";"4"};
{\ar@{=} "0";"10"};
{\ar_{g} "2";"12"};
{\ar^{d} "4";"14"};
{\ar_{h} "10";"12"};
{\ar_{h\ppr} "12";"14"};
{\ar_{g\ppr} "12";"22"};
{\ar^{e} "14";"24"};
{\ar@{=} "22";"24"};
{\ar@{}|\circlearrowright "0";"12"};
{\ar@{}|\circlearrowright "2";"14"};
{\ar@{}|\circlearrowright "12";"24"};
\endxy
\end{equation}
in $\cat$, and an $\Ebb$-extension $\del\pprr\in\Ebb(E,A)$ realized by $A\ov{h}{\lra}C\ov{h\ppr}{\lra}E$, which satisfy the following compatibilities.
\begin{itemize}
\item[{\rm (i)}] $D\ov{d}{\lra}E\ov{e}{\lra}F$ realizes $f\ppr\sas\del\ppr$,
\item[{\rm (ii)}] $d\uas\del\pprr=\del$,

\item[{\rm (iii)}] $f\sas\del\pprr=e\uas\del\ppr$. 
\end{itemize}

\item[{\rm (ET4)$^{\mathrm{op}}$}] Dual of {\rm (ET4)}.
\end{itemize}
\end{definition}

\subsection{Projective deflations}\label{Subsection_PD}

\begin{definition}\label{DefProjMorph}
A morphism $f\in\cat(X,Y)$ is said to be {\it projective} if $\Ebb(f,-)=f\uas$ is zero. Projective morphisms form an ideal of $\cat$ (\cite[Definition~1.22]{INP}). We will denote this ideal by $\Pcal$, and the quotient $\cat/\Pcal$ will be denoted by $\und{\cat}$ in the rest.

Dually, a morphism $f\in\cat(X,Y)$ is {\it injective} if $\Ebb(-,f)=f\sas$ is zero. We denote the ideal of injective morphisms by $\Ical$, and the quotient $\cat/\Ical$ by $\ovl{\cat}$.
\end{definition}

\begin{remark}\label{RemProjMorph}
The following holds. Dually for injectives.
\begin{enumerate}
\item If either of $X,Y\in\cat$ is a projective object, then any $f\in\cat(X,Y)$ is a projective morphism.
\item If $\CEs$ has enough projective objects, then a morphism is projective if and only if it factors through a projective object. In particular $\und{\cat}$ agrees with the usual stable category of $\cat$ in this case. 
\end{enumerate}
\end{remark}

\begin{definition}\label{DefProjDefl}
A projective morphism in $\cat$ which is also a deflation is called a {\it projective deflation}. We say that $\CEs$ {\it has enough projective morphisms} if any $C\in\cat$ admits a projective deflation $G\ov{g}{\lra}C$.
\end{definition}

\begin{remark}\label{RemProjDefl}
If $\CEs$ has enough projective objects, then it has enough projective morphisms. 
\end{remark}

Conversely, the following holds.
\begin{proposition}\label{ProjDefMin}
If a projective deflation $G\ov{g}{\lra}C$ is right minimal, then $G$ is a projective object.
\end{proposition}
\begin{proof}
Suppose that $G\ov{g}{\lra}C$ is a right minimal projective deflation.
If $Y\ov{y}{\lra}G$ is any deflation, then $g\circ y$ is also a deflation.
By projectivity of $g$, there exists $s\in\cat(G,Y)$ such that $g\circ y\circ s=g$.
Since $g$ is right minimal, it follows that $y\circ s$ is an isomorphism, and that $y$ is a retraction.
Hence $G$ is projective.
\end{proof}

\begin{corollary}\label{CorDefMin}
If every (indecomposable) $C\in\cat$ admits a right minimal projective deflation $G\to C$, then $\CEs$ has enough projective objects.
\end{corollary}
\begin{proof}
This is immediate from Proposition~\ref{ProjDefMin}.
\end{proof}

On the other hand, remark that if $p\colon P\to C$ is a deflation in $\CEs$ from a projective object $P$, then $[p\ 0]\colon P\oplus X\to C$ is a projective deflation for any $X\in\cat$.
This observation leads us to the following example in non-Krull--Schmidt case.
\begin{example}\label{ExNotProjButProj}
Let $K$ be an algebraically closed field, and let $Q$ be the quiver $1\leftarrow2$ of type $A_2$. Let $\A=\mod KQ$ be the abelian category of finite dimensional modules over the path algebra $KQ$. Put $P_1=1,P_2=\bsm2\\1\esm,S_2=2$ and $M(l,m,n)=P_1^{\oplus l}\oplus P_2^{\oplus m}\oplus S_2^{\oplus n}$ for $l,m,n\in\Nbb_{\ge0}$ for simplicity, and let $p\colon P_2\to S_2$ denote the canonical epimorphism taking the quotient by $\operatorname{rad} P_2$. 

Define $\cat\se\A$ to be the extension-closed (hence exact) subcategory consisting of those $X\in\A$ whose dimension vector $\dim X=[c_1\ c_2]$ satisfy $c_1\le c_2$. Equivalently, if we decompose $X$ into $X\cong M(l,m,n)$ in $\A$, then $X$ belongs to $\cat$ if and only if $n\ge l$. Especially we have $P_2,S_2,P_1\oplus S_2\in\cat$, but $P_1\notin\cat$.
We see the following for $M(l,m,n)\in\cat$.
\begin{enumerate}
\item $M(0,m,0)\in\cat$ is projective in $\A$, hence so is in $\cat$ for any $m\ge0$.
\item If $n\ge l>0$, then $M(l,m,n)\in\cat$ is not projective in $\cat$, since $\Ext^1_{\A}(M(l,m,n),M(1,0,2))\cong \Ext^1_{\A}(S_2,P_1)^{\oplus n}\ne0$.
\item If $n\ge l>0$, then $M(l,m,n)\in\cat$ does not admit any deflation from a projective object in $\cat$. Indeed, there is no epimorphism from any $M(0,m\ppr,n\ppr)$ already in $\A$, since $\operatorname{top}M(0,m\ppr,n\ppr)=S_2^{m\ppr+n\ppr}$.
\item For any $n\ge l\ge0$, there is a conflation
\[ M(n,0,u)\to M(l,m+n,u)\ov{q}{\lra}M(l,m,n) \]
in $\cat$ where we put $u=\max(l,n)$, isomorphic to the direct sum of the short exact sequences
\begin{eqnarray*}
&0\to P_1^{\oplus n}\to P_1^{\oplus l}\oplus P_2^{\oplus m}\oplus P_2^{\oplus n}\ov{\id\oplus\id\oplus p^{\oplus n}}{\lra}P_1^{\oplus l}\oplus P_2^{\oplus m}\oplus S_2^{\oplus n}\to0,&\\
&0\to S_2^{\oplus u}\ov{\id}{\lra}S_2^{\oplus u}\to0\to0&
\end{eqnarray*}
 in $\A$. In particular, the deflation $q$ is of the form
\[ [p\ppr\ 0]\colon P\oplus S\to M(l,m,n) \]
for a projective module $P=P_1^{\oplus l}\oplus P_2^{\oplus m}\oplus P_2^{\oplus n}$ in $\A$, and $q$ can be shown to be a projective morphism in $\cat$.

\end{enumerate}
The above {\rm (1)}$,\ldots$,{\rm (4)} shows that $\cat$ has enough projective morphisms, while it does not have enough projective objects.
\end{example}

\begin{remark}
Let $\Ucal\se\cat$ be a full additive subcategory consisting of projective-injective objects, and let $\cat/\Ucal$ denote the quotient of $\cat$ by $\Ucal$. Let $(\cat/\Ucal,\Ebb/\Ucal,\sfr/\Ucal)$ be the induced extriangulated category (\cite[Proposition~3.30]{NP1}).
Then any projective deflation $f$ in $\CEs$ induces a projective deflation in $(\cat/\Ucal,\Ebb/\Ucal,\sfr/\Ucal)$.
In particular if $\CEs$ has enough projective morphisms, then so does $(\cat/\Ucal,\Ebb/\Ucal,\sfr/\Ucal)$.
\end{remark}

\begin{definition}\label{DefDominant}
An $\sfr$-triangle $F\ov{f}{\lra}G\ov{g}{\lra}C\ov{\vt}{\dra}$ is called {\it dominant} if $g$ is a projective deflation. Also we call an $\Ebb$-extension $\vt$ {\it dominant} if it is realized by a dominant $\sfr$-triangle. Remark that $\vt$ is dominant if and only if $\vt\ush\colon\cat(F,-)\to\Ebb(C,-)$ is epimorphic.

Dually, an extension $\iota\in\Ebb(C,A)$ is called {\it codominant} if $\iota\ssh\colon\cat(-,C)\to\Ebb(-,A)$ is epimorphic, and an $\sfr$-triangle is {\it codominant} if it realizes a codominant extension.
\end{definition}

\subsection{Recollection of results on defects}\label{Subsection_Defect}

We recall the definitions of \emph{Auslander's defects} \cite{Aus1}, closely related to \emph{effaceable functors} \cite{Groth}. In the setting of extriangulated categories, they were introduced by Ogawa \cite{Ogawa}.

\begin{definition}
A \emph{right $\cat$-module} is a contravariant additive functor $\cat\op \to \Mod R$ to the category of $R$-modules. Dually, a \emph{left $\cat$-module} is a covariant additive functor to $R$-modules. Right and left $\cat$-modules form abelian categories  $\Mod \cat$ and $\Mod \cat\op$, respectively. 
\end{definition}

These categories have enough projectives, which are given by direct summands of direct sums of representable functors $\cat(-,A)=\Hom_{\cat}(-, A) \in \Mod \cat$, resp. $\cat(A,-)=\Hom_{\cat}(A, -) \in \cat\Mod$. 

\begin{definition}\label{Def_Subcats}
A $\cat$-module $M$ is \emph{finitely generated} if admits an epimorphism $\Hom_{\cat}(-, X) \twoheadrightarrow M$
from a representable functor. It is \emph{finitely presented} if it is a cokernel of a morphism of representable functors. 
A finitely presented module is called \emph{coherent} if each of its finitely generated submodule is also finitely presented.
\end{definition}

We have the following embeddings of full additive subcategories:

$$\coh(\cat) \hookrightarrow \fp(\cat) \hookrightarrow \fg(\cat) \hookrightarrow \Mod \cat.$$

Here the first three categories are the categories of coherent, finitely presented and finitely generated right $\cat$-modules, respectively.

The category of coherent modules is always abelian and closed under kernels, cokernels and extensions in $\Mod\cat$ (\cite[Proposition 1.5]{Herzog97}, see also \cite[Appendix B]{Fiorot16}). For $\fp(\cat)$ this holds if and only if $\cat$ has weak kernels. 

\begin{definition}{\rm (\cite[Definition~2.4]{Ogawa})}
The \emph{contravariant defect} of  a conflation $X \ov{f}\rightarrowtail Y \ov{g}\twoheadrightarrow Z$ is the cokernel of the morphism
$\Hom_{\cat}(-, g)\colon \Hom_{\cat}(-, Y) \to \Hom_{\cat}(-, Z)$ in $\Mod \cat$.
Note that it is also the kernel of 
$\Ebb(-, f)\colon \Ebb(-, X) \to \Ebb(-, Y)$.
\end{definition}

For an extriangulated category $\CEs$, let $\defff \Ebb$ be the full subcategory of right $\cat$-modules isomorphic to defects of conflations in $\CEs$.
We will also use the following notations in this article, especially in Section~\ref{Subsection_PR}.
\begin{definition}\label{DefOntoDefect}
For any extension $\del\in\Ebb(C,A)$, put $\Thh_{\del}=\Im\big(\del\ssh\colon\cat(-,C)\to\Ebb(-,A)\big)$. This is nothing but the contravariant defect of the sequence realizing $\del$.
We have a commutative diagram
\[
\xy
(-22,0)*+{\cat(-,C)}="0";
(5,0)*+{}="1";
(0,10)*+{\und{\cat}(-,C)}="2";
(0,-10)*+{\Thh_{\del}}="4";
(-6,0)*+{}="5";
(22,0)*+{\Ebb(-,A)}="6";
{\ar^{\und{(-)}} "0";"2"};
{\ar_{\und{\wp}_{\del}} "2";"4"};
{\ar^{\und{\del}\ssh} "2";"6"};
{\ar_{\wp_{\del}} "0";"4"};
{\ar_{i_{\del}} "4";"6"};
{\ar@{}|\circlearrowright "0";"1"};
{\ar@{}|\circlearrowright "5";"6"};
\endxy
\]
such that $i_{\del}\circ\wp_{\del}=\del\ssh$ holds.

For any morphism $(a,c)\colon {}_A\del_C\to{}_B\rho_D$ of $\Ebb$-extensions, there exists a unique morphism $\eta=\eta_{(a,c)}\colon\Thh_{\del}\to\Thh_{\rho}$
which makes
\[
\xy
(-18,11)*+{\cat(-,C)}="0";
(0,5)*+{\Thh_{\del}}="1";
(0,12)*+{}="1.5";
(18,11)*+{\Ebb(-,A)}="2";
(-18,-11)*+{\cat(-,D)}="10";
(0,-5)*+{\Thh_{\rho}}="11";
(0,-12)*+{}="11.5";
(18,-11)*+{\Ebb(-,B)}="12";
{\ar_{\wp_{\del}} "0";"1"};
{\ar_{i_{\del}} "1";"2"};
{\ar^{\del\ssh} "0";"2"};
{\ar^{\wp_{\rho}} "10";"11"};
{\ar^{i_{\rho}} "11";"12"};
{\ar_{\rho\ssh} "10";"12"};
{\ar_{c\circ-} "0";"10"};
{\ar_{\eta_{(a,c)}} "1";"11"};
{\ar^{a\sas} "2";"12"};
(10,0)*+{}="a";
(12,0)*+{}="b";
(-10,0)*+{}="aa";
(-12,0)*+{}="bb";
{\ar@{}|\circlearrowright "1";"1.5"};
{\ar@{}|\circlearrowright "11";"11.5"};
{\ar@{}|\circlearrowright "a";"b"};
{\ar@{}|\circlearrowright "aa";"bb"};
\endxy
\]
commutative in $\Mod\cat$.
Especially, $\eta$ is an epimorphism if $c$ is a split epimorphism, while $\eta$ is a monomorphism if $a$ is a split monomorphism.
\end{definition}

\begin{remark}\label{RemDeffp}
We can also regard as $\defff(\Ebb)\se\fp(\und{\cat})$. Indeed, exact sequence
\[ \cat(-,B)\ov{y\circ-}{\lra}\cat(-,C)\ov{\wp_{\del}}{\lra}\Thh_{\del}\to0 \]
in $\Mod\cat$ induces an exact sequence
\[ \und{\cat}(-,B)\ov{\und{y}\circ-}{\lra}\und{\cat}(-,C)\ov{\wp_{\del}}{\lra}\Thh_{\del}\to0 \]
in $\Mod\und{\cat}$ for any $\sfr$-triangle $A\ov{x}{\lra}B\ov{y}{\lra}C\ov{\del}{\dra}$, since we have $\Thh_{\del}|_{\Pcal}=0$.
\end{remark}

The following results are known.
\begin{theorem}{\rm  (\cite[Proposition~2.9]{Enomoto3}, \cite[Proposition~2.5]{Ogawa})}
For any extriangulated category $\CEs$, the category $\defff(\Ebb)$ is a Serre subcategory of $\coh(\cat)$.
\end{theorem}

\begin{lemma} {\rm (\cite[Proposition~2.16]{Ogawa}, see also \cite[Section~4]{INP})}
If $\CEs$ has enough projective objects, then the stable category $\und{\cat}$ has weak kernels and there exist isomorphisms between
$$\defff \Ebb \overset\sim\to \fp(\und{\cat}) \overset\sim\to \coh(\und{\cat}).$$
\end{lemma}

\section{Positive extensions}\label{Section_Positive}

\subsection{Definition of positive extensions}

Since $(\Mod R,\otr,R)$ is a complete and cocomplete symmetric monoidal closed category, we have the following bicategory of bimodules.
\begin{definition}\label{DefBicatB}
Bicategory $\mathbb{B}$ is given by the following.
\begin{itemize}
\item $0$-cells are small $R$-linear categories.
\item For a pair $\Asc,\Bsc$ of $0$-cells, the morphism category from $\Asc$ to $\Bsc$ is the category of $\Bsc$-$\Asc$-bimodules
\[ \mathbb{B}(\Asc,\Bsc)=\Bsc\Mod\Asc, \]
namely, the category of $R$-linear functors $\Asc\op\otr\Bsc\to\Mod R$. Remark that $R$-linear functors $\Asc\op\otr\Bsc\to\Mod R$ naturally correspond to $R$-bilinear functors $\Asc\op\ti\Bsc\to\Mod R$, bijectively.
\item For any $G\in\mathbb{B}(\Bsc,\cat)$ and $F\in\mathbb{B}(\Asc,\Bsc)$, their composition $G\dia F\in\mathbb{B}(\Asc,\cat)$ is defined by using coends as
\[ (G\dia F)(X,Y)=G(-,Y)\ot_{\Bsc}F(X,-)=\int^{B\in \Bsc}G(B,Y)\otr F(X,B) \]
for each $X\in\Asc,Y\in\cat$.
By the universality of coends, this gives the composition functor
\[ \dia\colon\mathbb{B}(\Bsc,\cat)\ti\mathbb{B}(\Asc,\Bsc)\to\mathbb{B}(\Asc,\cat) \]
for any triplet of $0$-cells $\Asc,\Bsc,\cat$ in $\mathbb{B}$.
\end{itemize}

In particular for each small $R$-linear category $\cat$, the  category $\mathbb{B}(\cat,\cat)=\cat\Mod\cat$ is equipped with the structure of a monoidal category $(\cat\Mod\cat,\dia,\cat(-,-))$.
For any $F\in\cat\Mod\cat$, we define $F^{\dia n}\in\cat\Mod\cat$ inductively by
\[ F^{\dia n}=F^{\dia (n-1)}\dia F \]
for any $n\in\Nbb_{\ge0}$.

In particular if $\CEs$ is extriangulated, then we obtain an $R$-bilinear functor $\Ebb^n:=\Ebb^{\dia n}\colon\cat\op\ti\cat\to\Mod R$ for any $n>0$. This is an extriangulated version of the construction given in \cite{Y}. Also, we put $\Ebb^0=\Hom_{\cat}$.
\end{definition}

\begin{remark}
By definition, if $\Ebb^n = 0$, we also have $\Ebb^k = 0$ for any $k \geq n$.
\end{remark}

\begin{remark}\label{RemFn}
Let $\cat$ and $F$ be as above.
For any $s,t\in\Nbb_{\ge0}$, we have a natural isomorphism $F^{\dia (s+t)}\cong F^{\dia s}\dia F^{\dia t}$.
\end{remark}

\begin{remark}\label{RemLaterAgree}
Later in Corollary \ref{CorPosExtEnoughProj}, we will see that if $\CEs$ has enough projective objects or enough injective objects, then $\Ebb^n$ defined above becomes isomorphic to those defined in \cite{HLN} and \cite{LN}.
\end{remark}

The aim of this section is to show the following.
\begin{theorem}\label{ThmPositiveExtension}
Let $\CEs$ be a small $R$-linear extriangulated category. Let $A\ov{x}{\lra}B\ov{y}{\lra}C\ov{\del}{\dra}$ be any $\sfr$-triangle.
\begin{enumerate}
\item We have a long exact sequence
\begin{eqnarray*}
&\hspace{-2cm}\cat(-,A)\to\cat(-,B)\to\cat(-,C)\to\Ebb(-,A)\to\cdots&\\
&\cdots\to\Ebb^{n-1}(-,C)\to\Ebb^n(-,A)\to\Ebb^n(-,B)\to\Ebb^n(-,C)\to\cdots&
\end{eqnarray*}
in $\Mod\cat$.
\item Dually, we have a long exact sequence
\begin{eqnarray*}
&\hspace{-2cm}\cat(C,-)\to\cat(B,-)\to\cat(A,-)\to\Ebb(C,-)\to\cdots&\\
&\cdots\to\Ebb^{n-1}(A,-)\to\Ebb^n(C,-)\to\Ebb^n(B,-)\to\Ebb^n(A,-)\to\cdots&
\end{eqnarray*}
in $\cat\Mod$.
\end{enumerate}
\end{theorem}

Since {\rm (2)} can be shown dually, we will mainly deal with {\rm (1)} in the rest.

\subsection{Description of elements in the coends}

Let $\cat$ be a small $R$-linear category.
\begin{definition}\label{DefSetCoend}
Let $E\in\Mod\cat$ and $F\in\cat\Mod$ be any pair of objects.
\begin{enumerate}
\item Define a set $S=S(E,F)$ by
\[ S(E,F)=\coprod_{M\in\cat}\Big(E(M)\ti F(M)\Big). \]
\item Define $\sim$ to be the equivalence relation on $S$, generated by
\[ (\kap,F(f)(\lam))\sim(E(f)(\kap),\lam) \]
where $f\in\cat(M,N),\kap\in E(N),\lam\in F(M)$ are arbitrary.
\end{enumerate}

We denote by $\ovl{(\rho,\lam)}$ the equivalence class to which $(\rho,\lam)\in E(M)\ti F(M)$ belongs.
\end{definition}

\begin{remark}\label{RemSetCoend}
The quotient set $S/\!\sim$ is nothing but the coend
\[ \int^{M\in\cat}(U\circ E)(M)\ti(U\circ F)(M) \]
taken over $\Sett$, where $U\colon\Mod R\to\Sett$ denotes the forgetful functor.
\end{remark}

\begin{lemma}\label{LemDescribeCoend}
Let $E\in\Mod\cat$ and $F\in\cat\Mod$ be any pair of objects.
Assume $\cat$ has finite direct sums. Then the following holds.
\begin{enumerate}
\item $S/\!\sim$ has the following structure of an $R$-module.
\begin{itemize}
\item For $(\rho_i,\lam_i)\in E(M_i)\ti F(M_i)$ for $i=1,2$, the sum of their equivalence classes is defined by
\[ \ovl{(\rho_1,\lam_1)}+\ovl{(\rho_2,\lam_2)}=\ovl{([\rho_1\ \rho_2],\left[\bsm\lam_1\\\lam_2\esm\right])} \]
where
$[\rho_1\ \rho_2]=E(p_1)(\rho_1)+E(p_2)(\rho_2)\in E(M_1\oplus M_2)$
and
$\left[\bsm\lam_1\\\lam_2\esm\right]=F(\iota_1)(\lam_1)+F(\iota_2)(\lam_2)\in F(M_1\oplus M_2)$
are given by using the direct sum
\[
\xy
(-19,0)*+{M_1}="0";
(-16,1)*+{}="10";
(-16,-1)*+{}="20";
(-8,1)*+{}="11";
(-8,-1)*+{}="21";
(0,0)*+{M_1\oplus M_2}="2";
(8,1)*+{}="13";
(8,-1)*+{}="23";
(19,0)*+{M_2}="4";
(16,1)*+{}="14";
(16,-1)*+{}="24";
{\ar^{\iota_1} "10";"11"};
{\ar^{p_1} "21";"20"};
{\ar_{p_2} "23";"24"};
{\ar_{\iota_2} "14";"13"};
\endxy.
\]
\item For $(\rho,\lam)\in E(M)\ti F(M)$ and $r\in R$, define the multiplication by $r$ by
\[ r\ovl{(\rho,\lam)}=\ovl{(r\rho,\lam)}. \]
Remark that we also have
\[ \ovl{(r\rho,\lam)}=\ovl{(E(r\cdot\id_{M})(\rho),\lam)}=\ovl{(\rho,F(r\cdot\id_{M})(\lam))}=\ovl{(\rho,r\lam)} \]
by the $R$-linearity of $E$ and $F$. Consequently, any $\rho\in E(M)$ or $\lam\in F(M)$ give the zero element $0=\ovl{(\rho,0)}=\ovl{(0,\lam)}$ in $S/\!\sim$.
\end{itemize}
\item For any $M\in\cat$, the canonical map
\[ \gam_M\colon E(M)\ti F(M)\to S/\!\sim \ ;\ (\rho,\lam)\mapsto \ovl{(\rho,\lam)} \]
is $R$-bilinear.
\end{enumerate}
\end{lemma}
\begin{proof}
{\rm (1)} For $(\rho_i,\lam_i)\in E(M_i)\ti F(M_i)$ for $i=1,2$, the twisting morphism $t=\left[\begin{array}{cc}0&1\\1&0\end{array}\right]\in\cat(M_1\oplus M_2,M_2\oplus M_1)$ gives
\begin{eqnarray*}
\ovl{([\rho_1\ \rho_2],\left[\bsm\lam_1\\\lam_2\esm\right])}
&=&\ovl{(E(t)([\rho_2\ \rho_1]),\left[\bsm\lam_1\\\lam_2\esm\right])}\\
&=&\ovl{([\rho_2\ \rho_1],F(t)(\left[\bsm\lam_1\\\lam_2\esm\right]))}=\ovl{([\rho_2\ \rho_1],\left[\bsm\lam_2\\\lam_1\esm\right])}.
\end{eqnarray*}
For any $f\in\cat(M,N),\kap\in E(N),\lam\in F(M)$ and $(\rho,\nu)\in E(L)\ti F(L)$, the morphism $f\oplus\id_L\in\cat(M\oplus L,N\oplus L)$ gives
\[ ([E(f)(\kap)\ \rho],\left[\bsm \lam\\\nu\esm\right])\sim([\kap\ \rho],\left[\bsm F(f)(\lam)\\\nu\esm\right]). \]

One can easily check the well-definedness of the addition using these properties.
The other conditions can be also confirmed easily. 

{\rm (2)} For any $\rho,\rho\ppr\in E(M)$ and $\lam\in F(M)$, we have
\[ (\rho+\rho\ppr,\lam)=(E(\Delta)([\rho\ \rho\ppr]),\lam)\sim([\rho\ \rho\ppr],F(\Delta)(\lam))=([\rho\ \rho\ppr],\left[\bsm\lam\\\lam\esm\right]) \]
for the diagonal morphism $\Delta\in\cat(M,M\oplus M)$, which shows $\ovl{(\rho+\rho\ppr,\lam)}=\ovl{(\rho,\lam)}+\ovl{(\rho\ppr,\lam)}$. The other conditions are obviously satisfied. 
\end{proof}

\begin{proposition}\label{PropDescribeCoend}
Let $E\in\Mod\cat$ and $F\in\cat\Mod$ be any pair of objects, and assume that $\cat$ has finite direct sums. Equip $S/\!\sim$ with the structure of an $R$-module given in Lemma~\ref{LemDescribeCoend}. Then the map
\[ \vp\colon S/\!\sim\ \to\int^{M\in\cat}E(M)\otr F(M) \]
which sends $\ovl{(\rho,\lam)}$ to the element represented by $\rho\ot\lam$, gives an isomorphism of $R$-modules.
\end{proposition}
\begin{proof}
$\vp$ is obtained as a map of sets given by the universality of the coend $S/\!\sim$. By Lemma~\ref{LemDescribeCoend} {\rm (2)}, we also have an $R$-linear map $\int^{M\in\cat}E(M)\otr F(M)\to S/\!\sim$, which gives the inverse of $\vp$. In particular $\vp$ becomes $R$-linear.
\end{proof}

\subsection{The notion of trivialization}\label{Subsection_Trivialization}

Let $\CEs$ be a small $R$-linear extriangulated category, and $F\in\cat\Mod$ be any object, throughout this subsection. We may identify $\int^{M\in\cat}E(M)\otr F(M)$ with $S/\!\sim$ through the isomorphism given in Proposition~\ref{PropDescribeCoend} in the rest.

Remark that the composition in $\mathbb{B}$ induces functors
\[ -\dia F\colon\mathbb{B}(\cat,\cat)\to\mathbb{B}(\Mod R,\cat),\quad -\dia F\colon\mathbb{B}(\cat,\Mod R)\to\mathbb{B}(\Mod R,\Mod R). \]
Through the equivalences
\[ \cat\Mod(\Mod R)\simeq\cat\Mod,\quad (\Mod R)\Mod(\Mod R)\simeq\Mod R \]
respectively, they correspond to $R$-linear functors
\[ -\dia F\colon\cat\Mod\cat\to\cat\Mod,\quad -\dia F\colon\Mod\cat\to\Mod R, \]
which we denote by the same symbol.
We also use the following notation.
\begin{definition}\label{DefNotationCup}
Put $\EF=\Ebb\dia F\in\cat\Mod$ for simplicity.
By definition, $\EF\colon\cat\to\Mod R$ is an $R$-linear functor defined by coends, satisfying
\[ \EF(A)=\int^{M\in\cat}\Ebb(M,A)\otr F(M) \]
for any object $A\in\cat$.

For any pair $(\rho,\lam)\in\Ebb(M,A)\ti F(M)$, we denote its equivalence class in $\EF(A)$ by $\rho\cup\lam=\ovl{(\rho,\lam)}$ in the following argument.
\end{definition}

\begin{claim}\label{ClaimComplex_sTriangle}
For any $\sfr$-triangle $A\ov{x}{\lra}B\ov{y}{\lra}C\ov{\del}{\dra}$, the sequence
\begin{equation}\label{Complex_sTriangle}
F(A)\ov{F(x)}{\lra}F(B)\ov{F(y)}{\lra}F(C)\ov{\del\cup-}{\lra}
\EF(A)\ov{\EF(x)}{\lra}\EF(B)\ov{\EF(y)}{\lra}\EF(C)
\end{equation}
is a complex in $\Mod R$.
\end{claim}
\begin{proof}
Remark that we have an exact sequence
\[ \cat(-,A)\ov{x\circ-}{\lra}\cat(-,B)\ov{y\circ-}{\lra}\cat(-,C)\ov{\del\ssh}{\lra}
\Ebb(-,A)\ov{x\sas}{\lra}\Ebb(-,B)\ov{y\sas}{\lra}\Ebb(-,C) \]
in $\Mod\cat$. Applying the additive functor $-\dia F\colon\Mod\cat\to\Mod R$, we obtain a complex in $\Mod R$ which is naturally isomorphic to $(\ref{Complex_sTriangle})$.
\end{proof}

\begin{remark}\label{RemComplex_sTriangle}
Later in Lemma~\ref{LemHigherExt}, we will show that $(\ref{Complex_sTriangle})$ becomes exact whenever $F$ is half-exact (Definition~\ref{DefHalfExact}) on $\CEs$.
\end{remark}

The following definition generalizes the notion of \emph{a trivialization of a 2-extension in an abelian category} considered by Dyckerhoff in \cite[Section 2.4.2]{Dyckerhoff}.

\begin{definition}\label{DefTrivialization}
Let $F\in\cat\Mod$ be any object. For any pair of objects $A,M\in\cat$, an element $(\rho,\lam)\in \Ebb(M,A)\ti F(M)$ is said to {\it have a trivialization} if it satisfies the following condition.
\begin{itemize}
\item For an $\sfr$-triangle which realizes $\rho$
\begin{equation}\label{sTri_rho}
A\ov{p}{\lra}X\ov{q}{\lra}M\ov{\rho}{\dra},
\end{equation}
there exists $\mu\in F(X)$ such that $\lam=F(q)(\mu)$.
\end{itemize}

Obviously, this condition does not depend on the choice of an $\sfr$-triangle $(\ref{sTri_rho})$ realizing $\rho$. Abbreviately, we call $\mu$ a {\it trivialization} of the pair $(\rho,\lam)$ with respect to $(\ref{sTri_rho})$.
\end{definition}

\begin{definition}\label{DefHalfExact}{\rm (\cite[Definition~2.7]{Ogawa})}
An object $F\in\cat\Mod$ is said to be {\it half-exact} on $\CEs$ if
\[ F(A)\ov{F(x)}{\lra}F(B)\ov{F(y)}{\lra}F(C) \]
is exact for any $\sfr$-conflation $A\ov{x}{\lra}B\ov{y}{\lra}C$.
\end{definition}

\begin{lemma}\label{LemTrivialization1}
Assume that $F\in\cat\Mod$ is half-exact on $\CEs$. Let $f\in\cat(M,N)$ be any morphism, and let $\rho\in\Ebb(N,A)$ and $\lam\in F(M)$ be any pair of elements.
Then $(f\uas\rho,\lam)\in\Ebb(M,A)\ti F(M)$ has a trivialization if and only if $(\rho,F(f)(\lam))\in\Ebb(N,A)\ti F(N)$ has a trivialization.
\end{lemma}
\begin{proof}
Realize $\rho$ and $f\uas\rho$ to obtain a morphism of $\sfr$-triangles
\[
\xy
(-12,6)*+{A}="0";
(0,6)*+{X}="2";
(12,6)*+{M}="4";
(24,6)*+{}="6";
(-12,-6)*+{A}="10";
(0,-6)*+{Y}="12";
(12,-6)*+{N}="14";
(24,-6)*+{}="16";
{\ar^{p} "0";"2"};
{\ar^{q} "2";"4"};
{\ar@{-->}^{f\uas\rho} "4";"6"};
{\ar@{=} "0";"10"};
{\ar_{g} "2";"12"};
{\ar^{f} "4";"14"};
{\ar_{s} "10";"12"};
{\ar_{t} "12";"14"};
{\ar@{-->}_{\rho} "14";"16"};
{\ar@{}|\circlearrowright "0";"12"};
{\ar@{}|\circlearrowright "2";"14"};
\endxy,
\]
so that
\[ X\ov{\left[\bsm -q\\ g\esm\right]}{\lra}M\oplus Y\ov{[f\ t]}{\lra}N\ov{p\sas\rho}{\dra} \]
becomes an $\sfr$-triangle (see \cite[Proposition~1.20]{LN} or \cite[Corollary~3.16]{NP1}). Assumption on $F$ implies that
\[
\xy
(-8,6)*+{F(X)}="0";
(8,6)*+{F(M)}="2";
(-8,-6)*+{F(Y)}="4";
(8,-6)*+{F(N)}="6";
{\ar^{F(q)} "0";"2"};
{\ar_{F(g)} "0";"4"};
{\ar^{F(f)} "2";"6"};
{\ar_{F(t)} "4";"6"};
{\ar@{}|\circlearrowright "0";"6"};
\endxy
\]
is a weak pullback.
Thus the existence of $\mu\in F(X)$ such that $F(q)(\mu)=\lam$ is equivalent to the existence of $\nu\in F(Y)$ such that $F(t)(\nu)=F(f)(\lam)$.
\end{proof}

\begin{definition}\label{DefTrivializationForEqClass}
Assume that $F\in\cat\Mod$ is half-exact on $\CEs$. We say that $\al=\rho\cup\lam\in\EF(A)$ has a trivialization if its representative $(\rho,\lam)\in\Ebb(M,A)\ti F(M)$ has a trivialization.
By Lemma~\ref{LemTrivialization1}, this does not depend on the choice of a representative.
\end{definition}

\begin{proposition}\label{PropTrivialization}
Assume that $F\in\cat\Mod$ is half-exact on $\CEs$.
For any $A,M\in\cat$ and any $\al\in\EF(A)$, the following are equivalent.
\begin{enumerate}
\item $\al=0$ in $\EF(A)$.
\item $\al$ has a trivialization.
\end{enumerate}
\end{proposition}
\begin{proof}
\noindent\und{$(1)\Rightarrow(2)$}
If $\al=0$, then we may choose $(0,0)\in \Ebb(0,A)\ti F(0)$ as a representative of $\al$, which has a trivialization in the obvious way.

\smallskip
\noindent\und{$(2)\Rightarrow(1)$}
Take a representative $(\rho,\lam)\in\Ebb(M,A)\ti F(M)$ of $\al$. Realize $\rho$ to obtain an $\sfr$-triangle $A\ov{p}{\lra}X\ov{q}{\lra}M\ov{\rho}{\dra}$. If $\al$ has a trivialization, there exists $\mu\in F(X)$ such that $F(q)(\mu)=\lam$. This gives
\[ (\rho,\lam)=(\rho,F(q)(\mu))\sim(q\uas\rho,\mu)=(0,\mu), \]
which shows $\al=0\cup\mu=0$.
\end{proof}

\subsection{Exactness for positive extensions}

Let $\CEs$ and $F$ be as in Subsection~\ref{Subsection_Trivialization}.
\begin{lemma}\label{LemHigherExt}
Assume that $F\in\cat\Mod$ is half-exact on $\CEs$.
Then for any $\sfr$-triangle $A\ov{x}{\lra}B\ov{y}{\lra}C\ov{\del}{\dra}$, the sequence
\[ F(A)\ov{F(x)}{\lra}F(B)\ov{F(y)}{\lra}F(C)\ov{\del\cup-}{\lra}
\EF(A)\ov{\EF(x)}{\lra}\EF(B)\ov{\EF(y)}{\lra}\EF(C) \]
is exact in $\Mod R$.
\end{lemma}
\begin{proof}
This is a complex by Claim~\ref{ClaimComplex_sTriangle}, which is exact at $F(B)$ by assumption.

\medskip
\noindent\und{Exactness at $F(C)$}
Let $\lam\in F(C)$ be any element. Suppose $\del\cup\lam=0$. By Proposition~\ref{PropTrivialization}, this means that $(\del,\lam)\in \Ebb(C,A)\ti F(A)$ has a trivialization, hence there exists $\mu\in F(B)$ such that $F(y)(\mu)=\lam$.

\medskip
\noindent\und{Exactness at $\EF(A)$}
Let $M\in\cat$ be any object, and let $(\rho,\lam)\in\Ebb(M,A)\ti F(M)$ be any element. Suppose that $\EF(x)(\rho\cup\lam)=(x\sas\rho)\cup\lam$ is zero in $\EF(B)$. Realize $\rho$ to obtain an $\sfr$-triangle
$A\ov{p}{\lra}X\ov{q}{\lra}M\ov{\rho}{\dra}$.

By \cite[Proposition~3.15]{NP1}, we obtain a diagram made of morphisms of $\sfr$-triangles
\[
\xy
(-7,7)*+{A}="2";
(7,7)*+{B}="4";
(21,7)*+{C}="6";
(35,7)*+{}="8";
(-7,-7)*+{X}="12";
(7,-7)*+{Y}="14";
(21,-7)*+{C}="16";
(35,-7)*+{}="18";
(-7,-21)*+{M}="22";
(7,-21)*+{M}="24";
(-7,-35)*+{}="32";
(7,-35)*+{}="34";
{\ar^{x} "2";"4"};
{\ar^{y} "4";"6"};
{\ar@{-->}^{\del} "6";"8"};
{\ar_{p} "2";"12"};
{\ar^{} "4";"14"};
{\ar@{=} "6";"16"};
{\ar_{} "12";"14"};
{\ar_{y\ppr} "14";"16"};
{\ar@{-->}_{p\sas\del} "16";"18"};
{\ar_{q} "12";"22"};
{\ar^{q\ppr} "14";"24"};
{\ar@{=} "22";"24"};
{\ar@{-->}_{\rho} "22";"32"};
{\ar@{-->}^{x\sas\rho} "24";"34"};
{\ar@{}|\circlearrowright "2";"14"};
{\ar@{}|\circlearrowright "4";"16"};
{\ar@{}|\circlearrowright "12";"24"};
\endxy
\]
satisfying $y^{\prime\ast}\del+q^{\prime\ast}\rho=0$. Since $(x\sas\rho,\lam)\in \Ebb(M,B)\ti F(M)$ has a trivialization by Proposition~\ref{PropTrivialization}, there exists $\mu\in F(Y)$ such that $F(q\ppr)(\mu)=\lam$. Then $-F(y\ppr)(\mu)=F(-y\ppr)(\mu)\in F(C)$ gives
\[ \del\cup(F(-y\ppr)(\mu))=(-y^{\prime\ast}\del)\cup\mu=(q^{\prime\ast}\rho)\cup\mu=\rho\cup(F(q\ppr)(\mu))=\rho\cup\lam \]
in $\EF(A)$.

\medskip
\noindent\und{Exactness at $\EF(B)$}
Let $M\in\cat$ be any object, and let $(\rho,\lam)\in\Ebb(M,B)\ti F(M)$ be any element. Suppose that $\EF(y)(\rho\cup\lam)=(y\sas\rho)\cup\lam$ is zero in $\EF(C)$. Realize $\rho$ to obtain an $\sfr$-triangle
$B\ov{p}{\lra}X\ov{q}{\lra}M\ov{\rho}{\dra}$.
By {\rm (ET4)}, we obtain a diagram made of morphisms of $\sfr$-triangles
\[
\xy
(-21,7)*+{A}="0";
(-7,7)*+{B}="2";
(7,7)*+{C}="4";
(-21,-7)*+{A}="10";
(-7,-7)*+{X}="12";
(7,-7)*+{Y}="14";
(-7,-21)*+{M}="22";
(7,-21)*+{M}="24";
{\ar^{x} "0";"2"};
{\ar^{y} "2";"4"};
{\ar^{\del}@{-->} "4";(19,7)};
{\ar@{=} "0";"10"};
{\ar_{p} "2";"12"};
{\ar^{p\ppr} "4";"14"};
{\ar_{p\circ x} "10";"12"};
{\ar_{y\ppr} "12";"14"};
{\ar@{-->}^{\del\ppr} "14";(19,-7)};
{\ar_{q} "12";"22"};
{\ar^{q\ppr} "14";"24"};
{\ar@{=} "22";"24"};
{\ar@{-->}_{\rho} "22";(-7,-34)};
{\ar@{-->}^{y\sas\rho} "24";(7,-34)};
{\ar@{}|\circlearrowright "0";"12"};
{\ar@{}|\circlearrowright "2";"14"};
{\ar@{}|\circlearrowright "12";"24"};
\endxy
\]
satisfying
$x\sas\del\ppr=q^{\prime\ast}\rho$.
Since $(y\sas\rho,\lam)\in\Ebb(M,C)\ti F(M)$ has a trivialization, there is $\mu\in F(Y)$ such that $F(q\ppr)(\mu)=\lam$. Then $\del\ppr\cup\mu\in\EF(A)$ gives
\[ \EF(x)(\del\ppr\cup\mu)=(x\sas\del\ppr)\cup\mu=(q^{\prime\ast}\rho)\cup\mu=\rho\cup(F(q\ppr)(\mu))=\rho\cup\lam \]
in $\EF(B)$.
\end{proof}

\bigskip
Now we prove {\rm (1)} in Theorem~\ref{ThmPositiveExtension}.
\begin{proof}[Proof of Theorem~\ref{ThmPositiveExtension} {\rm (1)}]
It suffices to show that
\begin{equation}\label{ExSeqFor_n}
\Ebb^{n-1}(-,C)\ov{\del\cup-}{\lra}\Ebb^n(-,A)\ov{x\sas}{\lra}\Ebb^n(-,B)\ov{y\sas}{\lra}\Ebb^n(-,C)
\end{equation}
is exact for any $n\in\Nbb_{>0}$ by induction.
Remark that the case $n=1$ follows from \cite[Corollary~3.12]{NP1}.

Suppose this was shown for $n-1$. Then in particular $\Ebb^{n-1}(X,-)\in\cat\Mod$ is half-exact on $\CEs$, for any $X\in\cat$. Lemma~\ref{LemHigherExt} applied to $F=\Ebb^{n-1}(X,-)$ shows that
\begin{equation}\label{ExSeqFor_nX}
F(C)\ov{\del\cup-}{\lra}\EF(A)\ov{\EF(x)}{\lra}\EF(B)\ov{\EF(y)}{\lra}\EF(C)
\end{equation}
is exact in $\Mod R$. Since $\Ebb^n\cong\Ebb\dia\Ebb^{n-1}$ in $\cat\Mod\cat$, we have a natural isomorphism $\EF\cong\Ebb^n(X,-)$ in $\cat\Mod$. Through this isomorphism, the sequence $(\ref{ExSeqFor_nX})$ becomes isomorphic to
\[ \Ebb^{n-1}(X,C)\ov{\del\cup-}{\lra}\Ebb^n(X,A)\ov{x\sas}{\lra}\Ebb^n(X,B)\ov{y\sas}{\lra}\Ebb^n(X,C). \]
As $X\in\cat$ was arbitrary, this shows the exactness of $(\ref{ExSeqFor_n})$.
\end{proof}

Let us see relations with projective deflations. 
In the rest, to make the notations concise, we will often write 
\[ \del\ssh=\del\cup-\colon\Ebb^n(-,C)\to\Ebb^{n+1}(-,A)\ \ \text{and}\ \ 
\del\ush=-\cup\del\colon\Ebb^n(A,-)\to\Ebb^{n+1}(C,-) \]
for any $\del\in\Ebb(C,A)$ and any $n\in\Nbb_{>0}$.

\begin{corollary}\label{CorDominant_fromThm}
Let $F\ov{f}{\lra}G\ov{g}{\lra}C\ov{\vt}{\dra}$ be a dominant $\sfr$-triangle. Then for any $n\in\Nbb_{>0}$,
\begin{equation}\label{SeqPosi}
\Ebb^{n-1}(G,-)\ov{f\uas}{\lra}\Ebb^{n-1}(F,-)\ov{\vt\ush}{\lra}\Ebb^n(C,-)\to0
\end{equation}
is exact in $\cat\Mod$. Namely $\Ebb^n(C,-)$ together with $\vt\ush$ gives a cokernel of $\Ebb^{n-1}(G,-)\ov{f\uas}{\lra}\Ebb^{n-1}(F,-)$.
\end{corollary}
\begin{proof}
This is immediate from Theorem~\ref{ThmPositiveExtension}, since $g\uas=\Ebb^n(g,-)=0$ holds for any $n\in\Nbb_{>0}$ by definition.
\end{proof}

If $\CEs$ has enough projective morphisms, then $\{\Ebb^n,\del\ssh\}_{n\ge0}$ is universal among sequences of functors $\{E^n,\ep^n\}_{n\ge0}$ of similar kind. More precisely, the following holds (Below, $\{E^n\}$ is an arbitrary sequence, for which we do not require $E^n=E^{\dia n}$. For a more general treatment of such sequences, see Subsection~\ref{Subsection_PA}).
\begin{proposition}\label{PropUnivPositiveDefl}
Assume that $\CEs$ has enough projective morphisms.
Let $E^{\mr}=\{E^n\}_{n\ge0}$ be a sequence of objects $E^n\in\cat\Mod\cat$ satisfying $E^0=\cat(-,-)$, equipped with a collection of morphisms $\ep^n_{\del}\in\cat(E^n(A,-),E^{n+1}(C,-))$ in $\cat\Mod$ for any $\Ebb$-extension $\del\in\Ebb(C,A)$ and $n\ge0$, which is natural with respect to morphisms of $\Ebb$-extensions.
The following holds.
\begin{enumerate}
\item For any $\sfr$-triangle $A\ov{x}{\lra}B\ov{y}{\lra}C\ov{\del}{\dra}$, there is an associated complex
\begin{equation}\label{CPX0}
\cdots\ov{\ep^{n-1}_{\del}}{\lra}E^{n}(C,-)\ov{E^n(y,-)}{\lra}E^{n}(B,-)\ov{E^n(x,-)}{\lra}E^{n}(A,-)\ov{\ep^n_{\del}}{\lra}\cdots
\end{equation}
in $\cat\Mod$ starting from $\cat(C,-)$ and continuing to the right.
\item The pair $\Ebb^{\mr}=\{\Ebb^n\}_{n\ge0}$ equipped with $\del\ush$ is universal among such pairs, in the sense that there exists a sequence of natural transformations
$\{\vp^n\colon \Ebb^n\to E^n\}_{n\ge0}$ which satisfies the following conditions.
\begin{itemize}
\item[{\rm (i)}] $\vp^0=\id$.
\item[{\rm (ii)}] For any $n\ge0$, the following diagram
\[
\xy
(-12,7)*+{\Ebb^{n}(A,-)}="0";
(12,7)*+{\Ebb^{n+1}(C,-)}="2";
(-12,-7)*+{E^{n}(A,-)}="4";
(12,-7)*+{E^{n+1}(C,-)}="6";
{\ar^{\del\ush} "0";"2"};
{\ar_{\vp^{n}_{A,-}} "0";"4"};
{\ar^{\vp^{n+1}_{C,-}} "2";"6"};
{\ar_{\ep_{\del}^n} "4";"6"};
{\ar@{}|\circlearrowright "0";"6"};
\endxy
\]
is commutative in $\cat\Mod$ for any $\del\in\Ebb(C,A)$.
\end{itemize}
\item Suppose that $E^{\mr}$ also satisfies $E^{n}(p,-)=0$ for any $n>0$ and any projective deflation $p$, and that it makes $(\ref{CPX0})$ exact for any $\sfr$-triangle. Then $\vp^n\colon \Ebb^n\to E^n$ is an isomorphism for any $n\ge0$.
\end{enumerate}
Dually for the case with enough injective morphisms.
\end{proposition}
\begin{proof}
{\rm (1)} The naturality with respect to morphisms of extensions show that $\ep_{(a\sas c\uas\del)}^n\colon E^n(A\ppr,-)\to E^{n+1}(C\ppr,-)$ is equal to the composition of 
\[ E^n(A\ppr,-)\ov{E^n(a,-)}{\lra}E^n(A,-)\ov{\ep_{\del}^n}{\lra}E^{n+1}(C,-)\ov{E^{n+1}(c,-)}{\lra}E^{n+1}(C\ppr,-) \]
for any $\del\in\Ebb(C,A),a\in\cat(A,A\ppr),c\in\cat(C\ppr,C)$ and any $n\ge0$.
This shows that $\ep_0^n=0$ holds for $0\in\Ebb(C,A)$, and that $(\ref{CPX0})$ indeed becomes a complex since we have $E^{n+1}(y,-)\circ \ep_{\del}^n=\ep_{(y\uas\del)}^n=0$ and $\ep_{\del}^n\circ E^n(x,-)=\ep_{(x\sas\del)}^n=0$. 

{\rm (2)} This can be shown in the same way as for the universality of usual derived functors on abelian or exact categories. Since we will give a similar proof in detail for {\it negative} extensions later in Lemma~\ref{LemUniversality} in Subsection~\ref{Subsection_Uni}, here we only give a brief explanation of how to construct $\vp^n$ inductively.

For $n=0$, we just put $\vp^0=\id$. Let $n>0$, and suppose that we have constructed $\vp^k$ for $k<n$ satisfying {\rm (i),(ii)}. For any $C\in\cat$, there exists a dominant $\sfr$-triangle $F\ov{f}{\lra}G\ov{g}{\lra}C\ov{\vt}{\dra}$ by assumption.
Then by the hypothesis of induction and by the universality of cokernel, there exists $\vp^n_{C,-}\colon \Ebb^n(C,-)\to E^n(C,-)$ unique so that
\[
\xy
(-37,7)*+{\Ebb^{n-1}(G,-)}="2";
(-9,7)*+{\Ebb^{n-1}(F,-)}="4";
(18,7)*+{\Ebb^{n}(C,-)}="6";
(35,7)*+{0}="8";
(-37,-7)*+{E^{n-1}(G,-)}="12";
(-9,-7)*+{E^{n-1}(F,-)}="14";
(18,-7)*+{E^{n}(C,-)}="16";
{\ar^{f\uas} "2";"4"};
{\ar^{\vt\ush} "4";"6"};
{\ar^{} "6";"8"};
{\ar_{\vp^{n-1}_{G,-}} "2";"12"};
{\ar_{\vp^{n-1}_{F,-}} "4";"14"};
{\ar^{\vp^{n}_{C,-}} "6";"16"};
{\ar_{E^{n-1}(f,-)} "12";"14"};
{\ar_{\ep_{\vt}^{n-1}} "14";"16"};
{\ar@{}|\circlearrowright "2";"14"};
{\ar@{}|\circlearrowright "4";"16"};
\endxy
\]
becomes commutative. It is straightforward to check that $\vp^n=\{\vp_{C,-}^n\}_{C\in\cat}$ gives a natural transformation $\vp^n\colon \Ebb^n\to E^n$ and that $\vp=\{\vp^n\}_{n\ge 0}$ satisfies the required conditions {\rm (i),(ii)}.

{\rm (3)} This is obvious, since the above proof suggests that $E^{\mr}$ possesses the same universality as $\Ebb^{\mr}$. 
\end{proof}

\begin{corollary}\label{CorPosExtEnoughProj}
If $\CEs$ has enough projective objects and enough injective objects, then $\Ebb^n$ for $n>0$ are isomorphic to those in \cite[Section~5.1]{HLN} or in \cite[Section~5.1]{LN}.
\end{corollary}
\begin{proof}
This is a particular case of Proposition~\ref{PropUnivPositiveDefl} {\rm (3)} and its dual.
\end{proof}

\begin{remark}\label{RemInsteadPositive}
As the above argument suggests, we may also harmlessly define functors $\Ebb^n$ for $n>0$ equipped with $\del\ush$ for any essentially small extriangulated category $\CEs$ with enough projective morphisms, inductively by
\[ \Ebb^n(C,-):=\Cok\big(\Ebb^{n-1}(G,-)\ov{f\uas}{\lra}\Ebb^{n-1}(F,-)\big) \]
by choosing dominant $\sfr$-triangle $F\ov{f}{\lra}G\ov{g}{\lra}C\ov{\vt}{\dra}$ for each $C\in\cat$. The usual argument (similar to the above proposition) shows that this is well-defined up to unique natural isomorphisms, independently of the choices of dominant $\sfr$-triangles.
Dually for the case with enough injective morphisms.
This alternative definition is essentially the iterative construction of \emph{(co)satellites} of $\Hom$ (considered as a functor in the first, resp. in the second argument), see \cite{CE, M, MR} and references therein.
\end{remark}

Proposition~\ref{PropUnivPositiveDefl} also implies the following (the reader might find it interesting to treat the case $n=2$ by direct computation):
\begin{corollary} Let $\cat$ be a triangulated category with shift functor $\Sigma$. View $\cat$ as an extriangulated category with $\Ebb=\cat(-,\Sigma-)$. Then, for any $n\geq 0$, we have $\Ebb^n \cong \cat(-,\Sigma^n-)$.
\end{corollary}

\begin{lemma}\label{LemProjDefl}
For any $C\in\cat$, the following are equivalent.
\begin{enumerate}
\item There is a projective deflation $G\ov{g}{\lra}C$.
\item There is a dominant $\Ebb$-extension $\vt\in\Ebb(C,F)$.
\item There is a dominant $\sfr$-triangle $F\ov{f}{\lra}G\ov{g}{\lra}C\ov{\vartheta}{\dra}$.
\item $\Ebb(C,-)$ is finitely presented.
\item $\Ebb(C,-)$ is finitely generated.
\end{enumerate}
\end{lemma}
\begin{proof}
Equivalences {\rm (1)} $\EQ$ {\rm (2)} $\EQ$ {\rm (3)} are obvious.
{\rm (3)} $\Rightarrow$ {\rm (4)} follows from Corollary~\ref{CorDominant_fromThm}.
{\rm (4)} $\Rightarrow$ {\rm (5)} is trivial. 
It remains to show {\rm (5)} $\Rightarrow$ {\rm (2)}. Suppose that there is an epimorphism $\cat(F,-)\ov{\wp}{\lra}\Ebb(C,-)$. By Yoneda lemma, $\vartheta=\wp_F(\id_F)\in\Ebb(C,F)$ satisfies $\wp=\vt\ssh$, hence is dominant.
\end{proof}

\begin{proposition}\label{PropCoherentPositive}
Assume that $\CEs$ has enough projective morphisms. Then $\Ebb^n(C,-)$ is finitely presented for any $C\in\cat$ and any $n\in\Nbb_{\ge0}$.
\end{proposition}
\begin{proof}
We proceed by an induction on $n\in\Nbb_{\ge0}$. This is obvious for $n=0$ (and also shown for $n=1$ in Lemma~\ref{LemProjDefl}).
Assume $n>0$, and suppose that we have shown the statement for $n-1$.
Take a dominant $\sfr$-triangle $F\to G\to C\ov{\vt}{\dra}$. By Corollary~\ref{CorDominant_fromThm}, this induces a right exact sequence
\[ \Ebb^{n-1}(G,-)\to\Ebb^{n-1}(F,-)\ov{\vt\ush}{\lra}\Ebb^n(C,-)\to0 \]
in $\cat\Mod$. Since $\Ebb^{n-1}(F,-)$ and $\Ebb^{n-1}(G,-)$ are finitely presented by the hypothesis, so is $\Ebb^n(C,-)$.
\end{proof}

\begin{corollary}\label{CorEquivProjDeflAdded}
For any $\CEs$, the following are equivalent.
\begin{enumerate}
\item $\CEs$ has enough projective morphisms.
\item $\Ebb(C,-)$ is finitely generated for any $C\in\cat$.
\item $\Ebb^n(C,-)$ is finitely presented for any $C\in\cat$ and any $n\in\Nbb_{\ge0}$.
\end{enumerate}
\end{corollary}
\begin{proof}
This is immediate from Lemma~\ref{LemProjDefl} and Proposition~\ref{PropCoherentPositive}.
\end{proof}

As a corollary, existence of projective deflations is ensured for example in the following situation.
\begin{corollary}\label{CorExistProjDefl}
Assume that $\cat$ is Krull--Schmidt with a finite number of isoclasses of indecomposables.
Let $C\in\cat$ be any object. If $\un{X\in\ind(\cat)}{\oplus}\Ebb(C,X)$ is finitely generated over $R$, then $\Ebb(C,-)$ is finitely presented. Here, $\ind(\cat)$ denotes a set of representatives of the isoclasses of indecomposable objects in $\cat$.

In particular if $\un{X\in\ind(\cat)}{\oplus}\Ebb(C,X)$ is finitely generated over $R$ for any $C\in\ind(\cat)$, then $\CEs$ has enough projective morphisms.
\end{corollary}
\begin{proof}
Let $\{\thh_i\in\Ebb(C,X_i)\mid 1\le i\le s\}$ be a set of generators of $\un{X\in\ind(\cat)}{\oplus}\Ebb(C,X)$. Put $F=\un{1\le i\le s}{\oplus}X_i$.
Realize $\vt=\left[\bsm\thh_1\\\vdots\\\thh_s\esm\right]$ to obtain an $\sfr$-triangle $F\ov{f}{\lra}G\ov{g}{\lra}C\ov{\vt}{\dra}$. By Lemma~\ref{LemProjDefl}, it suffices to show $\Ebb(g,X)=0$ for any $X\in\cat$.
Since $\cat$ is Krull--Schmidt, we may assume that $X$ is indecomposable. Let $\sig\in\Ebb(C,X)$ be any element. By the definition of $\vt$, there exists $x\in\cat(F,X)$ such that $x\sas\vt=\sig$. Thus we obtain $g\uas\sig=x\sas g\uas\vt=0$, as desired.
\end{proof}

\subsection{Positive global dimension}\label{Subsection_PositiveGlobalDim}

Let $\CEs$ be an extriangulated category which has $\Ebb^n$ for $n\ge0$. More precisely, assume that $\CEs$ is small, or essentially small with enough projective morphisms or enough injective morphisms.
\begin{definition}\label{DefPosGlDim}
We say that an extriangulated category $\CEs$ has \emph{positive global dimension} $m$, for some $m  \in \mathbb{Z}_{\geq 0}$, if $\Ebb^m \neq 0$ and $\Ebb^{m+1} = 0$. Equivalently, if $\Ebb^m \neq 0$ and $\Ebb^n = 0$ for any $n > m$. If such $m$ does not exist, we say that $\CEs$ has \emph{infinite positive global dimension}.
\end{definition}

\begin{remark}\label{RemHighExt_Diff}
If $\CEs$ is extriangulated and if $\Dcal\se\cat$ is an extension-closed category, then $\Dcal$ can be equipped with an induced structure of an extriangulated category $(\Dcal,\Ebb_{\Dcal},\sfr_{\Dcal})$.
Though it satisfies $\Ebb_{\Dcal}(X,Y)=\Ebb(X,Y)$ for any $X,Y\in\Dcal$ by definition, it does not necessarily satisfy $(\Ebb_{\Dcal})^n(X,Y)\cong\Ebb^n(X,Y)$ for $n\ge 2$ .

Remark that while $(\Ebb_{\Dcal})^m=0$ for some $m\in\Nbb_{>0}$ should imply $(\Ebb_{\Dcal})^n=0$ for all $n\ge m$ by definition, restrictions of $\{\Ebb^n\}_{n\in\Nbb_{>0}}$ to $\Dcal$ do not have this property in general. 
In fact, there are plenty of examples of extension-closed subcategories $\Dcal$ of a triangulated category $\Tcal$ which satisfies
$\Tcal(\Dcal,\Dcal [m])=0$ and $\Tcal(\Dcal,\Dcal [n])\ne0$ for $n>m>0$, as the following example shows.
\end{remark}

\begin{example}
Let $\ca$ be an abelian category. It is well known that $\ca$ is the heart of the canonical $t$-structure on its bounded derived category $\mathcal{D}^b(\ca)$, and $\Ext_{\A}^n(A, B) = \cd^b(\ca)(A, B[n])$ holds for all $n \in \mathbb{N}, A, B \in \ca$. If we now consider an $m$-periodic derived category $\cd_m(\ca)$ of $\ca$ (defined as the Verdier quotient of the category of $m$-periodic complexes up to homotopy by the category of $m$-periodic acyclic complexes), it does not have any $t$-structure. Since this category is periodic, we have $\cd_m(\ca)(A, B[km]) = \cd_m(\ca)(A, B) = \ca(A, B)$ for any $k \in \mathbb{Z}$. In general, this does not coincide with $\Ext^{km}(A, B)$.
\end{example}

\begin{example}
\begin{itemize}
\item[(i)] For an exact category, Definition~\ref{DefPosGlDim} agrees with the usual definition of the global dimension.
\item[(ii)] Let $\Tcal$ be a triangulated category and $A, B$ be a pair of objects with $\Tcal(A, B) \neq 0$. We have $$\Ebb_{\Tcal}^n(A[n], B) = \Tcal(A[n], B[n])\cong\Tcal(A, B) \neq 0$$ 
for any $n > 0$. Thus, any non-zero triangulated category has infinite global dimension.
\end{itemize}
\end{example}

\begin{example} \cite{LZ}
Assume that $\CEs$ is an extriangulated category with enough projectives and enough injectives. Let $(\Ucal, \Vcal)$ be a hereditary cotorsion pair on $\CEs$ and $\Wcal = \Ucal \cap \Vcal$. Denote $\Wcal = \Wcal_0$ and define by induction $\Wcal_{i+1} := \mbox{Cone}(\Wcal_i, \Wcal)$, for $i \geq 0$. By \cite[Lemma 5.4]{LZ}, the category $\Wcal_i$ has positive global dimension at most $i$. For $i = 1,$ this notion generalizes \emph{extended cohearts of co-$t-$structures on triangulated categories}, considered by Pauksztello and Zvonareva in \cite{PZ}.
\end{example}

\begin{example}\label{ExNTerm}
Let $\Lambda$ be a finite-dimensional algebra and let $\proj \Lambda $ be the category of finitely generated projectives over $\Lambda$. The category $K^{[-n, 0]}(\proj \Lambda)$ of $n$-term complexes with projective components in $K^b(\mod \Lambda)$ is an extriangulated category of positive global dimension $n$. It has enough projectives given by $\add \Lambda$ and enough injectives given by $\add \Lambda [n]$. 
\end{example}

\section{Connected sequences and $\delta$-functors}\label{Subsection_PA}

Let $\CEs$ be an $R$-linear extriangulated category, as before.

\subsection{Connected sequences of functors}

\begin{definition}\label{DefNegativeExtGeneral}
Let $\Fm=\{F^n\colon\cat\op\ti\cat\to\Mod R\}_{n\in\Zbb}$ be a sequence of $R$-bilinear functors.
We define as follows.
\begin{enumerate}
\item A {\it covariant connected sequence of functors} $(\Fm,\del\ssh)$
is a pair of $\Fm$ and a collection of morphisms $\del\ssh=\del_{F,\sharp}^n\in\cat(F^n(-,C),F^{n+1}(-,A))$ in $\Mod\cat$ for any $\Ebb$-extension $\del\in\Ebb(C,A)$ and $n\in\Zbb$, which is natural with respect to morphisms of $\Ebb$-extensions.

\item A {\it contravariant connected sequence of functors} $(\Fm,\del\ush)$ 
is a pair of $\Fm$ and a collection of morphisms
 $\del\ush=\del_{F}^{\sharp,n}\in\cat(F^n(A,-),F^{n+1}(C,-))$ in $\Mod\cat$ for any $\Ebb$-extension $\del\in\Ebb(C,A)$ and $n\in\Zbb$, which is natural with respect to morphisms of $\Ebb$-extensions.

\item A {\it bivariant connected sequence of functors} $(\Fm,\del\ssh,\del\ush)$  
is a triplet in which $(\Fm,\del\ssh)$ is a covariant, and $(\Fm,\del\ush)$ is a contravariant connected sequence, respectively.
\end{enumerate}
We can also consider bounded (at least on one side) sequences of functors. For $a \in \left\{- \infty\right\} \cup \Zbb$ and $b \in \Zbb \cup \left\{+ \infty \right\}$ such that $a<b$, we define covariant connected sequences of functors $(\Fm=\{F^n\colon\cat\op\ti\cat\to\Mod R\}_{n\in [a, b]}, \del\ssh)$ in the same way as above. We will also denote it by $((F^n),n\in [a, b], \del\ssh)$ or $((F^n)_{n\in [a, b]}, \del\ssh)$ in the sequel. 
Similarly, we define bounded (above, below, or on both sides) contravariant and bivariant connected sequences. 
In the sequel, we abbreviately write $f\sas=F^n(-,f)$ and  $f\uas=F^n(f,-)$ for any morphism $f$ in $\cat$, similarly as for $\Ebb^{\mr}$.
\end{definition}

\begin{proposition}\label{PropCPX}
The following holds.
\begin{enumerate}
\item If $(\Fm,\del\ssh)$ is a covariant connected sequence of functors, then
\begin{equation}\label{CPX1}
\cdots\ov{\del\ssh}{\lra}F^{n}(-,A)\ov{x\sas}{\lra}F^{n}(-,B)\ov{y\sas}{\lra}F^{n}(-,C)\ov{\del\ssh}{\lra}\cdots
\end{equation}
becomes a complex in $\Mod \cat$ for any $\sfr$-triangle $A\ov{x}{\lra}B\ov{y}{\lra}C\ov{\del}{\dra}$.
\item Similarly, if $(\Fm,\del\ush)$ is a contravariant connected sequence of functors, then
\begin{equation}\label{CPX2}
\cdots\ov{\del\ush}{\lra}F^{n}(C,-)\ov{y\uas}{\lra}F^{n}(B,-)\ov{x\uas}{\lra}F^{n}(A,-)\ov{\del\ush}{\lra}\cdots
\end{equation}
becomes a complex in $\cat\Mod$ for any $\sfr$-triangle $A\ov{x}{\lra}B\ov{y}{\lra}C\ov{\del}{\dra}$.
\end{enumerate}
\end{proposition}
\begin{proof}
This can be shown in a similar way as in Proposition~\ref{PropUnivPositiveDefl} {\rm (1)}.
\end{proof}

\begin{remark}
In this article, we will mainly consider connected sequences $(\Fm, \del\ssh)$ that also satisfy
\begin{itemize}
\item $F^n=\Ebb^n$, and $\del\ssh\colon F^n(-,C)\to F^{n+1}(-,A)$ agrees with $\del\ssh\colon \Ebb^n(-,C)\to \Ebb^{n+1}(-,A)$ for each $A,C\in\cat$
\end{itemize}
for any $n>0$ (or for any $n\ge0$). In such a case, we simply say that $(\Fm, \del\ssh)$ \emph{has} $F^n=\Ebb^n$ \emph{with the canonical connecting morphisms} for $n>0$ (resp. for $n\ge0$). Similarly for contravariant/bivariant ones.
\end{remark}

In the rest, what we state about covariant connected sequences has its contravariant counterpart, and vice versa.

\begin{notation}\label{DefCohomF1}
Let $(\Fm,\del\ssh)$ be a covariant connected sequence of functors.

Since the complex $(\ref{CPX1})$ is determined uniquely up to isomorphisms by the extension $\del$ and does not depend on the choice of representative of $\sfr(\del)=[A\ov{x}{\lra}B\ov{y}{\lra}C]$, we will often abbreviate it to $\Gam(\del)$. Especially, its acyclicity does not depend on the choice of representatives.
\end{notation}

\subsection{$\delta$-functors}

\begin{definition}\label{DefAcyclicF}
Let $(\Fm,\del\ssh)$ be an unbounded covariant connected sequence.
We say that an $\Ebb$-extension $\del$ is {\it $(\Fm,\del\ssh)$-acyclic} or simply {\it $\Fm$-acyclic} if $\Gam(\del)$ is acyclic (see Notation~\ref{DefCohomF1}). We say that $(\Fm,\del\ssh)$ is a {\it $\delta$-functor on} $\CEs$ (or {\it $\sfr$-acyclic}, or just {\it acyclic}) if any $\Ebb$-extension $\del$ is $(\Fm,\del\ssh)$-acyclic. 

Similarly, an unbounded contravariant connected sequence of functors $(\Fm,\del\ush)$ is a $\del$-functor if $(\ref{CPX2})$
is acyclic for any $\sfr$-triangle $A\ov{x}{\lra}B\ov{y}{\lra}C\ov{\del}{\dra}$.

A bivariant connected sequence $(\Fm,\del\ssh,\del\ush)$ is called a {\it bivariant $\delta$-functor} if both $(\Fm,\del\ssh)$ and $(\Fm,\del\ush)$ are $\delta$-functors.
\end{definition}

The following is obvious from the definition.
\begin{remark}\label{RemFAcyclicInduce}
Let $(\Fm,\del\ssh)$ be a covariant sequence having $F^n=\Ebb^n$ with the canonical connecting morphisms for $n\ge0$ which is acyclic on $\CEs$. Acyclicity inherits to relative theories and extension-closed subcategories in the following way.
\begin{enumerate}
\item For any closed subfunctor $\Fbb\se\Ebb$, if we define $(F^{\prime\bullet},\del\ppr\ssh)$ by replacing positive parts by $\Fbb^n$, then $(F^{\prime\bullet},\del\ppr\ssh)$ is also acyclic on $(\cat,\Fbb,\sfr_{\Fbb})$.
\item For any extension-closed subcategory $\Dcal\se\cat$, the restriction of $(\Fm,\del\ssh)$ naturally gives an acyclic covariant connected sequence of functors $(\Fm,\del\ssh)|_{\Dcal\op\ti\Dcal}$ on $(\Dcal,\Ebb_{\Dcal},\sfr_{\Dcal})$.
\end{enumerate}
Similarly for contravariant/bivariant ones.
\end{remark}

It is natural to define versions of $\delta$-functors for bounded (below, above, or on both sides) connected sequences (that have a chance to be completed to actual $\delta$-functors).

\begin{definition}
For $a \in \left\{- \infty\right\} \cup \Zbb$ and $b \in \Zbb \cup \left\{+ \infty \right\}$ such that $a<b$, we say that a covariant connected sequences of functors $(\Fm=\{F^n\colon\cat\op\ti\cat\to\Mod R\}_{n\in [a, b]}, \del\ssh)$ is a \emph{partial $\delta$-functor} if 
for any $\sfr$-triangle $A\ov{x}{\lra}B\ov{y}{\lra}C\ov{\del}{\dra}$ the complex
\[ F^a(-,A)\ov{F^a(-,x)}{\lra}
F^a(-,B)\ov{F^a(-,y)}{\lra}
F^a(-,C)\ov{\del^a_{F,\sharp}}{\lra}F^{a+1}(-,A)\ov{F^{a+1}(-,x)}{\lra}\cdots
\ov{F^b(-,y)}{\lra}
F^b(-,C) \]
is exact except at the end-terms. 
Similarly, we define contravariant and bivariant partial $\delta$-functors. 
\end{definition}

\begin{example}
\begin{itemize}
    \item Each bifunctor which is half-exact in one of the arguments is a partial $\delta$-functor. In particular, $\Hom$ and $\Ebb$ are half-exact in either argument, and so they are both bivariant partial $\delta$-functors.

    \item More generally, by the results in the previous section, each interval in the connected sequence $(\Ebb^{\mr}=\{\Ebb^n\}_{n\ge0}, \del\ssh, \del\ush)$ is a bivariant partial $\delta$-functor.
\end{itemize}
\end{example}

\begin{example}\label{ExAcyclic}
Let $\CEs$ be as before. The following are examples of acyclic bivariant connected sequence of functors $(\Fm,\del\ssh,\del\ush)$ of $\Ebb$ satisfying $F^0=\Hom_{\cat}$.
\begin{enumerate}
\item Suppose that $\CEs$ is exact, in the sense that any inflation is monomorphic and any deflation is epimorphic. If we define $\Fm$ by
\[ F^n=\begin{cases}
\Ebb^n&n>0\\
\Hom_{\cat}&n=0\\
0&n<0
\end{cases},
\]
then it has a natural structure of acyclic bivariant connected sequence of functors of $\Ebb$, i.e. of a bivariant $\delta$-functor.
\item Suppose that $\CEs$ is an extension-closed subcategory of a triangulated category $(\Tcal,[1],\triangle)$. If we define $\Fm$ by
\[ F^n=\Tcal(-,-[n])|_{\cat\op\ti\cat}\quad(\fa n\in\Zbb), \]
then as in Remark~\ref{RemFAcyclicInduce} {\rm (2)}, it has a natural structure of a bivariant $\delta$-functor on $\Ebb$.
\end{enumerate}
\end{example}

\begin{example}\label{ExINP}
The following is a consequence of \cite[Theorem~1.26]{INP}.
\begin{enumerate}
\item Suppose that $\CEs$ has enough projective objects. If we define $\Fm$ by
\[ F^n=\begin{cases}
\Ebb^n&n>0\\
\und{\cat}(-,\Om^{n}(-))&n\le0
\end{cases},
\]
then it has a natural structure of acyclic covariant connected sequence of functors of $\Ebb$. Here $\Om\colon\und{\cat}\to\und{\cat}$ denotes the syzygy functor (\cite[Definition-Proposition~1.25]{INP}).
\item Suppose that $\CEs$ has enough injective objects. If we define $\Fm$ by
\[ F^n=\begin{cases}
\Ebb^n&n>0\\
\ovl{\cat}(\Sig^{n}(-),-)&n\le0
\end{cases},
\]
then it has a natural structure of acyclic contravariant connected sequence of functors on $\Ebb$.  Here $\Sig\colon\ovl{\cat}\to\ovl{\cat}$ denotes the cosyzygy functor.
\end{enumerate}
\end{example}

Given a partial $\delta$-functor, we have natural universality conditions on its continuations. Note that the directions of natural transformations are different for continuations to the left and to the right.

\begin{definition}
Let $((F^n), n \in [a, b]; \del\ush)$ be a partial $\delta$-functor. Let $((E^n), n \in [a, c]; \del\ush)$, with $c > b$, be a partial $\delta$-functor such that  $E^i = F^i$ for $i \in [a, b]$. We say that $((E^n), n \in [a, c]; \del\ush)$ is universal among such partial $\delta$-functors if for each $((G^n), n \in [a, c]; \del\ush)$ such that $E^i = F^i$ for $i \in [a, b]$, there  exists a unique sequence of natural transformations
$\{\vp^n\colon E^n\to G^n\}_{n \in [a, c]}$ which satisfies the following conditions.
\begin{itemize}
\item[{\rm (i)}] $\vp^n=\id$ for any $n\in [a, b]$.
\item[{\rm (ii)}] For any $n \in [a, c]$, the following diagram
\[
\xy
(-12,7)*+{E^{n}(A,-)}="0";
(12,7)*+{E^{n+1}(C,-)}="2";
(-12,-7)*+{G^{n}(A,-)}="4";
(12,-7)*+{G^{n+1}(C,-)}="6";
{\ar^{\del\ush} "0";"2"};
{\ar_{\vp^{n}_{A,-}} "0";"4"};
{\ar^{\vp^{n+1}_{C,-}} "2";"6"};
{\ar_{\del\ush} "4";"6"};
{\ar@{}|\circlearrowright "0";"6"};
\endxy
\]
is commutative in $\cat\Mod$ for any $\del\in\Ebb(C,A)$.
\end{itemize}
Similar definitions apply for covariant and for bivariant partial $\delta$-functors. 
\end{definition}

\begin{definition}
Let $(F^n), n \in [a, b]; \del\ush)$ be a partial $\delta$-functor. Let $((E^n), n \in [d, b]; \del\ush)$, with $ d < a$, be a partial $\delta$-functor such that  $E^i = F^i$ for $i \in [a, b]$. We say that $((E^n), n \in [d, b]; \del\ush)$ is universal among such partial $\delta$-functors if for each $((G^n), n \in [d, b]; \del\ush)$ such that $E^i = F^i$ for $i \in [a, b]$, there exists a sequence of natural transformations
$\{\vp^n\colon G^n\to E^n\}_{n \in [d, b]}$ which satisfies the following conditions.
\begin{itemize}
\item[{\rm (i)}] $\vp^n=\id$ for any $n\in [a, b]$.
\item[{\rm (ii)}] For any $n \in [d, b]$, the following diagram
\[
\xy
(-12,7)*+{G^{n}(A,-)}="0";
(12,7)*+{G^{n+1}(C,-)}="2";
(-12,-7)*+{E^{n}(A,-)}="4";
(12,-7)*+{E^{n+1}(C,-)}="6";
{\ar^{\del\ush} "0";"2"};
{\ar_{\vp^{n}_{A,-}} "0";"4"};
{\ar^{\vp^{n+1}_{C,-}} "2";"6"};
{\ar_{\ep_{\del}^n} "4";"6"};
{\ar@{}|\circlearrowright "0";"6"};
\endxy
\]
is commutative in $\cat\Mod$ for any $\del\in\Ebb(C,A)$.
\end{itemize}
Similar definitions apply for covariant and for bivariant partial $\delta$-functors.
\end{definition}

\begin{remark}
Note that a bivariant partial $\delta$-functor extending a given bivariant partial $\delta$-functor $((F^n)_{n \in [a, b]}, \del\ssh, \del\ush)$ is said to be universal if it is universal among such partial bivariant $\delta$-functors. It is not required to be universal among covariant partial $\delta$-functors  extending $(F^n, \del\ssh)$ or universal among contravariant partial $\delta$-functors extending $(F^n, \del\ush)$.
\end{remark}

\begin{definition}
Let $F$ be a bifunctor on $\CEs$ which is half-exact in one of the arguments. We can consider it as a partial $\delta$-functor with $a = b = 0$. 

Let $((E^n)_{n \geq 0}, \del\ssh)$ (resp. $((E^n)_{n \geq 0}, \del\ush)$) be universal among partial $\delta$-functors in non-negative degrees having $F$ in degree $0$. We say that $E^n$ is the \emph{$n$-th right derived functor} of $F$.

Similarly, let $((E^n)_{n \leq 0}, \del\ssh)$ (resp. $((E^n)_{n \leq 0}, \del\ush)$) be universal among partial $\delta$-functors in non-positive degrees having $F$ in degree $0$. We say that $E^{-n}$ is the \emph{$n$-th left derived functor} of $F$. 
\end{definition}

\begin{definition}
Let $F$ be a bifunctor which is half-exact in each argument. Let $((E^n)_{n \geq 0}, \del\ssh, \del\ush)$ be universal among partial bivariant $\delta$-functors in non-negative degrees having $F$ in degree $0$. We say that $E^n$ is the \emph{balanced $n$-th right derived functor} of $F$, or \emph{bivariant $n$-th right derived functor} of $F$. 

Similarly, let $((E^n)_{n \leq 0}, \del\ssh, \del\ush)$ be universal among partial bivariant $\delta$-functors in non-positive degrees having $F$ in degree $0$. We say that $E^{-n}$ is the \emph{balanced $n$-th left derived functor} of $F$, or \emph{bivariant $n$-th left derived functor} of $F$. 
\end{definition}

\begin{example}
By Proposition \ref{PropUnivPositiveDefl}, $\Ebb^n, n > 0$, is the $n$-th right derived functor of $\Hom$. More precisely, it is both the covariant and the contravariant derived functor, which is then the balanced derived functor.
\end{example}

Later on, we will discuss left derived functors of the bifunctor $\Hom$. We will define its covariant and contravariant left derived functors. In general, they do not coincide, and so do not provide the balanced derived functors. We do expect, however, that for the bifunctor $\Hom$, there exist balanced versions of left derived functors. We cannot expect this for an arbitrary bifunctor which is half-exact in each argument, but for $\Hom$ and $\Ebb$, the existence of such a balance seems probable.

\subsection{Cohomology and relation to relative theories} \label{SectionDeltaRel}

Throughout this subsection, let $(\Fm,\del\ssh)$ be a covariant connected sequence of functors. 
\begin{definition}\label{DefCohomF2}
As in Definition~\ref{DefAcyclicF}, to any $\sfr$-triangle $A\ov{x}{\lra}B\ov{y}{\lra}C\ov{\del}{\dra}$, we may associate 
a complex $\Gam_F(A\ov{x}{\lra}B\ov{y}{\lra}C\ov{\del}{\dra})$ in $\Mod\cat$
\[ \cdots\ov{\del\ssh}{\lra}F^{n}(-,A)\ov{x\sas}{\lra}F^{n}(-,B)\ov{y\sas}{\lra} F^n(-,C)\ov{\del\ssh}{\lra}F^{n+1}(-,A)\ov{x\sas}{\lra}\cdots \]
with $F^0(-,A),F^0(-,B),F^0(-,C)$ sitting in degrees $-2,-1,0$, respectively.
As before, we will often denote it abbreviately by $\Gam_F(\del)$ in what follows.

Denote the $n$-th cohomology of $\Gam_F(\del)$ by $H_F^{n}(\del)$. By definition it is a functor $H_F^{n}(\del)\colon \cat\op\to\Mod R$ which sends each $X\in\cat$ to $(H_F^{n}(\del))(X)\in\Mod R$, the $n$-th cohomology of the complex of $R$-modules $(\Gam_F(\del))(X)$.
\end{definition}

\begin{remark}
If $\sfr(\del)=[A\ov{x}{\lra}B\ov{y}{\lra}C]$, we have
\begin{eqnarray*}
H_F^{3n}(\del)&\cong&\Ker(\del_{F,\sharp}^{n})\, /\, \Im(F^{n-1}(-,y)),\\
H_F^{3n-1}(\del)&\cong&\Ker(F^{n}(-,y))\, /\, \Im(F^{n}(-,x)),\\
H_F^{3n-2}(\del)&\cong&\Ker(F^{n}(-,x))\, /\, \Im(\del_{F,\sharp}^{n-1})
\end{eqnarray*}
for any $n\in\Zbb$ by definition.
\end{remark}

\begin{proposition}\label{PropConeCohom}
For any $\sfr$-triangle $A\ov{x}{\lra}B\ov{y}{\lra}C\ov{\del}{\dra}$, the following holds.
\begin{enumerate}
\item Let
\begin{equation}\label{t1}
\xy
(-12,6)*+{A}="0";
(0,6)*+{B}="2";
(12,6)*+{C}="4";
(24,6)*+{}="6";
(-12,-6)*+{A\ppr}="10";
(0,-6)*+{B\ppr}="12";
(12,-6)*+{C}="14";
(24,-6)*+{}="16";
{\ar^{x} "0";"2"};
{\ar^{y} "2";"4"};
{\ar@{-->}^{\del} "4";"6"};
{\ar_{a} "0";"10"};
{\ar_{b} "2";"12"};
{\ar@{=} "4";"14"};
{\ar_{x\ppr} "10";"12"};
{\ar_{y\ppr} "12";"14"};
{\ar@{-->}_{a\sas\del} "14";"16"};
{\ar@{}|\circlearrowright "0";"12"};
{\ar@{}|\circlearrowright "2";"14"};
\endxy
\end{equation}
be a morphism of $\sfr$-triangles such that
\begin{equation}\label{e1}
A\ov{\left[\bsm x\\ a\esm\right]}{\lra}B\oplus A\ppr\ov{[b\ -x\ppr]}{\lra}B\ppr\ov{y^{\prime\ast}\del}{\dra}
\end{equation}
is an $\sfr$-triangle. Then we have a long exact sequence
\[
\cdots\to H_F^{n-1}(a\sas\del)\to H_F^{n}(y^{\prime\ast}\del)\to H_F^{n}(\del)\to H_F^{n}(a\sas\del)\to\cdots\quad(n\in\Zbb)
\]
of cohomologies.
\item Let
\begin{equation}\label{t2}
\xy
(-12,6)*+{A}="0";
(0,6)*+{B\pprr}="2";
(12,6)*+{C\pprr}="4";
(24,6)*+{}="6";
(-12,-6)*+{A}="10";
(0,-6)*+{B}="12";
(12,-6)*+{C}="14";
(24,-6)*+{}="16";
{\ar^{x\pprr} "0";"2"};
{\ar^{y\pprr} "2";"4"};
{\ar@{-->}^{c\uas\del} "4";"6"};
{\ar@{=} "0";"10"};
{\ar_{b\ppr} "2";"12"};
{\ar_{c} "4";"14"};
{\ar_{x} "10";"12"};
{\ar_{y} "12";"14"};
{\ar@{-->}_{\del} "14";"16"};
{\ar@{}|\circlearrowright "0";"12"};
{\ar@{}|\circlearrowright "2";"14"};
\endxy
\end{equation}
be a morphism of $\sfr$-triangles such that
\begin{equation}\label{e2}
B\pprr\ov{\left[\bsm -y\pprr\\ b\ppr\esm\right]}{\lra}C\pprr\oplus B\ppr\ov{[c\ y]}{\lra}C\ppr\ov{x\pprr\sas\del}{\dra}
\end{equation}
is an $\sfr$-triangle. Then we have a long exact sequence
\[
\cdots\to H_F^{n-1}(x\pprr\sas\del)\to H_F^{n}(c\uas\del)\to H_F^{n}(\del)\to H_F^{n}(x\pprr\sas\del)\to\cdots\quad(n\in\Zbb)
\]
of cohomologies.
\end{enumerate}
\end{proposition}
\begin{proof}
Since {\rm (1)} can be shown in a similar way as {\rm (2)}, it suffices to show {\rm (2)}. If we are given $(\ref{t2})$, it yields a morphism of complexes $\vp\colon\Gam_F(c\uas\del)\to\Gam(\del)$. Its mapping cone $C(\vp)$ is naturally quasi-isomorphic to $\Gam_F(x\pprr\sas\del)$.
\end{proof}

\begin{corollary}\label{CorConeCohom}
For any $\Ebb$-extension $\del\in\Ebb(C,A)$, the following holds.
\begin{enumerate}
\item For any morphism $a\in\cat(A,A\ppr)$, there exists some deflation $y\ppr\in\cat(B\ppr,C)$ which induces an exact sequence $(\ref{e1})$.
\item For any morphism $c\in\cat(C\pprr,C)$, there exists some inflation $x\pprr\in\cat(A,B\pprr)$ which induces an exact sequence $(\ref{e2})$.
\end{enumerate}
\end{corollary}
\begin{proof}
{\rm (1)} The given data gives rise to a morphism of $\sfr$-triangle $(\ref{t1})$ such that $(\ref{e1})$ becomes an $\sfr$-triangle. Thus Proposition~\ref{PropConeCohom} {}\rm (1) can be applied. Similarly for {\rm (2)}.
\end{proof}

\begin{proposition}\label{PropAcyclic}
Let $\del\in\Ebb(C,A)$ be any $\Ebb$-extension. The following are equivalent.
\begin{enumerate}
\item $x\sas\del$ is $\Fm$-acyclic for any inflation $x$ from $A$.
\item $y\uas\del$ is $\Fm$-acyclic for any deflation $y$ to $C$.
\end{enumerate}
\end{proposition}
\begin{proof}
Since either of {\rm (1),(2)} in particular implies that $\del$ is $\Fm$-acyclic, this is immediate from Corollary~\ref{CorConeCohom}.
\end{proof}

\begin{proposition}\label{PropLast}
Suppose that $(\Fm,\del\ssh)$ is given, as before. Assume that there exists $m \in \Zbb$ such that $H^n(\delta) = 0$, for all $n > m$ and for all $\delta$. The following are equivalent.
\begin{enumerate}
\item $\Fm$ is a $\delta$-functor on $\CEs$.
\item For any $\Ebb$-extension $\del$, there exists an $\Fm$-acyclic $\Ebb$-extension $\thh$ and a morphism $a$ in $\cat$ such that $\del=a\sas\thh$.
\end{enumerate}
\end{proposition}
\begin{proof}
$(1)\Rightarrow(2)$ is trivial. Conversely, suppose that $(2)$ is satisfied. Then by using Corollary~\ref{CorConeCohom}, we can show that 
\begin{itemize}
\item $H^{n}(\del)=0$ for any $\Ebb$-extension $\del$
\end{itemize}
holds for any $n\in\Zbb$ by a descending induction on $n$, since we have $H^n(\del)=0$ for $n > m$.
\end{proof}

The dual statement to Proposition \ref{PropLast} holds for $(\Fm,\del\ssh)$ and $m$ such that $H^n(\delta) = 0$, for all $n < m$ and for all $\delta$.

The following is immediate from the above discussion.

\begin{corollary}\label{CorAcyclic}
Suppose that $(\Fm,\del\ssh)$ is a connected sequence of functors. Assume that there exists $m \in \Zbb$ such that either $((F^n), n > m; \del\ssh)$, or $((F^n), n < m; \del\ssh)$ is a partial $\delta$-functor. Define $\Fbb\se\Ebb$ by
\begin{align*} \Fbb(C,A) & = \{\del\in\Ebb(C,A)\mid x\sas\del\ \text{is}\ \Fm\text{-acyclic for any inflation $x$ from $A$}\} \\
& = \{\del\in\Ebb(C,A)\mid y\uas\del\ \text{is}\ \Fm\text{-acyclic for any deflation $y$ to $C$}\}. 
\end{align*}
Then $\Fbb\se\Ebb$ is a closed sub-bifunctor. The resulting $(\cat,\Fbb,\sfr|_{\Fbb})$ is the largest relative theory on which $(\Fm,\del\ssh)$ becomes a $\delta$-functor.
\end{corollary}

\begin{remark}
Assume that $(\cat,\Ebb,\sfr)$ satisfies condition (WIC) of~\cite{NP1}, meaning that if $g\circ f$ is an inflation, then $f$ is an inflation and dually for deflations. Then
$(\cat,\Fbb,\sfr|_{\Fbb})$ also satisfies condition (WIC).
\end{remark}

\section{Negative extensions}\label{Section_Negative}

In this section, we introduce covariant and contravariant versions of negative extensions, obtained from $\Ebb^n$ by taking $\cat$-duals. We discuss their balance under certain conditions and their universality. We conclude the section by some conjectures on the existence of balanced negative extensions in the general situation when these conditions are not necessarily satisfied.

For simplicity, we use the following notations.
\begin{itemize}
\item $\Ebb^0=\Hom_{\cat}$ as before, hence $\Ebb^0(X,Y)=\cat(X,Y)$ for any $X,Y\in\cat$.
\item $f\sas=f\circ-\colon\Ebb^0(-,A)\to\Ebb^0(-,B)$ and $f\uas=-\circ f\colon\Ebb^0(B,-)\to\Ebb^0(A,-)$ for any $f\in\cat(A,B)$.
\end{itemize}

\subsection{Negative extensions given by $\cat$-duals}\label{Subsection_CD}

Let $\CEs$ be a small extriangulated category.
\begin{definition}\label{DefNegativeExtension}
For each $n\in\Nbb_{\ge0}$,
we define $R$-bilinear functors
$\EbbI^{-n},\EbbII^{-n}\colon\cat\op\ti\cat\to\Mod R$
in the following way.
\begin{enumerate}
\item Using the $\cat$-dual $(-)^{\vee}\colon \cat\Mod\to\Mod\cat$, define $\EbbI^{-n}$ by the following.
\begin{itemize}
\item For any $A\in\cat$, define $\EbbI^{-n}(-,A)\in\Mod\cat$ by $\EbbI^{-n}(-,A)=\big( \Ebb^n(A,-) \big)^{\vee}$.
\item For any morphism $a\in\cat(A,A\ppr)$, define $a\sas=\EbbI^{-n}(-,a)\colon\EbbI^{-n}(-,A)\to\EbbI^{-n}(-,A\ppr)$ by $\EbbI^{-n}(-,a)=\big( \Ebb^n(a,-) \big)^{\vee}$. 
\end{itemize}
\item Using the $\cat$-dual $(-)^{\vee}\colon \Mod\cat\to\cat\Mod$, define $\EbbII^{-n}$ by the following.
\begin{itemize}
\item For any $C\in\cat$, define $\EbbII^{-n}(C,-)\in\cat\Mod$ by $\EbbII^{-n}(C,-)=\big( \Ebb^n(-,C) \big)^{\vee}$.
\item For any morphism $c\in\cat(C\ppr,C)$, define $c\uas=\EbbII^{-n}(c,-)\colon\EbbII^{-n}(C,-)\to\EbbII^{-n}(C\ppr,-)$ by $\EbbII^{-n}(c,-)=\big( \Ebb^n(-,c) \big)^{\vee}$. 
\end{itemize}
\end{enumerate}

We also put $\EbbI^n=\EbbII^n=\Ebb^n$ for positive integer $n$.
\end{definition}

\begin{remark}\label{RemNegativeExtension}
Let $n\in\Nbb_{\ge0}$ be any non-negative integer.
By definition, $\EbbI^{-n}$ satisfies
\[ \EbbI^{-n}(C,A)=(\cat\Mod)\big(\Ebb^{n}(A,-),\cat(C,-) \big) \]
for any $A,C\in\cat$. For any $a\in\cat(A,A\ppr), c\in\cat(C\ppr,C)$ and any $\vp\in\EbbI^{-n}(C,A)$, the element $a\sas c\uas\vp=c\uas a\sas\vp\in\EbbI^{-n}(C\ppr,A\ppr)$ is given the composition of
\[ \Ebb^n(A\ppr,-)\ov{a\uas}{\lra}\Ebb^n(A,-)\ov{\vp}{\lra}\cat(C,-)\ov{-\circ c}{\lra}\cat(C\ppr,-). \]
It can be easily checked that $\EbbI^{-n}\colon\cat\op\ti\cat\to\Mod R$ indeed forms an $R$-bilinear functor. Also remark that we naturally have $\EbbI^0=\Ebb^0$.
Similarly for $\EbbII^{\mr}$.
\end{remark}

The following completes $\EbbI^{\bullet},\EbbII^{\bullet}$ to connected sequences of functors. In non-negative degrees, these sequences coincide with the covariant, resp. contravariant partial $\delta$-functors given by $(\Ebb^n), n \ge 0$ together with the canonical connecting morphisms.
\begin{definition}\label{DefCdualProlong}
Let $\CEs$ be as above.
\begin{enumerate}
\item For any $\del\in\Ebb(C,A)$ and any $n\in\Nbb_{\ge0}$, define $\del\ssh\colon\EbbI^{-(n+1)}(-,C)\to\EbbI^{-n}(-,A)$ by
\[ \del\ssh=\big(\del\ush\colon\Ebb^n(A,-)\to\Ebb^{n+1}(C,-) \big)^{\vee} \]
using the $\cat$-dual $(-)^{\vee}\colon \cat\Mod\to\Mod\cat$. Together with those $\del\ssh$ which we already have for positive parts, it gives a covariant connected sequence $(\EbbI^{\bullet},\del\ssh)$.
\item  For any $\del\in\Ebb(C,A)$ and any $n\in\Nbb_{\ge0}$, define $\del\ush\colon\EbbII^{-(n+1)}(A,-)\to\EbbII^{-n}(C,-)$ by 
\[ \del\ush=\big(\del\ssh\colon\Ebb^n(-,C)\to\Ebb^{n+1}(-,A) \big)^{\vee} \]
using the $\cat$-dual $(-)^{\vee}\colon \Mod\cat\to\cat\Mod$. Similarly as in {\rm (1)}, it gives a contravariant connected sequence $(\EbbII^{\bullet},\del\ush)$.
\end{enumerate}
\end{definition}

In the rest we mainly deal with $\EbbI$, since $\EbbII$ can be dealt in a similar way.

The above definitions match well with the naive ones in the following specific cases.
\begin{proposition}\label{PropSpecific}
For any $n\in\Nbb_{>0}$ and any $X,Y\in\cat$, the following holds.
\begin{enumerate}
\item If $\CEs$ corresponds to a triangulated category $(\cat,[1],\triangle)$, then $\EbbI^{-n}(X,Y)\cong\cat(X,Y[-n])$.
\item If $\CEs$ is exact, or more generally if any inflation is monomorphic, then $\EbbI^{-n}(X,Y)=0$.
\end{enumerate}
\end{proposition}
\begin{proof}
{\rm (1)} Indeed we have
\begin{eqnarray*}
\EbbI^{-n}(X,Y)&=&(\cat\Mod)(\Ebb^n(Y,-),\cat(X,-))\\
&\cong&(\cat\Mod)(\cat(Y[-n],-),\cat(X,-))\cong\cat(X,Y[-n])
\end{eqnarray*}
by Yoneda lemma.

{\rm (2)} Let $\vp\in\EbbI^{-n}(X,Y)=(\cat\Mod)(\Ebb^n(Y,-),\cat(X,-))$ be any element.
It suffices to show that $\vp_Z(\al)=0$ holds for any $Z\in\cat$ and any $\al\in\Ebb^n(Y,Z)$. As before, by Proposition~\ref{PropDescribeCoend}, there exists $M\in\cat$, $\rho\in\Ebb(M,Z)$ and $\thh\in\Ebb^{n-1}(Y,M)$ such that $\al=\rho\cup\thh$. Realize $\rho$ to obtain an $\sfr$-triangle $Z\ov{z}{\lra}Q\to M\ov{\rho}{\dra}$, in which $z$ is a monomorphism. By the naturality of $\vp$,
\[
\xy
(-12,6)*+{\Ebb^n(Y,Z)}="0";
(12,6)*+{\cat(X,Z)}="2";
(-12,-6)*+{\Ebb^n(Y,Q)}="4";
(12,-6)*+{\cat(X,Q)}="6";
{\ar^{\vp_Z} "0";"2"};
{\ar_{z\sas} "0";"4"};
{\ar^{z\circ-} "2";"6"};
{\ar_{\vp_Q} "4";"6"};
{\ar@{}|\circlearrowright "0";"6"};
\endxy
\]
is commutative. In particular $z\circ\vp_Z(\al)=\vp_Q(z\sas\al)=0$ holds (because $z\sas\al=(z\sas\rho)\cup\theta=0\cup\theta=0$), hence $\vp_Z(\al)=0$ follows from the monomorphicity of $z$. 
\end{proof}

\subsection{Projective deflations and acyclicity}\label{Subsection_PD_Acycl}

Let $\CEs$ be as in the previous subsection. The aim of this subsection is to show that the existence of enough projective morphisms implies that $(\EbbI^{\bullet},\del\ssh)$ is a $\del$-functor (Theorem~\ref{ThmNegativeExtension}).

\begin{lemma}\label{LemDominant2}
Let $F\ov{f}{\lra}G\ov{g}{\lra}C\ov{\vt}{\dra}$ be a dominant $\sfr$-triangle. Then for any $n\in\Nbb_{>0}$,
\begin{equation}\label{SeqNega}
0\to\EbbI^{-n}(-,C)\ov{\vt\ssh}{\lra}\EbbI^{-(n-1)}(-,F)\ov{f\sas}{\lra}\EbbI^{-(n-1)}(-,G)
\end{equation}
is exact in $\Mod\cat$.
\end{lemma}
\begin{proof}
Sequence $(\ref{SeqNega})$ is the $\cat$-dual of $(\ref{SeqPosi})$ in Corollary~\ref{CorDominant_fromThm}, hence exact by the left exactness of $(-)^{\vee}\colon\cat\Mod\to\Mod\cat$. 
\end{proof}

\begin{proposition}\label{PropVanishNegative}
Assume that $\CEs$ has enough projective morphisms. Then for any morphism $x\in\cat(X,X\ppr)$, the following are equivalent.
\begin{enumerate}
\item $\EbbI^{-1}(x,-)=0$.
\item $\EbbI^{-n}(x,-)=0$ for any $n\in\Nbb_{>0}$.
\end{enumerate}
In particular if we take $x=\id_X$ for an object $X\in\cat$, we see that $\EbbI^{-1}(X,-)=0$ holds if and only if $\EbbI^{-n}(X,-)=0$ holds for any $n>0$.
\end{proposition}
\begin{proof}
{\rm (2)} $\Rightarrow$ {\rm (1)} is trivial. Let us show the converse.
Let $C\in\cat$ be any object. By assumption, there exists a dominant extension $\vt\in\Ebb(C,F)$. Then
\[
\xy
(-14,6)*+{\EbbI^{-n}(X\ppr,C)}="0";
(14,6)*+{\EbbI^{-(n-1)}(X\ppr,F)}="2";
(-14,-6)*+{\EbbI^{-n}(X,C)}="4";
(14,-6)*+{\EbbI^{-(n-1)}(X,F)}="6";
{\ar^(0.46){\vt\ssh} "0";"2"};
{\ar_{x\uas} "0";"4"};
{\ar^{x\uas} "2";"6"};
{\ar_(0.46){\vt\ssh} "4";"6"};
{\ar@{}|\circlearrowright "0";"6"};
\endxy
\]
is commutative, in which the horizontal arrows are monomorphic for any $n\ge 1$ by Lemma~\ref{LemDominant2}. Thus {\rm (2)} follows from {\rm (1)} by an induction on $n$.
\end{proof}

\begin{corollary}\label{CorExtFin}
Assume that $R=K$ is a field, and that $\cat$ is $\Hom$-finite over $K$. If $\CEs$ has enough projective morphisms, then $\dim_K\EbbI^n(X,Y)<\infty$ holds for any $X,Y\in\cat$ and any $n\in\Zbb$.
\end{corollary}
\begin{proof}
For non-negative $n$, this follows from Proposition~\ref{PropCoherentPositive}, since $\EbbI^n=\Ebb^n$ by definition. For negative $n$, this can be shown using the exactness of $(\ref{SeqNega})$ in Lemma~\ref{LemDominant2}, by an induction on $n$.
\end{proof}

\begin{theorem}\label{ThmNegativeExtension}
Assume that $\CEs$ has enough projective morphisms. Then $(\EbbI^{\bullet},\del\ssh)$ is a $\delta$-functor on $\CEs$.
\end{theorem}
\begin{proof}
Consider the following statements $\mathrm{(P_{\mathit{n}})}$ and $\mathrm{(Q_{\mathit{n}})}$ for each $n\in\Nbb_{\ge0}$.
\begin{itemize}
\item[$\mathrm{(P_{\mathit{n}})}$] Sequence
\[ \EbbI^{-(n+1)}(-,C)\ov{\del\ssh}{\lra}\EbbI^{-n}(-,A)\ov{x\sas}{\lra}\EbbI^{-n}(-,B) \]
is exact for any $\sfr$-triangle $A\ov{x}{\lra}B\ov{y}{\lra}C\ov{\del}{\dra}$.
\item[$\mathrm{(Q_{\mathit{n}})}$] Sequence
\[ \EbbI^{-n}(-,A)\ov{x\sas}{\lra}\EbbI^{-n}(-,B)\ov{y\sas}{\lra}\EbbI^{-n}(-,C)\ov{\del\ssh}{\lra}\EbbI^{-(n-1)}(-,A) \]
is exact for any $\sfr$-triangle $A\ov{x}{\lra}B\ov{y}{\lra}C\ov{\del}{\dra}$.
\end{itemize}
Remark that $\mathrm{(Q_{\mathit{0}})}$ holds. By an induction on $n$, it suffices to show the following {\rm (1)} and {\rm (2)} for any $n\in\Nbb_{\ge0}$.
\begin{enumerate}
\item $\mathrm{(Q_{\mathit{n}})}$ implies $\mathrm{(P_{\mathit{n}})}$.
\item $\mathrm{(P_{\mathit{n}})}$ and $\mathrm{(Q_{\mathit{n}})}$ imply $\mathrm{(Q}_{n+1}\mathrm{)}$.
\end{enumerate}

First we show {\rm (1)}. Suppose that $\mathrm{(Q_{\mathit{n}})}$ holds. Let $A\ov{x}{\lra}B\ov{y}{\lra}C\ov{\del}{\dra}$ be any $\sfr$-triangle. By assumption, there exists a dominant $\sfr$-triangle $F\to G\to C\ov{\vt}{\dra}$. Since $g\uas\del=0$ by projectivity, we obtain a morphism
\[
\xy
(-12,6)*+{F}="0";
(0,6)*+{G}="2";
(12,6)*+{C}="4";
(24,6)*+{}="6";
(-12,-6)*+{A}="10";
(0,-6)*+{B}="12";
(12,-6)*+{C}="14";
(24,-6)*+{}="16";
{\ar^{f} "0";"2"};
{\ar^{g} "2";"4"};
{\ar@{-->}^{\vt} "4";"6"};
{\ar_{a} "0";"10"};
{\ar_{b} "2";"12"};
{\ar@{=} "4";"14"};
{\ar_{x} "10";"12"};
{\ar_{y} "12";"14"};
{\ar@{-->}_{\del} "14";"16"};
{\ar@{}|\circlearrowright "0";"12"};
{\ar@{}|\circlearrowright "2";"14"};
\endxy
\]
of $\sfr$-triangles. This induces a commutative diagram
\[
\xy
(-58,6)*+{0}="0";
(-38,6)*+{\EbbI^{-(n+1)}(-,C)}="2";
(-12,6)*+{\EbbI^{-n}(-,F)}="4";
(12,6)*+{\EbbI^{-n}(-,G)}="6";
(36,6)*+{\EbbI^{-n}(-,C)}="8";
(-38,-6)*+{\EbbI^{-(n+1)}(-,C)}="12";
(-12,-6)*+{\EbbI^{-n}(-,A)}="14";
(12,-6)*+{\EbbI^{-n}(-,B)}="16";
(36,-6)*+{\EbbI^{-n}(-,C)}="18";
{\ar "0";"2"};
{\ar^(0.56){\vt\ssh} "2";"4"};
{\ar^{f\sas} "4";"6"};
{\ar^{g\sas} "6";"8"};
{\ar@{=} "2";"12"};
{\ar^{a\sas} "4";"14"};
{\ar_{b\sas} "6";"16"};
{\ar@{=} "8";"18"};
{\ar_(0.56){\del\ssh} "12";"14"};
{\ar_{x\sas} "14";"16"};
{\ar_{y\sas} "16";"18"};
{\ar@{}|\circlearrowright "2";"14"};
{\ar@{}|\circlearrowright "4";"16"};
{\ar@{}|\circlearrowright "6";"18"};
\endxy
\]
in $\Mod\cat$, whose top row is exact by $\mathrm{(Q_{\mathit{n}})}$ and Lemma~\ref{LemDominant2}. Also $\EbbI^{-n}(-,A)\ov{x\sas}{\lra}\EbbI^{-n}(-,B)\ov{y\sas}{\lra}\EbbI^{-n}(-,C)$ is exact by $\mathrm{(Q_{\mathit{n}})}$ in the bottom row. Then we can show easily that the middle square becomes a weak pullback, and that $\EbbI^{-(n+1)}(-,C)\ov{\del\ssh}{\lra}\EbbI^{-n}(-,A)\ov{x\sas}{\lra}\EbbI^{-n}(-,B)$ becomes exact as desired.

It remains to show {\rm (2)}. Suppose that $\mathrm{(P_{\mathit{n}})}$ and $\mathrm{(Q_{\mathit{n}})}$ holds. Let $A\ov{x}{\lra}B\ov{y}{\lra}C\ov{\del}{\dra}$ be any $\sfr$-triangle. Take a dominant $\sfr$-triangle $F\ov{f}{\lra}G\ov{g}{\lra}B\ov{\vt}{\dra}$. By $\mathrm{(ET4)\op}$, we obtain a commutative diagram
\[
\xy
(-6,18)*+{F}="-12";
(6,18)*+{F}="-14";
(-6,6)*+{E}="2";
(6,6)*+{G}="4";
(18,6)*+{C}="6";
(30,6)*+{}="8";
(-6,-6)*+{A}="12";
(6,-6)*+{B}="14";
(18,-6)*+{C}="16";
(30,-6)*+{}="18";
(-6,-18)*+{}="22";
(6,-18)*+{}="24";
{\ar@{=} "-12";"-14"};
{\ar_{f\ppr} "-12";"2"};
{\ar^{f} "-14";"4"};
{\ar^{x\ppr} "2";"4"};
{\ar^{y\ppr} "4";"6"};
{\ar@{-->}^{\del\ppr} "6";"8"};
{\ar_{g\ppr} "2";"12"};
{\ar^{g} "4";"14"};
{\ar@{=} "6";"16"};
{\ar_{x} "12";"14"};
{\ar_{y} "14";"16"};
{\ar@{-->}_{\del} "16";"18"};
{\ar@{-->}_{x\uas\vt} "12";"22"};
{\ar@{-->}_{\vt} "14";"24"};
{\ar@{}|\circlearrowright "-12";"4"};
{\ar@{}|\circlearrowright "2";"14"};
{\ar@{}|\circlearrowright "4";"16"};
\endxy
\]
made of $\sfr$-triangles. This induces a commutative diagram
\[
\xy
(-15,16)*+{0}="2";
(15,16)*+{0}="4";
%
(-45,6)*+{\EbbI^{-(n+1)}(-,A)}="10";
(-15,6)*+{\EbbI^{-(n+1)}(-,B)}="12";
(15,6)*+{\EbbI^{-(n+1)}(-,C)}="14";
(43,6)*+{\EbbI^{-n}(-,A)}="16";
(-45,-6)*+{\EbbI^{-(n+1)}(-,A)}="20";
(-15,-6)*+{\EbbI^{-n}(-,F)}="22";
(15,-6)*+{\EbbI^{-n}(-,E)}="24";
(43,-6)*+{\EbbI^{-n}(-,A)}="26";
(-15,-18)*+{\EbbI^{-n}(-,G)}="32";
(15,-18)*+{\EbbI^{-n}(-,G)}="34";
{\ar "2";"12"};
{\ar "4";"14"};
%
{\ar^{x\sas} "10";"12"};
{\ar^{y\sas} "12";"14"};
{\ar^(0.56){\del\ssh} "14";"16"};
{\ar@{=} "10";"20"};
{\ar_{\ups\ssh} "12";"22"};
{\ar^{\del\ppr\ssh} "14";"24"};
{\ar@{=} "16";"26"};
{\ar_(0.52){(x\uas\vt)\ssh} "20";"22"};
{\ar_{f\ppr\sas} "22";"24"};
{\ar_{g\ppr\sas} "24";"26"};
{\ar_{f\sas} "22";"32"};
{\ar^{x\ppr\sas} "24";"34"};
{\ar@{=} "32";"34"};
{\ar@{}|\circlearrowright "10";"22"};
{\ar@{}|\circlearrowright "12";"24"};
{\ar@{}|\circlearrowright "14";"26"};
{\ar@{}|\circlearrowright "22";"34"};
\endxy
\]
in $\Mod\cat$. The lower row containing $(x\uas\vt)\ssh$ is exact by $\mathrm{(P_{\mathit{n}})}$ and $\mathrm{(Q_{\mathit{n}})}$. Two columns in the middle are exact by Lemma~\ref{LemDominant2}, since $g$ and hence $y\ppr$ are projective.
Thus the upper row also becomes exact, which shows $\mathrm{(Q}_{n+1}\mathrm{)}$.
\end{proof}

\begin{corollary}\label{CorNegCount1}
Assume that $\CEs$ has enough projective morphisms, and enough injective {\bf objects}. Then the following are equivalent.
\begin{enumerate}
\item $\EbbI^{-1}(I,-)=0$ holds for any injective object $I\in\cat$.
\item $\EbbI^{-n}(I,-)=0$ holds for any injective object $I\in\cat$ and any $n>0$.
\item $\EbbI^{-n}(i,-)=0$ holds for any injective inflation $i$ and any $n>0$.
\item $\cat(I,x)\colon\cat(I,A)\to\cat(I,B)$ is monomorphic for any injective object $I\in\cat$ and any inflation $A\ov{x}{\lra}B$.
\end{enumerate}
\end{corollary}
\begin{proof}
{\rm (1)} $\EQ$ {\rm (2)} follows from Proposition~\ref{PropVanishNegative}. 
{\rm (2)} $\Rightarrow$ {\rm (3)} is obvious, since any injective inflation factors through an injective object by assumption, as in Remark~\ref{RemProjMorph}.
Also {\rm (3)} $\Rightarrow$ {\rm (2)} is immediate, if we take $i=\id_I$ in {\rm (3)}.

{\rm (1)} $\Rightarrow$ {\rm (4)} follows from the exactness of $\EbbI^{-1}(I,C)\to\cat(I,A)\ov{x\circ-}{\lra}\cat(I,B)$ for an $\sfr$-triangle $A\ov{x}{\lra}B\to C\dra$ associated to $x$, shown in Theorem~\ref{ThmNegativeExtension}.

It remains to show {\rm (4)} $\Rightarrow$ {\rm (1)}.
For any $C\in\cat$, take a dominant $\sfr$-triangle $F\ov{f}{\lra}G\to C\dra$. Then $0\to\EbbI^{-1}(I,C)\to\cat(I,F)\ov{f\circ-}{\lra}\cat(I,G)$ is exact by the same theorem or Lemma~\ref{LemDominant2}. Thus the monomorphicity of $\cat(I,F)\ov{f\circ-}{\lra}\cat(I,G)$ implies $\EbbI^{-1}(I,C)=0$.
\end{proof}

Counterparts of Theorem~\ref{ThmNegativeExtension} and Corollary~\ref{CorNegCount1} for $\EbbII^{\bullet}$ become as follows.
\begin{theorem}\label{ThmCounter}
Assume that $\CEs$ has enough injective morphisms, defined dually as in Definition~\ref{DefProjDefl}. Then $(\EbbII^{\bullet},\del\ush)$ is a $\delta$-functor on $\CEs$.
\end{theorem}
\begin{corollary}\label{CorNegCount2}
Assume that $\CEs$ has enough injective morphisms, and enough projective {\bf objects}. Then the following are equivalent.
\begin{enumerate}
\item $\EbbII^{-1}(-,P)=0$ holds for any projective object $P\in\cat$.
\item $\EbbII^{-n}(-,P)=0$ holds for any projective object $P\in\cat$ and any $n>0$.
\item $\EbbII^{-n}(-,p)=0$ holds for any projective deflation $p$ and any $n>0$.
\item $\cat(y,P)\colon\cat(C,P)\to\cat(B,P)$ is monomorphic for any projective object $P\in\cat$ and any deflation $B\ov{y}{\lra}C$.
\end{enumerate}
\end{corollary}
\begin{proof}
These can be shown in a similar way as Theorem~\ref{ThmNegativeExtension} and Corollary~\ref{CorNegCount1}.
\end{proof}

Regarding Corollaries \ref{CorNegCount1} and \ref{CorNegCount2}, we define conditions {\rm (NI)} and {\rm (NII)} as follows.
\begin{condition}\label{CondNINII}
For $\CEs$, consider the following conditions.
\begin{itemize}
\item[{\rm (NI)}] $\cat(I,x)\colon\cat(I,A)\to\cat(I,B)$ is monomorphic for any injective object $I\in\cat$ and any inflation $A\ov{x}{\lra}B$.
\item[{\rm (NII)}] $\cat(y,P)\colon\cat(C,P)\to\cat(B,P)$ is monomorphic for any projective object $P\in\cat$ and any deflation $B\ov{y}{\lra}C$.
\end{itemize}
\end{condition}

\begin{remark}
Remark that conditions {\rm (NI)}, {\rm (NII)} are trivially satisfied if $\CEs$ corresponds to an exact category or a triangulated category.
\end{remark}

Let us conclude this subsection with the following immediate corollary of Theorem~\ref{ThmNegativeExtension}, which relates the (alternating sums of) dimensions of $\EbbI^{-n}(X,Y)$ and $\und{\cat}(X,\Om^nY)$.
\begin{corollary}\label{CorRelatesStable}
Assume that $R=K$ is a field and $\cat$ is $\Hom$-finite over $K$, and that $\CEs$ has enough projective morphisms.
Suppose that $Y\in\cat$ admits a complex in $\cat$
\begin{equation}\label{ProjResolY}
P_n\ov{d_n}{\lra} P_{n-1}\ov{d_{n-1}}{\lra}\cdots \ov{d_1}{\lra} P_0\ov{d_0}{\lra} Y
\end{equation}
which satisfies the following conditions.
\begin{itemize}
\item $P_k$ is a projective object for any $0\le k\le n$.
\item There exist conflations
\[ \Om^kY\ov{m_k}{\lra} P_{k-1}\ov{p_{k-1}}{\lra} \Om^{k-1}Y \]
for $1\le k\le n$, such that $\Om^0Y=Y$, $\Om^nY=P_n$, and $d_k=m_k\circ p_k$ $(1\le k\le n-1)$, $d_0=p_0$, $d_n=m_n$. 
\end{itemize}
Also, we put $\Om^kY=0$ for $k>n$. Then the equation
\[ \sum_{k\ge0}(-1)^k\dim_K\EbbI^{-k}(X,Y)=\sum_{k\ge0}(-1)^k\dim_K\und{\cat}(X,\Om^kY)+\sum_{0\le k\le n}(-1)^k\dim_K\cat(X,P_k) \]
holds for any $X\in\cat$.
\end{corollary}
\begin{proof}
This follows from the exactness of
\[ 0\to\EbbI^{-1}(X,\Om^kY)\to\cat(X,\Om^{k+1}Y)\to\cat(X,P_k)\to\cat(X,\Om^kY)\to\und{\cat}(X,\Om^kY)\to0 \]
and the existence of isomorphisms
\[ \EbbI^{-1}(X,\Om^kY)\cong\EbbI^{-2}(X,\Om^{k-1}Y)\cong\cdots\cong\EbbI^{-(k+1)}(X,Y) \]
for any $0\le k\le n$, both of which are the consequences of Theorem~\ref{ThmNegativeExtension}. Remark also that we have $\EbbI^{-k}(X,Y)=0$ and $\und{\cat}(X,\Om^kY)=0$ for any $k\ge n+1$.
\end{proof}

\begin{remark}
If $\CEs$ has enough projective objects and if $\Ebb^{n+1}=0$, then any $Y\in\cat$ admits a sequence $(\ref{ProjResolY})$ with the required conditions.
\end{remark}

\subsection{Universality of $\EbbI^{\mr}$ and $\EbbII^{\mr}$}\label{Subsection_Uni}

Let $\CEs$ be as in the previous subsection.
A question one might have will be that whether the definition of $\EbbI^{\mr}$  (respectively $\EbbII^{\mr}$) is natural enough. In this subsection, we answer it with Proposition~\ref{PropUniversality} by showing that it is universal among covariant (resp. contravariant) connected sequence of functors $\Fm$ having $F^n=\Ebb^n$ with the canonical connecting morphisms for $n\ge0$, i.e. these are covariant (resp. contravariant) derived functors of the bifunctor $\Hom$, under the assumption of the existence of enough projective morphisms (resp. enough injective morphisms). 

\begin{lemma}\label{LemUniversality}
Assume that $\CEs$ has enough projective morphisms. Let $(\Fm,\del\ssh)$ and $(G^{\bullet},\del\ssh)$ be two covariant connected sequences of functors.
Suppose that they satisfy $F^n=G^n$ and $\del^n_{F,\sharp}=\del^n_{G,\sharp}$ for any $n\in\Nbb_{\ge0}$ and any $\del\in\Ebb(C,A)$.
If $(G^{\bullet},\del\ssh)$ satisfies the condition
\begin{itemize}
\item[$(\star)$] $(G^{\bullet},\del\ssh)$ is a $\del$-functor, and $G^{-n}(-,p)=0$ holds for any projective deflation $p$ and any $n>0$,
\end{itemize}
then there exists a unique \emph{morphism} $\vp^{\mr}\colon (F^{\mr},\del\ssh)\to(G^{\mr},\del\ssh)$ \emph{of covariant connected sequence of functors}, given by identities in non-negative degrees. That is, a sequence of natural transformations $\{\vp^n\colon F^n\to G^n\}_{n\in\Zbb}$ which satisfies the following conditions.
\begin{itemize}
\item[{\rm (i)}] $\vp^n=\id$ for any $n\ge0$.
\item[{\rm (ii)}] For any $n\in\Zbb$, the following diagram $\Dbf(n;\del)$
\[
\xy
(-12,7)*+{F^{n}(-,C)}="0";
(12,7)*+{F^{n+1}(-,A)}="2";
(-12,-7)*+{G^{n}(-,C)}="4";
(12,-7)*+{G^{n+1}(-,A)}="6";
{\ar^{\del\ssh} "0";"2"};
{\ar_{\vp^{n}_{-,C}} "0";"4"};
{\ar^{\vp^{n+1}_{-,A}} "2";"6"};
{\ar_{\del\ssh} "4";"6"};
{\ar@{}|\circlearrowright "0";"6"};
\endxy
\]
is commutative in $\Mod\cat$ for any $\del\in\Ebb(C,A)$.
\end{itemize}
\end{lemma}
\begin{proof}
By {\rm (i)}, it suffices to construct $\vp^n\colon F^{n}\to G^n$ for any $n\le -1$ by a descending induction on $n$, in order that they satisfy {\rm (ii)}. Assume $n\le-1$, and suppose that we have obtained $\vp^k$ for all $k\ge n+1$.
It suffices to show the following.
\begin{claim}\label{ClaimInd}
Suppose that we have $\{\vp^k\}_{k\ge n+1}$ satisfying {\rm (ii)} for any $k\ge n+1$. Then the following holds.
\begin{enumerate}
\item For any $C\in\cat$, choose any dominant $\sfr$-triangle
\begin{equation}\label{sDeflChosen}
Q\ov{q}{\lra}P\ov{p}{\lra}C\ov{\om}{\dra}.
\end{equation}
Then there exists a unique morphism $\vp^n_{-,C}\colon F^n(-,C)\to G^n(-,C)$ in $\Mod\cat$ such that $\Dbf(n;\om)$ becomes commutative with respect to the pre-existing $\vp^{n+1}$.
\item The morphisms $\vp^n_{-,C}\colon F^n(-,C)\to G^n(-,C)$ obtained in {\rm (1)} form a natural transformation $\vp^n\colon F^n\to G^n$, and is independent of the choices of $(\ref{sDeflChosen})$.
\item The natural transformation $\vp^n$ makes $\Dbf(n;\del)$ commutative for any $\del\in\Ebb(C,A)$.
\end{enumerate}
\end{claim}
Indeed it is obvious that the claim makes the induction work. It remains to show Claim~\ref{ClaimInd}.
\begin{proof}[Proof of Claim~\ref{ClaimInd}]
{\rm (1)} This is obvious from the existence of the following commutative diagram
\[
\xy
(-16,7)*+{F^{n}(-,C)}="2";
(12,7)*+{F^{n+1}(-,Q)}="4";
(44,7)*+{F^{n+1}(-,P)}="6";
(-34,-7)*+{0}="10";
(-16,-7)*+{G^{n}(-,C)}="12";
(12,-7)*+{G^{n+1}(-,Q)}="14";
(44,-7)*+{G^{n+1}(-,P)}="16";
{\ar^{\om\ssh} "2";"4"};
{\ar^{F^{n+1}(-,p)} "4";"6"};
{\ar_{\vp^{n+1}_{-,Q}} "4";"14"};
{\ar^{\vp^{n+1}_{-,P}} "6";"16"};
{\ar_{} "10";"12"};
{\ar_{\om\ssh} "12";"14"};
{\ar_{G^{n+1}(-,p)} "14";"16"};
{\ar@{}|\circlearrowright "4";"16"};
\endxy
\]
in which the top row is a complex, and the bottom row is an exact sequence by assumption.

{\rm (2)} Let $c\in\cat(C,C\ppr)$ be any morphism, and let $Q\ov{q}{\lra}P\ov{p}{\lra}C\ov{\om}{\dra}$ and $Q\ppr\ov{q\ppr}{\lra}P\ppr\ov{p\ppr}{\lra}C\ppr\ov{\om\ppr}{\dra}$ be any pair of dominant $\sfr$-triangles. The morphisms $\vp^n_{-,C}$ and $\vp^n_{-,C\ppr}$ obtained in {\rm (1)} makes $\Dbf(n;\om)$ and $\Dbf(n;\om\ppr)$ commutative respectively, by construction.
Since $\om$ is dominant, there exists $a\in\cat(Q,Q\ppr)$ such that $a\sas\om=c\uas\om\ppr$. Since $\vp^{n+1}$ is natural by the hypothesis, we have a commutative diagram 
\[
\xy
(-16,7)*+{F^{n+1}(-,Q)}="0";
(16,7)*+{F^{n+1}(-,Q\ppr)}="2";
(-16,-7)*+{G^{n+1}(-,Q)}="4";
(16,-7)*+{G^{n+1}(-,Q\ppr)}="6";
{\ar^{F^{n+1}(-,a)} "0";"2"};
{\ar_{\vp^{n+1}_{-,Q}} "0";"4"};
{\ar^{\vp^{n+1}_{-,Q\ppr}} "2";"6"};
{\ar_{G^{n+1}(-,a)} "4";"6"};
{\ar@{}|\circlearrowright "0";"6"};
\endxy
\]
in $\Mod\cat$. Commutativity of these diagrams and the monomorphicity of $\om\ppr\ssh\colon G^n(-,C\ppr)\to G^{n+1}(-,Q\ppr)$ shows that
\[
\xy
(-14,7)*+{F^{n}(-,C)}="0";
(14,7)*+{F^{n}(-,C\ppr)}="2";
(-14,-7)*+{G^{n}(-,C)}="4";
(14,-7)*+{G^{n}(-,C\ppr)}="6";
{\ar^{F^n(-,c)} "0";"2"};
{\ar_{\vp^{n}_{-,C}} "0";"4"};
{\ar^{\vp^{n}_{-,C\ppr}} "2";"6"};
{\ar_{G^n(-,c)} "4";"6"};
{\ar@{}|\circlearrowright "0";"6"};
\endxy
\]
is commutative. This means that $\vp^{n}_{-,C}$ is natural in $C\in\cat$, hence we obtain a natural transformation $\vp^n$. Moreover, if we apply the above argument to $c=\id_C$, it shows the uniqueness of $\vp^n$.

{\rm (3)} Let $\del\in\Ebb(C,A)$ be any element. Take a dominant extension $\om\in\Ebb(C,Q)$. Then there is $a\in\cat(Q,A)$ such that $a\sas\om=\del$. 
then we have the following commutative diagram
\[
\xy
(-30,7)*+{F^{n}(-,C)}="0";
(2,7)*+{F^{n+1}(-,Q)}="2";
(-2,16)*+{}="3";
(30,7)*+{F^{n+1}(-,A)}="4";
(-30,-7)*+{G^{n}(-,C)}="10";
(2,-7)*+{G^{n+1}(-,Q)}="12";
(-2,-16)*+{}="13";
(30,-7)*+{G^{n+1}(-,A)}="14";
{\ar@/^1.80pc/^{\del\ssh} "0";"4"};
{\ar^{\om\ssh} "0";"2"};
{\ar^{F^{n+1}(-,a)} "2";"4"};
{\ar_{\vp^n_{-,C}} "0";"10"};
{\ar^{\vp^{n+1}_{-,Q}} "2";"12"};
{\ar^{\vp^{n+1}_{-,A}} "4";"14"};
{\ar_{\om\ssh} "10";"12"};
{\ar_{G^{n+1}(-,a)} "12";"14"};
{\ar@/_1.80pc/_{\del\ssh} "10";"14"};
{\ar@{}|\circlearrowright "0";"12"};
{\ar@{}|\circlearrowright "2";"14"};
{\ar@{}|\circlearrowright "2";"3"};
{\ar@{}|\circlearrowright "12";"13"};
\endxy
\]
in which the left square is $\Dbf(n;\om)$, that is commutative by the construction. Thus the outer square $\Dbf(n;\del)$ also becomes commutative.
\end{proof}
\end{proof}

\begin{remark}
The above condition $(\star)$ is satisfied by the following examples we have seen so far.
\begin{itemize}
\item The covariant part of Example~\ref{ExAcyclic} {\rm (1)} (exact categories).
\item The covariant part of Example~\ref{ExAcyclic} {\rm (2)} when $\cat=\Tcal$ (triangulated categories).
\item $(\EbbI^{\bullet},\del\ssh)$ in Definition~\ref{DefCdualProlong}.
\item If $\CEs$ has enough projective objects, the one in Example~\ref{ExINP} {\rm (1)} (stable Hom).
\end{itemize}
\end{remark}

\begin{proposition}\label{PropUniversality}
Assume that $\CEs$ has enough projective morphisms.
Then $(\EbbI^{\mr},\del\ssh)$ is universal among covariant connected sequence of functors having $F^n=\Ebb^n$ with the canonical connecting morphisms for $n\ge0$.
\end{proposition}
\begin{proof}
This is immediate from Theorem~\ref{ThmNegativeExtension} and Lemma~\ref{LemUniversality}.
\end{proof}

\begin{corollary}
Assume that $\CEs$ has enough projective morphisms. Then $\EbbI^{-n}$ is the $n$-th left covariant derived functor of the bifunctor $\Hom$.
\end{corollary}

\begin{corollary}\label{CorUniqueness}
Assume that $\CEs$ has enough projective morphisms.
For any covariant $\delta$-functor $(\Fm,\del\ssh)$ having $F^n=\Ebb^n$ with the canonical connecting morphisms for $n\ge0$, the following are equivalent.
\begin{enumerate}
\item $(\Fm,\del\ssh)$ satisfies $(\star)$.
\item $(\Fm,\del\ssh)$ is universal among covariant connected sequence of functors having $F^n=\Ebb^n$ with the canonical connecting morphisms for $n\ge0$.
\item $(\Fm,\del\ssh)$ is isomorphic to $(\EbbI^{\mr},\del\ssh)$ as a covariant $\delta$-functor.
\end{enumerate}
\end{corollary}
\begin{proof}
This is immediate from Lemma~\ref{LemUniversality} and Proposition~\ref{PropUniversality}.
\end{proof}

Its counterpart for contravariant ones becomes as follows.
\begin{corollary}\label{CorUniquenessCounter}
Assume that $\CEs$ has enough injective morphisms. Then $\EbbII^{-n}$ is the $n$-th left contravariant derived functor of the bifunctor $\Hom$.
For any contravariant 
$\delta$-functor $(\Fm,\del\ush)$ having $F^n=\Ebb^n$ with the canonical connecting morphisms for $n\ge0$, the following are equivalent.
\begin{enumerate}
\item $(\Fm,\del\ush)$ satisfies $F^n(i,-)=0$ for any injective inflation $i$ and any $n<0$ .
\item $(\Fm,\del\ush)$ is universal among contravariant connected sequence of functors having $F^n=\Ebb^n$ with the canonical connecting morphisms for $n\ge0$.
\item $(\Fm,\del\ush)$ is isomorphic to $(\EbbII^{\mr},\del\ush)$ as a contravariant $\delta$-functor.
\end{enumerate}
\end{corollary}

The above results show that we should think of the partial $\delta$-functors $(\EbbI^{\mr},\del\ssh)$ and $(\EbbII^{\mr},\del\ush)$ as of the (collections of) derived functors of the covariant, resp. contravariant functor $\Hom$ with one argument active and the other inert. In other words, $(\EbbI^{-n}), n > 0$ and $(\EbbII^{-n}), n > 0$ should be seen as negative extensions in the extriangulated category $\CEs$. 

In general, thus defined negative extensions are not balanced: given an arbitrary extriangulated category, we often do not have $\EbbI^{-n} \cong \EbbII^{-n}$ for all $n > 0$.

As Remark~\ref{RemFAcyclicInduce} suggests, on  $\CEs$ there can exist various (non-canonical or non-intrinsic) balanced $\delta$-functors having $F^n=\Ebb^n$ with the canonical connecting morphisms for $n\ge0$. Indeed, there are plenty of such examples, e.g., any extension-closed subcategory of a triangulated category $\Tcal$ can be equipped with such a $\delta$-functor by restrictions of that for $\Tcal$. It may depend on the embedding (see Example~\ref{ExExtClosed}) and in general is not universal on $\CEs$.

In the case where $\CEs$ has enough projective morphisms and enough injective morphisms, since we do have universal covariant $(\EbbI^{\mr},\del\ssh)$ and contravariant $(\EbbII^{\mr},\del\ush)$, it is natural to look for precise conditions on $\CEs$ forcing them to agree and constitute a balanced universal $\delta$-functor $(\EbbI^{\mr}\cong\EbbII^{\mr},\del\ssh,\del\ush)$.
We will investigate such conditions in the next subsection.

\begin{remark}\label{RemInsteadNegative}
As the above argument suggests, we may also harmlessly define functors $\EbbI^{-n}$ for $n>0$ equipped with $\del\ssh$ for any essentially small extriangulated category $\CEs$ with enough projective morphisms, inductively by
\[ \EbbI^{-n}(-,C):=\Ker\big(\Ebb^{-(n-1)}(-,F)\ov{f\sas}{\lra}\Ebb^{-(n-1)}(-,G)\big) \]
by choosing dominant $\sfr$-triangle $F\ov{f}{\lra}G\ov{g}{\lra}C\ov{\vt}{\dra}$ for each $C\in\cat$. The usual argument (similar to the above Lemma~\ref{LemUniversality}) shows that this is well-defined up to unique natural isomorphisms, independently of the choices of dominant $\sfr$-triangles.
Dually for $\EbbII^{\mr}$, in the case with enough injective morphisms. This again can be seen as a definition of (left) satellites of the functor $\Hom$, see \cite{CE, M, MR} and references therein.
\end{remark}

We end this subsection with the following example.
\begin{example}\label{ExExtClosed}
Let $K$ be an algebraically closed field, and let $Q$ be the quiver $1\leftarrow2\leftarrow3$ of type $A_3$. For each $m>0$, put
\[ \cat_{(m)}=\add(1[1]\oplus \bsm3\\2\esm[-m])\se D^b(\mod KQ)=\Tcal. \]
Then $\cat_{(m)}\se\Tcal$ is extension-closed, hence has an induced extriangulated structure $(\cat_{(m)},\Ebb_{(m)},\sfr_{(m)})$. Indeed, we see easily that it splits, i.e. $\Ebb_{(m)}=0$ holds. In particular, those $\cat_{(m)}$ for $m>0$ are all equivalent as extriangulated categories.

For each $m>0$, there are two natural covariant connected sequence of functors of $\Ebb_{(m)}$ given by $(\Ebb_{(m)})_{\Irm}^{-n}$ and $\Tcal(-[n],-)|_{\cat_{(m)}}$. Obviously we have $(\Ebb_{(m)})_{\Irm}^{-n}=0$ for any $n>0$ since $\Ebb_{(m)}=0$. On the other hand, we have
\[ \Tcal((\bsm3\\2\esm [-m])[n],1[1])\begin{cases}\ne0 & n=m\\ =0 & n\ne m\end{cases}. \]

This example shows that
\begin{itemize}
\item[{\rm (i)}] For an extriangulated category $\CEs$ equipped with an embedding $\cat\hookrightarrow\Tcal$ as an extension-closed subcategory into a triangulated category, we do not necessarily have $\EbbI^{-n}\cong\Tcal(-[n],-)|_{\cat}$.
\item[{\rm (ii)}] Besides $\Tcal(-[n],-)|_{\cat}$ can differ, depending on the embedding $\cat\hookrightarrow\Tcal$.
\end{itemize}

A similar construction as above can be also performed to obtain non-splitting examples which suggest {\rm (i),(ii)}, for example by taking
\[ \cat_{(m)}=\add\big(\bsm3\\2\esm\oplus\{M[k]\mid M\in\ind KQ,k>0\}\big)\se D^b(\mod KQ) \]
for $m>0$ for the same quiver $Q$.
\end{example}

\subsection{Comparison of $\EbbI^{\mr}$ and $\EbbII^{\mr}$}\label{Subsection_CS}
Let us compare $\EbbI^{\mr}$ and $\EbbII^{\mr}$ under the existence of enough projective morphisms and enough injective morphisms.
By definition, we have $\EbbI^{-n}(-,f)=0$ for any $f\in\Pcal$ and any $n>0$, while we have $\EbbII^{-n}(f,-)=0$ for any $f\in\Ical$ and any $n>0$.
In this subsection, we will see that conversely these properties determine $\EbbI^{\mr}$ and $\EbbII^{\mr}$ to some extent. 

This property suggests us the following example. In fact,
it shows that there are indeed some cases where $\EbbI^{\mr}$ does not satisfy $\EbbI^{-n}(I,-)=0$ for injective object $I$, while $\EbbII^{-n}(I,-)=0$ always holds by definition.
\begin{example}\label{ExNonBivariant}
Let $K$ be an algebraically closed field, and let $Q$ be the quiver $1\leftarrow2\leftarrow3\leftarrow4$ of type $A_4$. Put
\[ \cat=\add(3[-1]\oplus 2\oplus \bsm4\\3\\2\esm\oplus\bsm4\\3\esm)\se D^b(\mod KQ)=\Tcal. \]
It is extension-closed, hence has an induced extriangulated structure $\CEs$. We see that $\CEs$ is hereditary, with enough projective objects and enough injective objects. Since $C=3[-1]$ is injective, we have $\EbbII^{-1}(C,-)=0$. On the other hand, if we put $A=\bsm4\\3\esm$, then
it has an $\sfr$-conflation $2\to\bsm4\\3\\2\esm\to A$ in which $\bsm4\\3\\2\esm$ is projective in $\cat$, hence
\[ 0\to\EbbI^{-1}(C,A)\to\cat(C,2)\to\cat(C,\bsm4\\3\\2\esm) \]
is exact by Lemma~\ref{LemDominant2}. Since $\cat(C,2)\cong\Ext_{KQ}^1(3,2)\ne0$ and $\cat(C,\bsm4\\3\\2\esm)\cong\Ext_{KQ}^1(3,\bsm4\\3\\2\esm)=0$, this shows $\EbbI^{-1}(C,A)\ne0$.
\end{example}

In view of Corollaries \ref{CorUniqueness} and \ref{CorUniquenessCounter}, if one expects for $\EbbI^{\mr}$ and $\EbbII^{\mr}$ to agree, an obviously necessary condition should be that $\EbbI^{-n}(i,-)=0$ for any injective inflation and $\EbbII^{-n}(-,p)=0$ for any projective inflation, for any $n>0$.
In the sequel, we will see that this condition is in fact essentially sufficient to give a numerical equation for their dimensions, while a slightly stronger condition ensures that they agree as functors. See Proposition~\ref{PropNegCount} and Theorem~\ref{ThmCoincidence} for detail.

\begin{lemma}\label{LemComparisonIandII}
Assume that $\CEs$ has enough projective morphisms and enough injective morphisms. Let $X,Y\in\cat$ be any pair of objects, and take arbitrary $\sfr$-triangles
\begin{equation}\label{sTri_Comparison} X\ov{i}{\lra}I\ov{j}{\lra}J\ov{\iota}{\dra}\ \ \text{and}\ \ Q\ov{q}{\lra}P\ov{p}{\lra}Y\ov{\om}{\dra}
\end{equation}
in which $i$ is an injective inflation and $p$ is a projective deflation.

Then the morphisms
\[ \iota\ush\circ \om\ssh\colon\EbbI^{-1}(X,Y)\ov{\om\ssh}{\lra}\cat(X,Q)\ov{\iota\ush}{\lra}\Ebb(J,Q) \]
and
\[ \om\ssh\circ \iota\ush\colon\EbbII^{-1}(X,Y)\ov{\iota\ush}{\lra}\cat(J,Y)\ov{\om\ssh}{\lra}\Ebb(J,Q), \]
have the same image in $\Ebb(J,Q)$. Equivalently, there is an isomorphism $\mu$ which makes the following diagram commutative.
\begin{equation}\label{mucomm}
\xy
(-26,6)*+{\EbbI^{-1}(X,Y)/(\om\ssh)^{-1}(\Ical(X,Q))}="0";
(26,6)*+{\EbbII^{-1}(X,Y)/(\iota\ush)^{-1}(\Pcal(J,Y))}="2";
(0,6)*+{}="3";
(0,-6)*+{\Ebb(J,Q)}="4";
{\ar_{\cong}^{\mu} "0";"2"};
{\ar_{\iota\ush\circ\om\ssh} "0";"4"};
{\ar^{\om\ssh\circ\iota\ush} "2";"4"};
{\ar@{}|\circlearrowright "3";"4"};
\endxy
\end{equation}
Here, we denoted the morphisms induced by $\iota\ush\circ\om\ssh$ and $\om\ssh\circ\iota\ush$ by the same symbols, and $(\om\ssh)^{-1}(\Ical(X,Q))$ denotes the preimage of $\Ical(X,Q)$ by the monomorphism $\om\ssh\colon\EbbI^{-1}(X,Y)\ov{\om\ssh}{\lra}\cat(X,Q)$. Similarly for $(\iota\ush)^{-1}(\Pcal(J,Y))$.
\end{lemma}
\begin{proof}
Since $\cat(I,Q)\ov{i\uas}{\lra}\cat(X,Q)\ov{\iota\ush}{\lra}\Ebb(J,Q)$ is exact, we have $\Ker\big(\iota\ush\colon\cat(X,Q)\to\Ebb(J,Q)\big)=\Ical(X,Q)$, and thus $\Ker\big(\iota\ush\circ\om\ssh\colon\EbbI^{-1}(X,Y)\to\Ebb(J,Q)\big)=(\om\ssh)\iv(\Ical(X,Q))$. Similarly we have $\Ker\big(\om\ssh\circ\iota\ush\colon\EbbII^{-1}(X,Y)\to\Ebb(J,Q)\big)=(\iota\ush)\iv(\Pcal(J,Y))$. Thus indeed it suffices to show the existence of isomorphism $\mu$ which makes $(\ref{mucomm})$ commutative.

By Theorems \ref{ThmNegativeExtension} and \ref{ThmCounter}, sequences
$0\to \EbbI^{-1}(X,Y)\ov{\om\ssh}{\lra}\cat(X,Q)\ov{\cat(X,q)}{\lra}\cat(X,P)$
and
$0\to \EbbII^{-1}(X,Y)\ov{\iota\ush}{\lra}\cat(J,Y)\ov{\cat(j,Y)}{\lra}\cat(I,Y)$ are exact.
Put $L_{X,Y}=\Ker\big(\cat(X,q)\big)=\om\ssh\big(\EbbI^{-1}(X,Y)\big)$ 
and $M_{X,Y}=\Ker\big(\cat(j,Y)\big)=\iota\ush\big(\EbbII^{-1}(X,Y)\big)$ 
for simplicity.
Since $\om\ssh\colon\EbbI^{-1}(X,Y)\to\cat(X,Q)$ and $\iota\ush\colon\EbbI^{-1}(X,Y)\to\cat(J,Y)$ are monomorphic, it suffices to give an isomorphism
\begin{equation}\label{IsomLM}
L_{X,Y}/\big(L_{X,Y}\cap \Ical(X,Q)\big)\ov{\cong}{\lra} M_{X,Y}/\big(M_{X,Y}\cap \Pcal(J,Y)\big)
\end{equation}
compatibly with $\iota\ush$ and $\om\ssh$.

For any $s\in L_{X,Y}$, we obtain a morphism of $\sfr$-triangles
\begin{equation}\label{Morph_s_t}
\xy
(-14,6)*+{X}="0";
(0,6)*+{I}="2";
(14,6)*+{J}="4";
(28,6)*+{}="6";
(-14,-6)*+{Q}="10";
(0,-6)*+{P}="12";
(14,-6)*+{Y}="14";
(28,-6)*+{}="16";
{\ar^{i} "0";"2"};
{\ar^{j} "2";"4"};
{\ar@{-->}^{\iota} "4";"6"};
{\ar_{s} "0";"10"};
{\ar_{0} "2";"12"};
{\ar^{t} "4";"14"};
{\ar_{q} "10";"12"};
{\ar_{p} "12";"14"};
{\ar@{-->}_(0.56){\om} "14";"16"};
{\ar@{}|\circlearrowright "0";"12"};
{\ar@{}|\circlearrowright "2";"14"};
\endxy
\end{equation}
for some $t\in M_{X,Y}$ by {\rm (ET3)}. Conversely, any $t\in M_{X,Y}$ gives such a morphism of $\sfr$-triangles $(\ref{Morph_s_t})$ for some $s\in L_{X,Y}$ by {\rm (ET3)$\op$}.
Moreover, in $(\ref{Morph_s_t})$ we have
\[
s\in \Ical(X,Q)\ \EQ\ s\ \text{factors through}\ i\ \EQ\ t\ \text{factors through}\ p\ \EQ\ t\in \Pcal(J,Y),
\]
hence obtain an isomorphism $L_{X,Y}/\big(L_{X,Y}\cap \Ical(X,Q)\big)\ov{\cong}{\lra} M_{X,Y}/\big(M_{X,Y}\cap \Pcal(J,Y)\big)$ as desired. By definition of a morphism of $\sfr$-triangles, any $s\in L_{X,Y}$ and $t\in M_{X,Y}$ corresponding through this isomorphism satisfies $\iota\ush(s)=\om\ssh(t)$.
\end{proof}

\begin{remark}\label{Rem_Rel_IY}
Let $\Tcal$ be a triangulated category and let $\Dcal\se\cat\se\Tcal$ be additive full subcategories. If $(\cat,\cat)$ is an $\Dcal$-mutation pair in the sense of \cite[Definition~2.5]{IY}, then by definition $\cat\se\Tcal$ is an extension-closed subcategory, hence extriangulated. In this case $\cat$ is Frobenius whose subcategory of projective-injectives agrees with $\Dcal$, as seen in \cite[Example~7.2]{NP1}, and the ideal quotient $\cat/\Dcal$ is known to be triangulated by \cite[Theorem~4.2]{IY}. In particular there is an isomorphism by shifting $(\cat/\Dcal)(X,\Om Y)\ov{\cong}{\lra}(\cat/\Dcal)(\Sig X,Y)$ for any $X,Y\in\cat$, where $\Sig$ (respectively $\Om$) denotes the positive (resp. negative) shift by $1$ in $\cat/\Dcal$. The above isomorphism $(\ref{IsomLM})$ can be viewed as a restriction of this isomorphism.
\end{remark}

\begin{proposition}\label{PropComparisonIandII}
Assume that $\CEs$ has enough projective morphisms and enough injective morphisms. Let $X,Y\in\cat$ be any pair of objects, and take arbitrary $\sfr$-triangles $(\ref{sTri_Comparison})$ 
in which $i$ is an injective inflation and $p$ is a projective deflation.
Then the cohomologies of the complexes
\[ C_{\Irm}^{\mr}\colon\quad \EbbI^{-1}(J,Y)\ov{j\uas}{\lra}\EbbI^{-1}(I,Y)\ov{i\uas}{\lra}\EbbI^{-1}(X,Y)\ov{\iota\ush\circ\om\ssh}{\lra}\Ebb(J,Q) \]
and
\[ C_{\IIrm}^{\mr}\colon\quad \EbbII^{-1}(X,Q)\ov{q\sas}{\lra}\EbbII^{-1}(X,P)\ov{p\sas}{\lra}\EbbII^{-1}(X,Y)\ov{\om\ssh\circ\iota\ush}{\lra}\Ebb(J,Q) \]
are isomorphic.
\end{proposition}

\begin{proof}
Put $\Vbf_{(\iota,\om)}=\Thh_{\om}(J)\cap\Thh^{\iota}(P)=\Ker\big(\Ebb(J,q)\big)\cap\Ker\big(\Ebb(j,Q)\big)$ in $\Ebb(J,Q)$, for simplicity.
Consider the following double complex $C^{\mr,\mr}$ induced from $\Ebb^n$ $(n\ge0)$, concentrated in the lower-right quadrant.
\[
\xy
(-40,30)*+{\cat(J,Q)}="0";
(-20,30)*+{\cat(J,P)}="1";
(0,30)*+{\cat(J,Y)}="2";
(20,30)*+{\Ebb(J,Q)}="3";
(40,30)*+{\Ebb(J,P)}="4";
(56,30)*+{\cdots}="5";
(-40,18)*+{\cat(I,Q)}="10";
(-20,18)*+{\cat(I,P)}="11";
(0,18)*+{\cat(I,Y)}="12";
(20,18)*+{\Ebb(I,Q)}="13";
(40,18)*+{\Ebb(I,P)}="14";
(56,18)*+{\cdots}="15";
(-40,6)*+{\cat(X,Q)}="20";
(-20,6)*+{\cat(X,P)}="21";
(0,6)*+{\cat(X,Y)}="22";
(20,6)*+{\Ebb(X,Q)}="23";
(40,6)*+{\Ebb(X,P)}="24";
(56,6)*+{\cdots}="25";
(-40,-6)*+{\Ebb(J,Q)}="30";
(-20,-6)*+{\Ebb(J,P)}="31";
(0,-6)*+{\Ebb(J,Y)}="32";
(20,-6)*+{\Ebb^2(J,Q)}="33";
(40,-6)*+{\vdots}="34";
(56,-6)*+{}="35";
(-40,-18)*+{\Ebb(I,Q)}="40";
(-20,-18)*+{\Ebb(I,P)}="41";
(0,-18)*+{\Ebb(I,Y)}="42";
(20,-18)*+{\vdots}="43";
(40,-18)*+{}="44";
(56,-18)*+{}="45";
(-40,-32)*+{\vdots}="50";
(-20,-32)*+{\vdots}="51";
(0,-32)*+{\vdots}="52";
{\ar^{q\sas} "0";"1"};
{\ar^{p\sas} "1";"2"};
{\ar^{\om\ssh} "2";"3"};
{\ar^{q\sas} "3";"4"};
{\ar "4";"5"};
{\ar_{j\uas} "0";"10"};
{\ar "1";"11"};
{\ar "2";"12"};
{\ar "3";"13"};
{\ar "4";"14"};
{\ar "10";"11"};
{\ar "11";"12"};
{\ar "12";"13"};
{\ar "13";"14"};
{\ar "14";"15"};
{\ar_{i\uas} "10";"20"};
{\ar "11";"21"};
{\ar "12";"22"};
{\ar "13";"23"};
{\ar "14";"24"};
{\ar "20";"21"};
{\ar "21";"22"};
{\ar "22";"23"};
{\ar "23";"24"};
{\ar "24";"25"};
{\ar_{\iota\ush} "20";"30"};
{\ar "21";"31"};
{\ar "22";"32"};
{\ar "23";"33"};
{\ar "24";"34"};
{\ar "30";"31"};
{\ar "31";"32"};
{\ar "32";"33"};
{\ar "33";"34"};
{\ar_{j\uas} "30";"40"};
{\ar "31";"41"};
{\ar "32";"42"};
{\ar "33";"43"};
{\ar "40";"41"};
{\ar "41";"42"};
{\ar "42";"43"};
{\ar "40";"50"};
{\ar "41";"51"};
{\ar "42";"52"};
\endxy
\]
If we take the kernels of the leftmost horizontal arrows, we obtain a complex
\[ C_{\Irm}^{\prime\mr}\colon\quad
\EbbI^{-1}(J,Y)\ov{j\uas}{\lra}\EbbI^{-1}(I,Y)\ov{i\uas}{\lra}\EbbI^{-1}(X,Y)\to\Ker\big(\Ebb(J,q)\big)\ov{a}{\lra}\Ker\big(\Ebb(I,q)\big)\to\cdots \]
by Lemma~\ref{LemDominant2} for $n=1$.
Similarly, if we take the kernels vertically, we obtain a complex
\[ C_{\mathrm{I\! I}}^{\prime\mr}\colon\quad
\EbbII^{-1}(X,Q)\ov{q\sas}{\lra}\EbbII^{-1}(X,P)\ov{p\sas}{\lra}\EbbII^{-1}(X,Y)\to\Ker\big(\Ebb(j,Q)\big)\ov{b}{\lra}\Ker\big(\Ebb(j,P)\big)\to\cdots. \]
Here, $a$ and $b$ are the unique morphisms given by the restriction of $\Ebb(j,Q)$ and $\Ebb(J,q)$, respectively. In particular, we have $\Ker a=\Ker b=\Vbf_{(\iota,\om)}$.
By the acyclic assembly lemma \cite[2.7.3]{Weibel}, we have quasi-isomorphisms
\begin{equation}\label{Qis_Seq}
C_{\Irm}^{\prime\mr}\ov{\text{qis.}}{\lra}\operatorname{Tot}(C^{\mr,\mr})\ov{\text{qis.}}{\lla}C_{\mathrm{I\! I}}^{\prime\mr}
\end{equation}
between $C_{\Irm}^{\prime\mr},C_{\mathrm{II}}^{\prime\mr}$ and the total complex $\operatorname{Tot}(C^{\mr,\mr})$ of $C^{\mr,\mr}$.
By truncating $(\ref{Qis_Seq})$ we obtain complexes
\[ C_{\Irm}^{\prime\prime\mr}\colon\quad \EbbI^{-1}(J,Y)\ov{j\uas}{\lra}\EbbI^{-1}(I,Y)\ov{i\uas}{\lra}\EbbI^{-1}(X,Y)\ov{u}{\lra}\Vbf_{(\iota,\om)} \]
and
\[ C_{\mathrm{I\! I}}^{\prime\prime\mr}\colon\quad \EbbII^{-1}(X,Q)\ov{q\sas}{\lra}\EbbII^{-1}(X,P)\ov{p\sas}{\lra}\EbbII^{-1}(X,Y)\ov{v}{\lra}\Vbf_{(\iota,\om)}, \]
which are connected by a pair of quasi-isomorphisms
$C_{\Irm}^{\prime\prime\mr}\ov{\text{qis.}}{\lra}D^{\mr}\ov{\text{qis.}}{\lla}C_{\mathrm{I\! I}}^{\prime\prime\mr}
$, with $D^{\mr}$ given by a truncation of $\operatorname{Tot}(C^{\mr,\mr})$. Thus, the assertion follows from Lemma~\ref{LemComparisonIandII}.
\end{proof}

\begin{corollary}\label{CorComparisonIandII}
Assume that $R=K$ is a field and $\cat$ is $\Hom$-finite over $K$, and that $\CEs$ has enough projective morphisms and enough injective morphisms. Suppose that the following conditions are satisfied.
\begin{itemize}
\item[{\rm (i)}] $\EbbI^{-1}(i,-)=0$ for any injective inflation.
\item[{\rm (ii)}] $\EbbII^{-1}(-,p)=0$ for any projective deflation.
\end{itemize}
Then $\dim_K\EbbI^{-1}(X,Y)=\dim_K\EbbII^{-1}(X,Y)$ and $\dim_K\EbbI^{-2}(X,Y)=\dim_K\EbbII^{-2}(X,Y)$ hold for any $X,Y\in\cat$.
\end{corollary}
\begin{proof}
Take any pair of $\sfr$-triangles $(\ref{sTri_Comparison})$ as in Lemma~\ref{LemComparisonIandII}.
By assumption, the resulting complexes $C_{\Irm}^{\mr},C_{\mathrm{I\! I}}^{\mr}$ split in the middle. Hence the following $2$-term complexes
\begin{equation}\label{2termIsoCoh1}
\EbbI^{-1}(X,Y)\ov{\iota\ush\circ\om\ssh}{\lra}\Ebb(J,Q)\quad \text{and}\quad \EbbII^{-1}(X,Y)\ov{\om\ssh\circ\iota\ush}{\lra}\Ebb(J,Q)
\end{equation}
have isomorphic cohomologies. Similarly, so do
\begin{equation}\label{2termIsoCoh2}
\EbbI^{-1}(J,Y)\ov{j\uas}{\lra}\EbbI^{-1}(I,Y)\quad \text{and}\quad \EbbII^{-1}(X,Q)\ov{q\sas}{\lra}\EbbII^{-1}(X,P).
\end{equation}
From $(\ref{2termIsoCoh1})$, we obtain $\dim_K\EbbI^{-1}(X,Y)=\dim_K\EbbII^{-1}(X,Y)$. 

Since $X,Y\in\cat$ were arbitrary, we may also apply this equality to $I,J$ and $P,Q$, to obtain
\begin{eqnarray*}
&&\dim_K\EbbI^{-1}(X,Q)-\dim_K\EbbI^{-1}(X,P)=\dim_K\EbbII^{-1}(X,Q)-\dim_K\EbbII^{-1}(X,P)\\
&&\hspace{-0.6cm}\un{(\ref{2termIsoCoh2})}{=}\dim_K\EbbI^{-1}(J,Y)-\dim_K\EbbI^{-1}(I,Y)
=\dim_K\EbbII^{-1}(J,Y)-\dim_K\EbbII^{-1}(I,Y)
\end{eqnarray*}
from $(\ref{2termIsoCoh2})$.
Since
$0\to\EbbI^{-2}(X,Y)\to\EbbI^{-1}(X,Q)\to\EbbI^{-1}(X,P)\to0$
and
$0\to\EbbII^{-2}(X,Y)\to\EbbII^{-1}(J,Y)\to\EbbII^{-1}(I,Y)\to0$
are exact by Theorems \ref{ThmNegativeExtension} and \ref{ThmCounter}, this means $\dim_K\EbbI^{-2}(X,Y)=\dim_K\EbbII^{-2}(X,Y)$.
\end{proof}

If we assume the existence of enough projective objects and injective objects, then we obtain equalities of dimensions in any degree. Moreover, such equalities are shown to be equivalent to {\rm (NI),(NII)} in Condition~\ref{CondNINII} as follows.
\begin{proposition}\label{PropNegCount}
Assume that $R=K$ is a field and $\cat$ is $\Hom$-finite over $K$, and that $\CEs$ has enough projective objects and enough injective objects. Then the following are equivalent.
\begin{enumerate}
\item $\CEs$ satisfies conditions {\rm (NI)} and {\rm (NII)}.
\item $\dim_K\EbbI^{-1}(X,Y)=\dim_K\EbbII^{-1}(X,Y)$ holds for any $X,Y\in\cat$.\item $\dim_K\EbbI^{-n}(X,Y)=\dim_K\EbbII^{-n}(X,Y)$ holds for any $X,Y\in\cat$ and any $n>0$.
\end{enumerate}
\end{proposition}
\begin{proof}
Remark that Corollaries~\ref{CorNegCount1} and \ref{CorNegCount2} can be applied, by assumption.
{\rm (1)} $\Rightarrow$ {\rm (2)} is shown in Corollary~\ref{CorComparisonIandII}. {\rm (2)} $\Rightarrow$ {\rm (1)} is obvious, 
since we always have $\EbbI^{-1}(-,P)=0$ for any projective object $P\in\cat$ and $\EbbII^{-1}(I,-)=0$ for any injective object $I\in\cat$ by definition. {\rm (3)} $\Rightarrow$ {\rm (2)} is trivial.

It remains to show that {\rm (2)} ($\EQ$ {\rm (1)}) implies {\rm (3)}. Assume that {\rm (2)} holds. We show
\[ \dim_K\EbbI^{-n}(X,Y)=\dim_K\EbbII^{-n}(X,Y)\quad(\fa X,Y\in\cat) \]
by an induction on $n\in\Nbb_{>0}$. By Corollary~\ref{CorComparisonIandII}, this holds for $n=1,2$. Assume $n\ge3$, and suppose that we have shown for all $1\le k\le n-1$. Let $X,Y\in\cat$ be arbitrary, and take $\sfr$-triangles
\[ X\to I\to X\ppr\ov{\iota}{\dra}\quad\text{and}\quad Y\ppr\to P\to Y\ov{\om}{\dra} \]
in which $P\in\cat$ is projective and $I\in\cat$ is injective. Then we have isomorphisms
\[ \om\ssh\colon\EbbI^{-k}(-,Y)\ov{\cong}{\lra}\EbbI^{-(k-1)}(-,Y\ppr)
\quad\text{and}\quad
\iota\ush\colon\EbbII^{-k}(X,-)\ov{\cong}{\lra}\EbbII^{-(k-1)}(X\ppr,-) \]
for any $k\ge 2$ by Theorems \ref{ThmNegativeExtension} and \ref{ThmCounter}. Thus we obtain
\begin{eqnarray*}
\dim_K\Ebb_I^{-n}(X,Y)&=&\dim_K\EbbI^{-(n-1)}(X,Y\ppr)\ =\ \dim_K\EbbII^{-(n-1)}(X,Y\ppr)\\
&=&\dim_K\EbbII^{-(n-2)}(X\ppr,Y\ppr)\ =\ \dim_K\EbbI^{-(n-2)}(X\ppr,Y\ppr)\\
&=&\dim_K\EbbI^{-(n-1)}(X\ppr,Y)\ =\ \dim_K\EbbII^{-(n-1)}(X\ppr,Y)\ =\ \dim_K\EbbII^{-n}(X,Y)
\end{eqnarray*}
by the hypothesis of the induction.
\end{proof}

Remark that the above {\rm (2),(3)} in Proposition~\ref{PropNegCount} are only numerical equations between $\EbbI$ and $\EbbII$. However, if we assume a slightly stronger condition than {\rm (NI),(NII)} on $\CEs$, we obtain isomorphisms $\EbbI^{-n}(X,Y)\cong\EbbII^{-n}(X,Y)$ as in Theorem~\ref{ThmCoincidence} below.

Regarding Lemma~\ref{LemComparisonIandII}, we define Conditions {\rm (NI$+$)} and {\rm (NII$+$)} as follows.
\begin{condition}
For $\CEs$, consider the following conditions.
\begin{itemize}
\item[{\rm (NI$+$)}] $(\om\ssh)^{-1}(\Ical(X,Q))=0$ holds in $\EbbI^{-1}(X,Y)$ for any $X,Y\in\cat$ and any dominant extension $\om\in\Ebb(Y,Q)$.
\item[{\rm (NII$+$)}] $(\iota\ush)^{-1}(\Pcal(J,Y))=0$ holds in $\EbbII^{-1}(X,Y)$ for any $X,Y\in\cat$ and any codominant extension $\iota\in\Ebb(J,X)$.
\end{itemize}
\end{condition}

\begin{remark} \label{NI+/NII+_ex_tri}
Similarly to {\rm (NI),(NII)}, Conditions {\rm (NI$+$),(NII$+$)} are also trivially satisfied if $\CEs$ corresponds to an exact category or a triangulated category.
\end{remark}

\begin{lemma} \label{weaker_NI+/NII+}
\begin{itemize}
    \item [(i)] Assume that the category $\CEs$ has enough injective objects. If we have $\cat(I, Q) = 0$ for all $I$ injective and all $Q$ appearing as the first terms in dominant $\mathfrak{s}$-triangles, then the condition {\rm (NI$+$)} holds in $\CEs$.
    \item [(ii)] Assume that the category $\CEs$ has enough projective objects. If we have $\cat(J, P) = 0$ for all $P$ projective and all $J$ appearing as the third terms in codominant $\mathfrak{s}$-triangles, then the condition {\rm (NII$+$)} holds in $\CEs$.
\end{itemize}
\end{lemma}

\begin{proof}
If $\CEs$ has enough injectives, the ideal $\Ical(X,Q)$ is formed precisely by the maps that factor through an injective object. The assumption in part (i) then implies that $\Ical(X,Q) = 0$ for any $X,Y\in\cat$ and any dominant extension $\om\in\Ebb(Y,Q)$. This implies the condition (NI$+$). Part (ii) is proved in a dual way.
\end{proof}

\begin{example}
Conditions {\rm (NI$+$),(NII$+$)} are satisfied for the category $K^{[-n, 0]}(\proj \Lambda)$ of complexes of length $(n + 1)$ of projectives $\Lambda$-modules up to homotopy, where $n \geq 0$ and $\Lambda$ is a finite-dimensional algebra. Indeed, the case $n = 0$ is covered by Remark \ref{NI+/NII+_ex_tri}, since the category is simply $\proj \Lambda$. If $n > 0$, since $0$ is a projective-injective object, it follows that the projective objects of $K^{[-n, 0]}(\proj \Lambda)$ are given precisely by complexes concentrated in degree $0$, and the injectives are given by complexes concentrated in degree $-n$. 

An object $Q$ appears as the first term of a dominant extension if and only if it is isomorphic to a complex in $K^b(\proj\Lambda)$ whose $-n$-th term is zero (and thus it has an inflation to $0$). But in $K^{[-n, 0]}(\proj \Lambda)$, there are no non-zero maps from complexes concentrated in degree $-n$ to complexes whose $-n$-th terms are zero. By \ref{weaker_NI+/NII+}(i), the condition (NI+) is satisfied. By dual arguments, the condition (NII+) is also satisfied.

\end{example}

\begin{lemma}\label{LemEquivCondNIpNIIp}
Assume that $\CEs$ has enough projective objects and enough injective objects. Then the following holds.
\begin{enumerate}
\item {\rm (NI$+$)} holds if and only if $\iota\ush\circ \om\ssh\colon\EbbI^{-1}(X,Y)
\to\Ebb(J,Q)$ is monomorphic for any $X,Y\in\cat$, any dominant $\om\in\Ebb(Y,Q)$ and any codominant $\iota\in\Ebb(J,X)$.
\item {\rm (NII$+$)} holds if and only if $\om\ssh\circ \iota\ush\colon\EbbII^{-1}(X,Y)
\to\Ebb(J,Q)$ is monomorphic for any $X,Y\in\cat$, any dominant $\om\in\Ebb(Y,Q)$ and any codominant $\iota\in\Ebb(J,X)$.
\item {\rm (NI$+$)} implies {\rm (NI)}.
\item {\rm (NII$+$)} implies {\rm (NII)}.
\item If $\CEs$ satisfies {\rm (NI)} and {\rm (NII)}, then {\rm (NI$+$)} holds if and only if {\rm (NII$+$)} holds.
\end{enumerate}
\end{lemma}
\begin{proof}
{\rm (1)} is obvious from $\Ker\big(\iota\ush\circ \om\ssh\big)=(\om\ssh)^{-1}(\Ical(X,Q))$. Similarly for {\rm (2)}.

{\rm (3)} In {\rm (NI$+$)}, if $X=I$ is an injective object, then we have
\[ \EbbI^{-1}(I,Y)=(\om\ssh)\iv(\Ical(I,\Om Y))=0 \]
since $\Ical(I,\Om Y)=\cat(I,\Om Y)$. Thus {\rm (NI$+$)} implies {\rm (NI)}.

{\rm (4)} Dually to {\rm (3)}, condition {\rm (NII)} is nothing but condition {\rm (NII$+$)} restricted to the case where $Y$ is projective.

{\rm (5)} Let $X,Y\in\cat$ be any pair of objects, and let
\begin{equation}\label{2term_again}
X\ov{i}{\lra}I\ov{j}{\lra}J\ov{\iota}{\dra}\ \ \text{and}\ \ Q\ov{q}{\lra}P\ov{p}{\lra}Y\ov{\om}{\dra}
\end{equation}
be any pair of $\sfr$-triangles in which $i$ is an injective inflation and $p$ is a projective deflation.
Under the assumption of {\rm (NI)} and {\rm (NII)}, the morphisms
\[ \EbbI^{-1}(X,Y)\ov{\iota\ush\circ\om\ssh}{\lra}\Ebb(J,Q)\quad \text{and}\quad \EbbII^{-1}(X,Y)\ov{\om\ssh\circ\iota\ush}{\lra}\Ebb(J,Q) \]
have isomorphic kernels by Proposition~\ref{PropComparisonIandII}, as seen in the proof of Corollary~\ref{CorComparisonIandII}. 
Thus {\rm (5)} follows from {\rm (1)} and {\rm (2)}.
\end{proof}

\begin{lemma}\label{LemCoincidence}
Assume that $\CEs$ has enough projective objects and enough injective objects.
Let $X\ov{i}{\lra}I\ov{j}{\lra}J\ov{\iota}{\dra}$ and $Q\ov{q}{\lra}P\ov{p}{\lra}Y\ov{\om}{\dra}$ be any pair of $\sfr$-triangles in which $I\in\cat$ is an injective object and $P\in\cat$ is a projective object.
The following holds.
\begin{enumerate}
\item Assume {\rm (NI$+$)} and {\rm (NII$+$)}. Then there is a unique isomorphism $\psi^{(-1)}_{X,Y}\colon\EbbI^{-1}(X,Y)\ov{\cong}{\lra}\EbbII^{-1}(X,Y)$ which makes
\[
\xy
(-16,8)*+{\EbbI^{-1}(X,Y)}="0";
(16,8)*+{\EbbII^{-1}(X,Y)}="2";
(0,-4)*+{\Ebb(J,Q)}="4";
(0,10)*+{}="3";
{\ar^{\psi^{(-1)}_{X,Y}}_{\cong} "0";"2"};
{\ar_(0.4){\iota\ush\circ \om\ssh} "0";"4"};
{\ar^(0.4){\om \ssh\circ \iota\ush} "2";"4"};
{\ar@{}|\circlearrowright "3";"4"};
\endxy
\]
commutative.
\item Assume {\rm (NI)} and {\rm (NII)}. Then there is a unique isomorphism $\al\colon\EbbI^{-1}(J,Y)\ov{\cong}{\lra}\EbbII^{-1}(X,Q)$ which makes
\[
\xy
(-16,8)*+{\EbbI^{-1}(J,Y)}="0";
(16,8)*+{\EbbII^{-1}(X,Q)}="2";
(0,-4)*+{\cat(J,Q)}="4";
(0,10)*+{}="3";
{\ar^{\al}_{\cong} "0";"2"};
{\ar_(0.4){\om\ssh} "0";"4"};
{\ar^(0.4){\iota\ush} "2";"4"};
{\ar@{}|\circlearrowright "3";"4"};
\endxy
\]
commutative.
\end{enumerate}
\end{lemma}
\begin{proof}
{\rm (1)} This is immediate from Lemma~\ref{LemComparisonIandII}.

{\rm (2)} Though this can be deduced from Proposition~\ref{PropComparisonIandII}, let us give a direct construction, for the sake of clarity.
In the commutative square
\[
\xy
(-14,7)*+{\cat(J,Q)}="0";
(14,7)*+{\cat(I,Q)}="2";
(-14,-7)*+{\cat(J,P)}="4";
(14,-7)*+{\cat(I,P)}="6";
{\ar^(0.52){\cat(j,Q)} "0";"2"};
{\ar_{\cat(J,q)} "0";"4"};
{\ar^{\cat(I,q)} "2";"6"};
{\ar_(0.52){\cat(j,P)} "4";"6"};
{\ar@{}|\circlearrowright "0";"6"};
\endxy
\]
morphisms $\cat(I,q)$ and $\cat(j,P)$ are monomorphic by Corollaries \ref{CorNegCount1} and \ref{CorNegCount2}. Since
\[ 0\to\EbbI^{-1}(J,Y)\ov{\om\ssh}{\lra}\cat(J,Q)\ov{\cat(J,q)}{\lra}\cat(J,P) \]
and 
\[ 0\to\EbbII^{-1}(X,Q)\ov{\iota\ush}{\lra}\cat(J,Q)\ov{\cat(j,Q)}{\lra}\cat(I,Q) \]
are exact by Theorems \ref{ThmNegativeExtension} and \ref{ThmCounter}, we obtain the desired isomorphism for kernels.
\end{proof}

\begin{theorem}\label{ThmCoincidence}
Assume that $\CEs$ has enough projective objects and enough injective objects.
If $\CEs$ satisfies {\rm (NI$+$)} and {\rm (NII+)}, then there exists a unique sequence of natural isomorphisms $\{\psi^{(n)}\colon\EbbI^n\ov{\cong}{\lra}\EbbII^n\}_{n\in\Zbb}$ which satisfies the following {\rm (i)} and {\rm (ii)}.
\begin{itemize}
\item[{\rm (i)}] $\psi^{(n)}=\id$ for any $n\ge0$.
\item[{\rm (ii)}] For any $n\in\Zbb$, the following diagram $\Dbf(n;\del,\rho)$
\[
\xy
(-42,8)*+{\EbbI^n(X,Y)}="0";
(-15,8)*+{\EbbI^{n+1}(X,Z)}="2";
(15,8)*+{\EbbII^{n+1}(X,Z)}="4";
(44,8)*+{\EbbII^{n+2}(W,Z)}="6";
(-42,-8)*+{\EbbII^n(X,Y)}="10";
(-15,-8)*+{\EbbII^{n+1}(W,Y)}="12";
(15,-8)*+{\EbbI^{n+1}(W,Y)}="14";
(44,-8)*+{\EbbI^{n+2}(W,Z)}="16";
{\ar^(0.46){\rho\ssh} "0";"2"};
{\ar^{\psi_{X,Z}^{(n+1)}}_{\cong} "2";"4"};
{\ar^{\del\ush} "4";"6"};
{\ar_{\psi^{(n)}_{X,Y}} "0";"10"};
{\ar_{\psi^{(n+2)}_{W,Z}}^{\cong} "16";"6"};
{\ar_(0.46){\del\ush} "10";"12"};
{\ar^{\psi_{W,Y}^{(n+1)}}_{\cong} "14";"12"};
{\ar_{\rho\ssh} "14";"16"};
{\ar@{}|\circlearrowright "0";"16"};
\endxy
\]
is commutative for any $\del\in\Ebb(W,X)$ and any $\rho\in\Ebb(Y,Z)$.
\end{itemize}
\end{theorem}
\begin{proof}
This can be shown in a similar way as Lemma~\ref{LemUniversality}.
As in {\rm (i)}, we already have $\psi^{(n)}$ for $n\ge0$, which obviously satisfies {\rm (ii)} for any $n\ge0$. Let us construct $\psi^{(-n)}$ for $n\in\Nbb_{\ge1}$ inductively. It suffices to show the following lemma.
\begin{lemma}\label{LemCoincideAsFunctor}
Suppose that we have obtained natural isomorphisms $\psi^{(k)}\colon\EbbI^k\ov{\cong}{\lra}\EbbII^k$ for all $k\ge -n+1$, such that {\rm (ii)} holds for any $k\ge -n+1$. Then the following holds.
\begin{enumerate}
\item For any $X,Y\in\cat$, choose $\sfr$-triangles
\begin{equation}\label{sTri_IP}
X\ov{i}{\lra}I\ov{j}{\lra}J\ov{\iota}{\dra}\ \ \text{and}\ \ Q\ov{q}{\lra}P\ov{p}{\lra}Y\ov{\om}{\dra}
\end{equation}
arbitrarily, so that $I\in\cat$ is an injective object and $P\in\cat$ is a projective object. Then there is a unique isomorphism $\psi^{(-n)}_{X,Y}\colon\EbbI^{-n}(X,Y)\ov{\cong}{\lra}\EbbII^{-n}(X,Y)$ which makes $\Dbf(-n;\iota,\om)$ commutative.
\item The isomorphisms $\psi^{(-n)}_{X,Y}$ obtained in {\rm (1)} form a natural isomorphism $\psi^{(-n)}\colon\EbbI^{-n}\ov{\cong}{\lra}\EbbII^{-n}$, which is independent of the choices of $(\ref{sTri_IP})$.
\item The natural isomorphism $\psi^{(-n)}$ obtained in {\rm (2)} makes $\Dbf(-n;\del,\rho)$ commutative for any $\del\in\Ebb(W,X)$ and any $\rho\in\Ebb(Y,Z)$, with respect to the pre-existing $\psi^{(-n+1)}$ and $\psi^{(-n+2)}$.
\end{enumerate}
\end{lemma}
Indeed, it is obvious that this lemma makes the induction work.
It remains to show Lemma~\ref{LemCoincideAsFunctor}.
\begin{proof}[Proof of Lemma~\ref{LemCoincideAsFunctor}]
{\rm (1)} Remark that $\sfr$-triangles of the form $(\ref{sTri_IP})$ exist by assumption. The top and the bottom rows of $\Dbf(-n;\iota,\om)$ are
\[ \EbbI^{-n}(X,Y)\ov{\om\ssh}{\lra}\EbbI^{-n+1}(X,Q)
\un{\cong}{\ov{\psi^{(-n+1)}_{X,Q}}{\lra}}\EbbII^{-n+1}(X,Q)
\ov{\iota\ush}{\lra}\EbbII^{-n+2}(J,Q) \]
and
\begin{equation}\label{ComposMonom}
\EbbII^{-n}(X,Y)\ov{\iota\ush}{\lra}\EbbII^{-n+1}(J,Y)
\un{\cong}{\ov{\psi^{(-n+1)}_{J,Y}}{\lla}}\EbbI^{-n+1}(J,Y)
\ov{\om\ssh}{\lra}\EbbI^{-n+2}(J,Q),
\end{equation}
respectively. Each of their compositions is
\begin{itemize}
\item an isomorphism if $n\ge3$ by Theorems \ref{ThmNegativeExtension} and \ref{ThmCounter},
\item a monomorphism if $n=2$ by the same theorems,
\item a monomorphism if $n=1$ by Lemma~\ref{LemEquivCondNIpNIIp}.
\end{itemize}
Therefore the isomorphism $\psi^{(-n)}_{X,Y}$ is unique in either case, and exists
\begin{itemize}
\item obviously as a composition of relevant isomorphisms if $n\ge3$,
\item by Lemma~\ref{LemCoincidence} {\rm (2)} if $n=2$, indeed we may take $\psi^{(-2)}_{X,Y}$ to be the composition of the isomorphisms
$
\EbbI^{-2}(X,Y)\un{\cong}{\ov{\om\ssh}{\lra}}
\EbbI^{-1}(X,Q)\un{\cong}{\ov{\psi^{(-1)}_{X,Q}}{\lra}}
\EbbII^{-1}(X,Q)\un{\cong}{\ov{\al}{\lla}}
\EbbI^{-1}(J,Y)\un{\cong}{\ov{\psi^{(-1)}_{J,Y}}{\lra}}
\EbbII^{-1}(J,Y)\un{\cong}{\ov{\iota\ush}{\lla}}\EbbII^{-2}(X,Y),
$
\item by Lemma~\ref{LemCoincidence} {\rm (1)} if $n=1$.
\end{itemize}
Thus {\rm (1)} is shown.

{\rm (2)} For $X,Y\in\cat$, let $(\ref{sTri_IP})$ and $\psi^{(-n)}_{X,Y}$ be as in {\rm (1)}. Let $x\in\cat(X,X\ppr)$ be any morphism and let $X\ppr\ov{i\ppr}{\lra}I\ppr\ov{j\ppr}{\lra}J\ppr\ov{\iota\ppr}{\dra}$ be any $\sfr$-triangle in which $I\ppr\in\cat$ is an injective object. By {\rm (1)}, there also exists a unique $\psi^{(-n)}_{X\ppr,Y}$ which makes $\Dbf(-n;\iota\ppr,\om)$ commutative.
We may obtain a morphism of $\sfr$-triangles as follows.
\[
\xy
(-14,6)*+{X}="0";
(0,6)*+{I}="2";
(14,6)*+{J}="4";
(28,6)*+{}="6";
(-14,-6)*+{X\ppr}="10";
(0,-6)*+{I\ppr}="12";
(14,-6)*+{J\ppr}="14";
(28,-6)*+{}="16";
{\ar^{i} "0";"2"};
{\ar^{j} "2";"4"};
{\ar@{-->}^{\iota} "4";"6"};
{\ar_{x} "0";"10"};
{\ar_{f} "2";"12"};
{\ar^{g} "4";"14"};
{\ar_{i\ppr} "10";"12"};
{\ar_{j\ppr} "12";"14"};
{\ar@{-->}_{\iota\ppr} "14";"16"};
{\ar@{}|\circlearrowright "0";"12"};
{\ar@{}|\circlearrowright "2";"14"};
\endxy
\]
This induces commutative diagrams
\[
\xy
(-43,6)*+{\EbbI^{-n}(X\ppr,Y)}="0";
(-15,6)*+{\EbbI^{-n+1}(X\ppr,Q)}="2";
(16,6)*+{\EbbII^{-n+1}(X\ppr,Q)}="4";
(46,6)*+{\EbbII^{-n+2}(J\ppr,Q)}="6";
(-43,-6)*+{\EbbI^{-n}(X,Y)}="10";
(-15,-6)*+{\EbbI^{-n+1}(X,Q)}="12";
(16,-6)*+{\EbbII^{-n+1}(X,Q)}="14";
(46,-6)*+{\EbbII^{-n+2}(J,Q)}="16";
{\ar^{\om\ssh} "0";"2"};
{\ar^{\psi^{(-n+1)}_{X\ppr,Q}}_{\cong} "2";"4"};
{\ar^{(\iota\ppr)\ush} "4";"6"};
{\ar_{x\uas} "0";"10"};
{\ar_{x\uas} "2";"12"};
{\ar^{x\uas} "4";"14"};
{\ar^{g\uas} "6";"16"};
{\ar_{\om\ssh} "10";"12"};
{\ar_{\psi^{(-n+1)}_{X,Q}}^{\cong} "12";"14"};
{\ar_{\iota\ush} "14";"16"};
{\ar@{}|\circlearrowright "0";"12"};
{\ar@{}|\circlearrowright "2";"14"};
{\ar@{}|\circlearrowright "4";"16"};
\endxy
\]
and
\[
\xy
(-43,6)*+{\EbbII^{-n}(X\ppr,Y)}="0";
(-15,6)*+{\EbbII^{-n+1}(J\ppr,Y)}="2";
(16,6)*+{\EbbI^{-n+1}(J\ppr,Y)}="4";
(46,6)*+{\EbbI^{-n+2}(J\ppr,Q)}="6";
(-43,-6)*+{\EbbII^{-n}(X,Y)}="10";
(-15,-6)*+{\EbbII^{-n+1}(J,Y)}="12";
(16,-6)*+{\EbbI^{-n+1}(J,Y)}="14";
(46,-6)*+{\EbbI^{-n+2}(J,Q)}="16";
{\ar^{(\iota\ppr)\ush} "0";"2"};
{\ar_{\psi^{(-n+1)}_{J\ppr,Y}}^{\cong} "4";"2"};
{\ar^{\om\ssh} "4";"6"};
{\ar_{x\uas} "0";"10"};
{\ar_{g\uas} "2";"12"};
{\ar^{g\uas} "4";"14"};
{\ar^{g\uas} "6";"16"};
{\ar_{\iota\ush} "10";"12"};
{\ar^{\psi^{(-n+1)}_{J,Y}}_{\cong} "14";"12"};
{\ar_{\om\ssh} "14";"16"};
{\ar@{}|\circlearrowright "0";"12"};
{\ar@{}|\circlearrowright "2";"14"};
{\ar@{}|\circlearrowright "4";"16"};
\endxy
\]
and also
\[
\xy
(-16,6)*+{\EbbII^{-n+2}(J\ppr,Q)}="0";
(16,6)*+{\EbbI^{-n+2}(J\ppr,Q)}="2";
(-16,-6)*+{\EbbII^{-n+2}(J,Q)}="4";
(16,-6)*+{\EbbI^{-n+2}(J,Q)}="6";
{\ar_{\psi^{(-n+2)}_{J\ppr,Q}}^{\cong} "2";"0"};
{\ar_{g\uas} "0";"4"};
{\ar^{g\uas} "2";"6"};
{\ar^{\psi^{(-n+2)}_{J,Q}}_{\cong} "6";"4"};
{\ar@{}|\circlearrowright "0";"6"};
\endxy
\]
by the naturality of $\psi^{(k)}$ for $k\ge-n+1$. Pasting them with $\Dbf(-n;\iota,\om)$ and $\Dbf(-n;\iota\ppr,\om)$, we obtain the commutativity of 
\[
\xy
(-15,6)*+{\EbbI^{-n}(X\ppr,Y)}="0";
(15,6)*+{\EbbII^{-n}(X\ppr,Y)}="2";
(-15,-6)*+{\EbbI^{-n}(X,Y)}="4";
(15,-6)*+{\EbbII^{-n}(X,Y)}="6";
{\ar^{\psi^{(-n)}_{X\ppr,Y}}_{\cong} "0";"2"};
{\ar_{x\uas} "0";"4"};
{\ar^{x\uas} "2";"6"};
{\ar_{\psi^{(-n)}_{X,Y}}^{\cong} "4";"6"};
{\ar@{}|\circlearrowright "0";"6"};
\endxy
\]
since the composition of $(\ref{ComposMonom})$ is monomorphic. This shows the naturality of $\psi^{(-n)}_{X,Y}$ in $X\in\cat$. Moreover if we apply the above argument to $x=\id_X$, this shows that $\psi^{(-n)}_{X,Y}$ is independent of the choices of $\sfr$-triangles $X\ov{i}{\lra}I\ov{i}{\lra}J\ov{\iota}{\dra}$. In a similar way we can also show the naturality in $Y\in\cat$ and that $\psi_{X,Y}$ is independent of the choices of $Q\ov{q}{\lra}P\ov{p}{\lra}Y\ov{\om}{\dra}$.

{\rm (3)} Let $\del\in\Ebb(W,X)$ and $\rho\in\Ebb(Y,Z)$ be any pair of $\Ebb$-extensions, and let $(\ref{sTri_IP})$ be arbitrary as in {\rm (1)}. Then there exist $w\in\cat(W,J)$ and $z\in\cat(Q,Z)$ such that $w\uas\iota=\del$ and $z\sas\om=\rho$. They give a commutative diagram
\[
\xy
(-42,18)*+{\EbbI^{-n+1}(X,Z)}="0";
(-46,2)*+{}="1";
(-14,18)*+{\EbbII^{-n+1}(X,Z)}="2";
(14,27)*+{}="3";
(14,18)*+{\EbbII^{-n+2}(J,Z)}="4";
(38,8)*+{}="5";
(42,18)*+{\EbbII^{-n+2}(W,Z)}="6";
(-60,6)*+{\EbbI^{-n}(X,Y)}="8";
(-32,6)*+{\EbbI^{-n+1}(X,Q)}="10";
(-4,6)*+{\EbbII^{-n+1}(X,Q)}="12";
(24,6)*+{\EbbII^{-n+2}(J,Q)}="14";
(-60,-6)*+{\EbbII^{-n}(X,Y)}="16";
(-32,-6)*+{\EbbII^{-n+1}(J,Y)}="18";
(-46,-2)*+{}="19";
(-4,-6)*+{\EbbI^{-n+1}(J,Y)}="20";
(24,-6)*+{\EbbI^{-n+2}(J,Q)}="22";
(-42,-18)*+{\EbbII^{-n+1}(W,Y)}="24";
(-14,-18)*+{\EbbI^{-n+1}(W,Y)}="26";
(14,-27)*+{}="27";
(14,-18)*+{\EbbI^{-n+2}(W,Q)}="28";
(38,-8)*+{}="29";
(42,-18)*+{\EbbI^{-n+2}(W,Z)}="30";
(34,0)*+{_{\circlearrowright}}="c";
{\ar^{\rho\ssh} "8";"0"};
{\ar^{\psi^{(-n+1)}_{X,Z}}_{\cong} "0";"2"};
{\ar^{\iota\ush} "2";"4"};
{\ar^{w\uas} "4";"6"};
{\ar_{\om\ssh} "8";"10"};
{\ar_{\cong}^{\psi^{(-n+1)}_{X,Q}} "10";"12"};
{\ar_{\iota\ush} "12";"14"};
{\ar_{w\uas z\sas} "14";"6"};
{\ar_{\psi^{(-n)}_{X,Y}}^{\cong} "8";"16"};
{\ar^{\iota\ush} "16";"18"};
{\ar^{\psi^{(-n+1)}_{J,Y}}_{\cong} "20";"18"};
{\ar^{\om\ssh} "20";"22"};
{\ar^{w\uas z\sas} "22";"30"};
{\ar_{\del\ush} "16";"24"};
{\ar^{\psi^{(-n+1)}_{W,Y}}_{\cong} "26";"24"};
{\ar_{\om\ssh} "26";"28"};
{\ar_{z\sas} "28";"30"};
{\ar^{\cong}_{\psi^{(-n+2)}_{W,Z}} "30";"6"};
{\ar_{\cong}^{\psi^{(-n+2)}_{J,Q}} "22";"14"};
{\ar_(0.4){z\sas} "10";"0"};
{\ar_(0.4){z\sas} "12";"2"};
{\ar^(0.64){z\sas} "14";"4"};
{\ar^(0.36){w\uas} "18";"24"};
{\ar_(0.6){w\uas} "20";"26"};
{\ar_(0.6){w\uas} "22";"28"};
{\ar@/^1.80pc/^{\del\ush} "2";"6"};
{\ar@/_1.80pc/_{\rho\ssh} "26";"30"};
{\ar@{}|\circlearrowright "0";"1"};
{\ar@{}|\circlearrowright "3";"4"};
{\ar@{}|\circlearrowright "4";"5"};
{\ar@{}|\circlearrowright "24";"19"};
{\ar@{}|\circlearrowright "27";"28"};
{\ar@{}|\circlearrowright "28";"29"};
{\ar@{}|\circlearrowright "0";"12"};
{\ar@{}|\circlearrowright "2";"14"};
{\ar@{}|\circlearrowright "8";"22"};
{\ar@{}|\circlearrowright "18";"26"};
{\ar@{}|\circlearrowright "20";"28"};
\endxy
\]
in which, the middle rectangle is nothing but $\Dbf(-n;\iota,\om)$ that is commutative by the construction. Thus we obtain the commutativity of $\Dbf(-n;\del,\rho)$, which appears as the outer rectangle.
\end{proof}
\end{proof}

\begin{corollary}\label{CorCoincideAsFunctor}
Assume that $\CEs$ has enough projective objects and enough injective objects.
If $\CEs$ satisfies {\rm (NI$+$)} and {\rm (NII+)}, then $\{\EbbI^n\cong\EbbII^n\}_{n\in\Zbb}$ gives 
a universal bivariant $\delta$-functor having $\Ebb^n$ with the canonical connecting morphisms for $n\ge0$. It is unique by Corollary~\ref{CorUniqueness}. More precisely, if we define
\begin{itemize}
\item $\del\ush\colon\EbbI^{-n}(A,-)\to\EbbI^{-n+1}(C,-)$ to be the composition of
\[ \EbbI^{-n}(A,-)\un{\cong}{\ov{\psi^{(-n)}_{A,-}}{\lra}}\EbbII^{-n}(A,-)\ov{\del\ush}{\lra}\EbbII^{-n+1}(C,-)\un{\cong}{\ov{\psi^{(-n+1)}_{C,-}}{\lla}}\EbbI^{-n+1}(C,-), \]
\item $\del\ssh\colon\EbbII^{-n}(-,C)\to\EbbII^{-n+1}(-,A)$ to be the composition of
\[ \EbbII^{-n}(-,C)\un{\cong}{\ov{\psi^{(-n)}_{-,C}}{\lla}}\EbbI^{-n}(-,C)\ov{\del\ssh}{\lra}\EbbI^{-n+1}(-,A)\un{\cong}{\ov{\psi^{(-n+1)}_{-,A}}{\lra}}\EbbII^{-n+1}(-,A) \]
\end{itemize}
for any $\del\in\Ebb(C,A)$ and any $n$, then both $(\EbbI^{\mr},\del\ssh,\del\ush)$ and $(\EbbII^{\mr},\del\ssh,\del\ush)$ are 
bivariant $\delta$-functors for which $\psi^{\mr}=\{\psi^{(n)}\colon\EbbI^n\ov{\cong}{\lra}\EbbII^n\}_{n\in\Zbb}$ gives an isomorphism.
\end{corollary}
\begin{proof}
This is immediate from Theorem~\ref{ThmCoincidence} and the construction.
\end{proof}

\subsection{Conjectures} \label{Subsection_Conj}

An extriangulated category is called \emph{topological} if it is equivalent to the homotopy category of an exact $\infty$-category, see \cite{NP2}. Klemenc \cite{K} constructed a \emph{stable hull} for each exact $\infty$-category. That is, he proved that each exact $\infty$-category embeds into a stable $\infty$-category, and there is an embedding satisfying a reasonable universal property. Since the homotopy category of a stable $\infty$-category is triangulated, this shows (using \cite[Prop 4.28]{NP2}) that for each topological extriangulated category $\CEs$, there is a triangulated category $\Tcal$ and an embedding $f\colon \CEs \hookrightarrow \Tcal$ as an extension-closed subcategory.
This in turn indicates that $\CEs$ has a balanced negative extensions given by $\Ebb^{-n}(A, B) := \Tcal(f(A)[n], f(B))$ which may have some kind of universality.
The construction of Klemenc uses quasi-categorical enhancements, and thus this morphism spaces are not easy to calculate. We make the following conjectures on the existence of balanced $\Ebb^{-n}$ and their relation to $\EbbI^{\mr}$ and $\EbbII^{\mr}$.

\begin{conjecture} \label{ConjBivariant}
Assume that $\CEs$ is a topological extriangulated category. Then it admits a bivariant $\delta$-functor $(\Ebb^{\mr}, \del\ssh,\del\ush)$ having $F^n=\Ebb^n$ with the canonical connecting morphisms for $n\ge0$ and universal among such bivariant $\delta$-functors. Moreover, we have 
$$\Ebb^{-n}(I, -) \cong \EbbI^{-n}(I, -); \quad \Ebb^{-n}(-, P) \cong \EbbII^{-n}(-, P),$$
for all $I$ injective, $P$ projective, $n > 0$.
\end{conjecture}

\begin{conjecture} \label{ConjNegGlobDim}
Assume that Conjecture \ref{ConjBivariant} holds for $\CEs$ and $\CEs$ has positive global dimension $n$. Then $\Ebb^{-m} = 0$ for all $m > 2n$.
\end{conjecture}

For extriangulated categories for which the statement of  Conjecture \ref{ConjNegGlobDim} holds, one might be able to define the \emph{negative global dimension}.

\begin{example}
\begin{enumerate}[(i)]
    \item The negative global dimension of an exact category is $0$.
    \item Each non-zero triangulated category has infinite positive and negative global dimensions.
    \item If $\CEs$ has positive global dimension $n$ and is such that $\EbbI^m \cong \EbbII^m$, for all $m$, then by definition, $\Ebb^{-m} = 0$ for all $m > n$. Thus, the negative global dimension  of $\CEs$ is $-n$.
    \item As a special case of the previous point, in the category $K^{[-n,0]}(\proj \Lambda)$ of $n$-term complexes with finitely generated projective components over a finite-dimensional algebra, we have $\EbbI^m \cong \EbbII^m$, for all $m$. Thus, the negative global dimension is precisely $-n$. 
\end{enumerate}
\end{example}

The reader may be surprised by the bound $2n$ instead of $n$. The next example explains its appearance in the conjecture.

\begin{example}
Let $A$ be a non-positively graded dg-algebra. Consider the full extension-closed subcategory $\cat := A \ast A[1] \ast A[2] \ast \cdots \ast A[n]$ of $\cper A$ with the induced extriangulated structure. This is an example of an extriangulated category of positive global dimension $n$. Its projective objects are given by $\mbox{add} A$ and its injective objects are given by $\mbox{add} A[n]$.
Straightforward calculations show that any bivariant $\delta$-functor $F^{\bullet}$ of having $F^n=\Ebb^n$ with the canonical connecting morphisms for $n\ge0$ has to satisfy

$$F^{-2n}(A, A[n]) = \Hom(A[n], A).$$

For an arbitrary $A$, this might be non-zero. This means the that the negative global dimension of $\cat$ can be $-2n$. 
\end{example}

Conjectures \ref{ConjBivariant} and \ref{ConjNegGlobDim} do not seem to follow easily from Klemenc's work \cite{K}. Note though that Conjecture \ref{ConjBivariant} clearly holds for all triangulated categories, including those that do not admit enhancements, i.e. non-topological. This suggests that there is a chance that one would be able to define universal bivariant (=balanced) negative extensions for all extriangulated categories.

\section{Projective resolutions for defects}\label{Subsection_PR}

The $\delta$-functor in Example~\ref{ExINP} (2) gives projective resolutions of defects 
\emph{in the category $\defff \Ebb$}. Note that it is not a projective resolution in $\coh(\Ebb)$.

In this section, we will show that $\CEs$ has enough projective morphisms if and only if the inclusion $\defff\Ebb\hookrightarrow\fp\cat$ has a left adjoint (Corollary~\ref{CorDeflAdj}). As a result, we will see that any object in $\defff \Ebb$ admits a natural projective resolution in $\defff\Ebb$ whenever $\CEs$ has enough projective morphisms (Corollary~\ref{CorProjResol}). We use the notations in Subsection~\ref{Subsection_PD}.

Let us recall the definition of \emph{reflection} from \cite[Definition~3.1.1]{Borceux}. For the inclusion functor $\iota\colon\defff\Ebb\hookrightarrow\fp\cat$, it becomes as follows.
\begin{definition}
Let $\iota\colon\defff\Ebb\hookrightarrow\fp\cat$ denote the inclusion.
For an object $\Fbf\in\fp\cat$, a {\it reflection of $\Fbf$ along $\iota$} is a pair $(\Gam_{\Fbf},\gam_{\Fbf})$ of $\Gam_{\Fbf}\in\defff\Ebb$ and $\gam_{\Fbf}\in(\fp\cat)(\Fbf,\Gam_{\Fbf})$ such that
\[ -\circ\gam_{\Fbf}\colon(\defff\Ebb)(\Gam_{\Fbf},\Thh)\to(\fp\cat)(\Fbf,\Thh) \]
is an isomorphism for any $\Thh\in\defff\Ebb$.
\end{definition}

\begin{proposition}\label{PropDefProj}
Let $C\in\cat$ be any object, and let $F\ov{f}{\lra}G\ov{g}{\lra}C\ov{\vt}{\dra}$ be any dominant $\sfr$-triangle. 
Then the following holds.
\begin{enumerate}
\item $\und{\wp}_{\vt}\colon\und{\cat}(-,C)\to\Thh_{\vt}$ is an isomorphism in $\fp\und{\cat}$. In particular we have $\und{\cat}(-,C)\in\defff\Ebb$.
\item $(\Thh_{\vt},\wp_{\vt})$ is a reflection of $\cat(-,C)$ along $\iota$.
\item $\Thh_{\vt}$ is projective in $\defff\Ebb$.
\end{enumerate}
\end{proposition}
\begin{proof}
{\rm (1)} This follows from $\Im\big(\cat(-,G)\ov{g\circ-}{\lra}\cat(-,C)\big)=\Pcal(-,C)$.

{\rm (2)} Let $\Thh\in\defff\Ebb$ be any object. In the commutative diagram
\[
\xy
(-20,8)*+{(\fp\cat)(\cat(-,C),\Thh)}="0";
(-20,-8)*+{(\fp\und{\cat})(\und{\cat}(-,C),\Thh)}="2";
(4,4)*+{}="3";
(24,-8)*+{(\defff\Ebb)(\Thh_{\vt},\Thh)}="4";
{\ar^{-\circ\und{(-)}} "2";"0"};
{\ar_(0.46){\cong}^(0.46){-\circ\und{\wp}_{\vt}} "4";"2"};
{\ar_(0.46){-\circ\wp_{\vt}} "4";"0"};
{\ar@{}|\circlearrowright "2";"3"};
\endxy
\]
the bottom row is an isomorphism by {\rm (1)}, while 
\begin{equation}\label{IsoQuot}
-\circ\und{(-)}\colon (\fp\und{\cat})(\und{\cat}(-,C),\Thh)\ov{\cong}{\lra}(\fp\cat)(\cat(-,C),\Thh)
\end{equation}
is an isomorphism since $\Thh|_{\Pcal}=0$. Thus it follows that
\[ -\circ\wp_{\vt}\colon(\defff\Ebb)(\Thh_{\vt},\Thh)\to(\fp\cat)(\cat(-,C),\Thh) \]
is also an isomorphism.

{\rm (3)} This is immediate from the isomorphism of functors
\[ (\fp\cat)(\cat(-,C),-)|_{\defff\Ebb}\cong(\defff\Ebb)(\Thh_{\vt},-)\colon \defff\Ebb\to\Mod R \]
obtained in {\rm (2)}, and the projectivity of $\cat(-,C)$ in $\fp\cat$.

\end{proof}

\begin{corollary}\label{CorDomDef}
For any extension $\del\in\Ebb(C,A)$, the following are equivalent.
\begin{enumerate}
\item $\del$ is dominant.
\item $(\Thh_{\del},\wp_{\del})$ is a reflection of $\cat(-,C)$ along $\iota$.
\end{enumerate}
\end{corollary}
\begin{proof}
{\rm (1)} $\Rightarrow$ {\rm (2)} is shown in Proposition~\ref{PropDefProj} {\rm (2)}.

{\rm (2)} $\Rightarrow$ {\rm (1)} Realize $\del$ to obtain an $\sfr$-triangle 
$A\ov{x}{\lra}B\ov{y}{\lra}C\ov{\del}{\dra}$. 
Let $\rho\in\Ebb(C,X)$ be any extension. If {\rm (2)} holds, then there exists $\eta\in(\defff\Ebb)(\Thh_{\del},\Thh_{\rho})$ such that $\eta\circ\wp_{\del}=\wp_{\rho}$. Then $y\uas\del=0$ implies $y\uas\rho=0$, hence $y$ is a projective deflation.
\end{proof}

\begin{proposition}\label{PropMorphDef}
Let $C\in\cat$ and $\Thh\in\defff\Ebb$ be any pair of objects. For any morphism $\vp\in(\fp\cat)(\cat(-,C),\Thh)$, the following holds.
\begin{enumerate}
\item There exist $\del\in\Ebb(C,A)$ and a monomorphism $\eta\in(\Mod\cat)(\Thh_{\del},\Thh)$ such that $\eta\circ\wp_{\del}=\vp$.
\item Let $\del$ and $\eta$ be arbitrary as in {\rm (1)}. If $(\Thh,\vp)$ is a reflection of $\cat(-,C)$ along $\iota$, then $\eta$ is an isomorphism. In particular, $\vp$ becomes an epimorphism in $\Mod\cat$ in this case.
\item Let $\del$ and $\eta$ be arbitrary as in {\rm (1)}. Then $\del$ is dominant whenever $(\Thh,\vp)$ is a reflection of $\cat(-,C)$ along $\iota$.
\end{enumerate}
\end{proposition}
\begin{proof}
{\rm (1)} Replacing through an isomorphism, we may assume that there exists $\rho\in\Ebb(D,A)$ such that $\Thh=\Thh_{\rho}$. By the projectivity of $\cat(-,C)$, there exists $c\in\cat(-,C)$ which makes
\[
\xy
(-12,8)*+{\cat(-,C)}="0";
(-12,-8)*+{\cat(-,D)}="2";
(3,5)*+{}="3";
(12,-8)*+{\Thh_{\rho}}="4";
{\ar_{c\circ-} "0";"2"};
{\ar_(0.6){\wp_{\rho}} "2";"4"};
{\ar^{\vp} "0";"4"};
{\ar@{}|\circlearrowright "2";"3"};
\endxy
\]
commutative. If we put $\del=c\uas\rho\in\Ebb(C,A)$, then $\eta=\eta_{(\id_A,c)}\in(\defff\Ebb)(\Thh_{\del},\Thh_{\rho})$ satisfies the required properties.

{\rm (2)} If $(\Thh,\vp)$ is a reflection of $\cat(-,C)$ along $\iota$, then there should exist $\xi\in(\defff\Ebb)(\Thh,\Thh_{\del})$ such that $\xi\circ\vp=\wp_{\del}$. Then, since both $\id_{\Thh},\eta\circ\xi\in(\defff\Ebb)(\Thh,\Thh)$ satisfy $\id_{\Thh}\circ\vp=\vp$ and $(\eta\circ\xi)\circ\vp=\vp$,
it follows $\eta\circ\xi=\id_{\Thh}$. In particular $\eta$ is a split epimorphism, hence an isomorphism since it is monomorphic by {\rm (1)}.

{\rm (3)} Suppose that $(\Thh,\vp)$ is a reflection of $\cat(-,C)$ along $\iota$. Then $\eta$ becomes an isomorphism by {\rm (2)}, hence $(\Thh_{\del},\wp_{\del})$ also becomes a reflection of $\cat(-,C)$ along $\iota$. Thus $\del$ becomes dominant by Corollary~\ref{CorDomDef}.
\end{proof}

\begin{proposition}\label{CorDeflAdj}
For any $\CEs$, the following are equivalent.
\begin{enumerate}
\item The inclusion $\iota\colon\defff\Ebb\hookrightarrow\fp\cat$ has a left adjoint.
\item Any object $\Fbf\in\fp\cat$ has a reflection $(\Gam_{\Fbf},\gam_{\Fbf})$ along $\iota$.
\item For any $C\in\cat$, the object $\cat(-,C)\in\fp\cat$ has a reflection along $\iota$.
\item $\und{\cat}(-,C)\in\defff\Ebb$ holds for any $C\in\cat$.
\item $\CEs$ has enough projective morphisms.
\end{enumerate}
\end{proposition}
\begin{proof}
The equivalence {\rm (1)} $\EQ$ {\rm (2)} is well-known. Indeed if {\rm (2)} holds, then a formal argument shows that $\Fbf\mapsto \Gam_{\Fbf}$ constitutes a functor $\Gam\colon\fp\cat\to\defff\Ebb$ left adjoint to $\iota$, whose unit $\gam\colon\mathrm{Id}_{\fp\cat}\to\iota\circ\Gam$ is given by $\gam=\{\gam_{\Fbf}\}_{\Fbf\in\fp\cat}$.

{\rm (2)} $\Rightarrow$ {\rm (3)} is trivial. The equivalence {\rm (3)} $\EQ$ {\rm (5)} follows from Corollary~\ref{CorDomDef} and Proposition~\ref{PropMorphDef}.
{\rm (3)} $\Rightarrow$ {\rm (4)} follows from Proposition~\ref{PropDefProj} {\rm (1)}. Conversely if {\rm (4)} holds, then for any $C\in\cat$, the morphism
$\und{(-)}\colon\cat(-,C)\to\und{\cat}(-,C)$ gives a reflection of $\cat(-,C)$ along $\iota$ since $(\ref{IsoQuot})$ is an isomorphism.

It remains to show {\rm (3)} $\Rightarrow ${\rm (2)}. 
Suppose that {\rm (3)} holds, and let $\Fbf\in\fp\cat$ be any object. By definition, there exists a right exact sequence
\[ \cat(-,C\ppr)\ov{c\circ-}{\lra}\cat(-,C)\ov{\pi}{\lra}\Fbf\to0 \]
in $\Mod\cat$. Take dominant extensions $\vt\in\Ebb(C,F)$ and $\vt\ppr\in\Ebb(C\ppr,F\ppr)$. Realize them to obtain $\sfr$-triangles
$F\ov{f}{\lra}G\ov{g}{\lra}C\ov{\vt}{\dra}$ and
$F\ppr\ov{f\ppr}{\lra}G\ppr\ov{g\ppr}{\lra}C\ppr\ov{\vt\ppr}{\dra}$.
Since $\vt\ppr$ is dominant, we obtain a morphism of $\sfr$-triangles
\[
\xy
(-14,6)*+{F\ppr}="0";
(0,6)*+{G\ppr}="2";
(14,6)*+{C\ppr}="4";
(28,6)*+{}="6";
(-14,-6)*+{F}="10";
(0,-6)*+{G}="12";
(14,-6)*+{C}="14";
(28,-6)*+{}="16";
{\ar^{f\ppr} "0";"2"};
{\ar^{g\ppr} "2";"4"};
{\ar@{-->}^{\vt\ppr} "4";"6"};
{\ar_{a} "0";"10"};
{\ar^{b} "2";"12"};
{\ar^{c} "4";"14"};
{\ar_{f} "10";"12"};
{\ar_{g} "12";"14"};
{\ar@{-->}_{\vt} "14";"16"};
{\ar@{}|\circlearrowright "0";"12"};
{\ar@{}|\circlearrowright "2";"14"};
\endxy,
\]
which induces a commutative square
\[
\xy
(-11,6)*+{\cat(-,C\ppr)}="0";
(11,6)*+{\Thh_{\vt\ppr}}="2";
(-11,-6)*+{\cat(-,C)}="4";
(11,-6)*+{\Thh_{\vt}}="6";
{\ar^(0.62){\wp_{\vt\ppr}} "0";"2"};
{\ar_{c\circ-} "0";"4"};
{\ar^{\eta=\eta_{(a,c)}} "2";"6"};
{\ar_(0.62){\wp_{\vt}} "4";"6"};
{\ar@{}|\circlearrowright "0";"6"};
\endxy
\]
in $\fp\cat$. If we put $\Gam_{\Fbf}=\Cok\eta$, we obtain a commutative diagram
\begin{equation}\label{MorSeq1}
\xy
(-34,7)*+{\cat(-,C\ppr)}="2";
(-9,7)*+{\cat(-,C)}="4";
(11,7)*+{\Fbf}="6";
(27,7)*+{0}="8";
(-34,-7)*+{\Thh_{\vt\ppr}}="12";
(-9,-7)*+{\Thh_{\vt}}="14";
(11,-7)*+{\Gam_{\Fbf}}="16";
(27,-7)*+{0}="18";
{\ar^{c\circ-} "2";"4"};
{\ar^(0.6){\pi} "4";"6"};
{\ar^{} "6";"8"};
{\ar_{\wp_{\vt\ppr}} "2";"12"};
{\ar^{\wp_{\vt}} "4";"14"};
{\ar^{\gam_{\Fbf}} "6";"16"};
{\ar_{\eta} "12";"14"};
{\ar_{\cok\eta} "14";"16"};
{\ar_{} "16";"18"};
{\ar@{}|\circlearrowright "2";"14"};
{\ar@{}|\circlearrowright "4";"16"};
\endxy
\end{equation}
in $\fp\cat$. Let us to show that $(\Gam_{\Fbf},\gam_{\Fbf})$ gives a reflection of $\Fbf$ along $\iota$. 

Take any $\Thh\in\defff\Ebb$. If we apply $(\fp\cat)(-,\Thh)$ to $(\ref{MorSeq1})$,
then we obtain
\[
\xy
(-56,7)*+{0}="0";
(-34,7)*+{(\fp\cat)(\Fbf,\Thh)}="2";
(8,7)*+{(\fp\cat)(\cat(-,C),\Thh)}="4";
(56,7)*+{(\fp\cat)(\cat(-,C\ppr),\Thh)}="6";
(-56,-7)*+{0}="10";
(-34,-7)*+{(\defff\Ebb)(\Gam_{\Fbf},\Thh)}="12";
(8,-7)*+{(\defff\Ebb)(\Thh_{\vt},\Thh)}="14";
(56,-7)*+{(\defff\Ebb)(\Thh_{\vt\ppr},\Thh)}="16";
{\ar^{} "0";"2"};
{\ar^(0.42){-\circ\pi} "2";"4"};
{\ar^{} "4";"6"};
{\ar^{-\circ\gam_{\Fbf}} "12";"2"};
{\ar^{-\circ\wp_{\vt}}_{\cong} "14";"4"};
{\ar^{-\circ\wp_{\vt\ppr}}_{\cong} "16";"6"};
{\ar_{} "10";"12"};
{\ar_{-\circ \cok\eta} "12";"14"};
{\ar_{-\circ\eta} "14";"16"};
{\ar@{}|\circlearrowright "2";"14"};
{\ar@{}|\circlearrowright "4";"16"};
\endxy
\]
in which the rows are exact, the vertical morphisms in the middle and on the right are isomorphisms by Proposition~\ref{PropDefProj} {\rm (2)}. Hence
\[ -\circ\gam_{\Fbf}\colon(\defff\Ebb)(\Gam_{\Fbf},\Thh)\to(\fp\cat)(\Fbf,\Thh) \]
becomes an isomorphism as desired.
\end{proof}

\begin{corollary}\label{CorProjResol}
If $\CEs$ has enough projective morphisms, then $\defff\Ebb$ has enough projectives.
\end{corollary}
\begin{proof}
Let $\Thh\in\defff\Ebb$ be any object. We may assume that  there exists $\del\in\Ebb(C,A)$ such that $\Thh=\Thh_{\del}$. If we take a dominant extension $\vt\in\Ebb(C,F)$, then there exists $\eta\in(\defff\Ebb)(\Thh_{\vt},\Thh_{\del})$ such that $\eta\circ\wp_{\vt}=\wp_{\del}$. Then $\eta$ is epimorphic, and $\Thh_{\vt}$ is projective by Proposition~\ref{PropDefProj} {\rm (3)}.
\end{proof}

\bibliographystyle{plain}

\end{document}